\documentclass[11pt]{amsart}
\usepackage{palatino}
\usepackage{tikz-cd}
\usepackage{tikz}
\usepackage[all,cmtip]{xy}
\usepackage{mathtools}
\usepackage{bbm}

\usepackage{color}
\usepackage{ifthen}
\usetikzlibrary{backgrounds}
\usepackage{tkz-euclide}
\usepackage[outline]{contour}
\contourlength{1.5pt}
\usetikzlibrary{positioning}

\newcommand{\CC}{\mathbb{C}}

\newcommand{\GG}{\mathbb{G}}
\newcommand{\KK}{\mathbbm{k}}
\newcommand{\NN}{\mathbb{N}}
\newcommand{\PP}{\mathbb{P}}

\newcommand{\RR}{\mathbb{R}}
\newcommand{\ZZ}{\mathbb{Z}}

\newcommand{\cB}{\mathcal{B}}
\newcommand{\cC}{\mathcal{C}}

\newcommand{\cE}{\mathcal{E}}
\newcommand{\cF}{\mathcal{F}}
\newcommand{\cG}{\mathcal{G}}
\newcommand{\cH}{\mathcal{H}}
\newcommand{\cI}{\mathcal{I}}
\newcommand{\cK}{\mathcal{K}}

\newcommand{\cO}{\mathcal{O}}
\newcommand{\cP}{\mathcal{P}}
\newcommand{\cQ}{\mathcal{Q}}
\newcommand{\cR}{\mathcal{R}}
\newcommand{\cS}{\mathcal{S}}
\newcommand{\cT}{\mathcal{T}}
\newcommand{\cU}{\mathcal{U}}
\newcommand{\cV}{\mathcal{V}}
\newcommand{\cW}{\mathcal{W}}
\newcommand{\cX}{\mathcal{X}}
\newcommand{\cY}{\mathcal{Y}}

\newcommand{\fc}{\mathfrak{c}}

\newcommand{\ff}{\mathfrak{f}}

\newcommand{\fm}{\mathfrak{m}}

\newcommand{\fu}{\mathfrak{u}}

\newcommand{\fv}{\mathfrak{v}}
\newcommand{\fw}{\mathfrak{w}}

\newcommand{\dact}{\boldsymbol{.}}
\newcommand{\lra}{\longrightarrow}

\DeclareMathOperator{\add}{\mathsf{add}}

\DeclareMathOperator{\Aut}{Aut}

\DeclareMathOperator{\Char}{char}

\DeclareMathOperator{\Coh}{\mathsf{Coh}}
\DeclareMathOperator{\msCoker}{\mathsf{Coker}}
\DeclareMathOperator{\CR}{\mathsf{CR}}
\DeclareMathOperator{\coker}{coker}

\DeclareMathOperator{\msD}{\mathsf{D}}

\DeclareMathOperator{\EIP}{\mathsf{EIP}}

\DeclareMathOperator{\EKP}{\mathsf{EKP}}

\DeclareMathOperator{\End}{End}

\DeclareMathOperator{\Ext}{Ext}
\DeclareMathOperator{\ev}{ev}

\DeclareMathOperator{\GL}{GL}
\DeclareMathOperator{\Gr}{Gr}

\DeclareMathOperator{\gen}{\mathsf{gen}}
\DeclareMathOperator{\HH}{H}

\DeclareMathOperator{\Hom}{Hom}

\DeclareMathOperator{\msHom}{\mathsf{Hom}}
\DeclareMathOperator{\Inf}{Inf}
\DeclareMathOperator{\Inj}{Inj}
\DeclareMathOperator{\im}{im}
\DeclareMathOperator{\id}{id}
\DeclareMathOperator{\msInd}{\mathsf{Ind}}
\DeclareMathOperator{\Iso}{Iso}
\DeclareMathOperator{\iso}{iso}

\DeclareMathOperator{\msKer}{\mathsf{Ker}}

\DeclareMathOperator{\Mat}{Mat}

\DeclareMathOperator{\modd}{mod}
\DeclareMathOperator{\Mor}{Mor}

\DeclareMathOperator{\msIm}{\mathsf{Im}}
\DeclareMathOperator{\msim}{\mathsf{im}}

\DeclareMathOperator{\mspl}{\mathsf{pl}}

\DeclareMathOperator{\pr}{pr}

\DeclareMathOperator{\ql}{q\ell}
\DeclareMathOperator{\Rad}{Rad}

\DeclareMathOperator{\Rat}{Rat}

\DeclareMathOperator{\reg}{reg}
\DeclareMathOperator{\rep}{rep}
\DeclareMathOperator{\repp}{rep_{proj}}
\DeclareMathOperator{\reppi}{rep_{inj}}
\DeclareMathOperator{\Res}{Res}
\DeclareMathOperator{\res}{res}
\DeclareMathOperator{\rk}{rk}

\DeclareMathOperator{\SL}{SL}

\DeclareMathOperator{\StVect}{\mathsf{StVect}}

\DeclareMathOperator{\TilTheta}{\widetilde{\Theta}}

\DeclareMathOperator{\tr}{tr}

\DeclareMathOperator{\udim}{\underline{\dim}}
\DeclareMathOperator{\Vect}{\mathsf{Vect}}

\usepackage{pictex}
\usepackage{amscd}
\usepackage{latexsym}
\usepackage{amssymb}
\usepackage{eucal}
\usepackage{eufrak}
\usepackage{verbatim}

\numberwithin{equation}{section}

\newtheorem{Theorem}{Theorem}[section]
\newtheorem{Lemma}[Theorem]{Lemma}

\theoremstyle{Theorem}
\newtheorem{Thm}{Theorem}[subsection]
\newtheorem{Lem}[Thm]{Lemma}
\newtheorem{Prop}[Thm]{Proposition}
\newtheorem{Cor}[Thm]{Corollary}

\newtheorem*{thm*}{Theorem A}
\newtheorem*{thm**}{Theorem B}

\theoremstyle{remark}
\newtheorem*{Remark}{Remark}
\newtheorem*{Remarks}{Remarks}
\newtheorem*{Definition}{Definition}
\newtheorem*{Example}{Example}
\newtheorem*{Examples}{Examples}

\numberwithin{equation}{section}

\begin{document}

\title[Kronecker Representations and Steiner Bundles]{Representations of Kronecker Quivers and Steiner Bundles on Grassmannians} 

\author[D. Bissinger and R. Farnsteiner]{Daniel Bissinger \lowercase{and} Rolf Farnsteiner}

\address[Daniel Bissinger]{Mathematisches Seminar, Christian-Albrechts-Universit\"at zu Kiel, Heinrich-Hecht-Platz 6, 24118 Kiel, Germany}
\email{bissinger@math.uni-kiel.de}
\address[Rolf Farnsteiner]{Mathematisches Seminar, Christian-Albrechts-Universit\"at zu Kiel, Heinrich-Hecht-Platz 6, 24118 Kiel, Germany}
\email{rolf@math.uni-kiel.de}

\date{\today}

\makeatletter
\makeatother


\begin{abstract} Let $\KK$ be an algebraically closed field. Connections between representations of the generalized Kronecker quivers $K_r$ and vector bundles on $\PP^{r-1}$ have been known for quite some time. This 
article is concerned with a particular aspect of this correspondence, involving more generally Steiner bundles on Grassmannians $\Gr_d(\KK^r)$ and certain full subcategories $\repp(K_r,d)$ of relative projective 
$K_r$-representations. Building on a categorical equivalence first explicitly established by Jardim and Prata \cite{JP15}, we employ representation-theoretic techniques provided by Auslander-Reiten theory and reflection 
functors to organize indecomposable Steiner bundles in a manner that facilitates the study of bundles enjoying certain properties such as uniformity and homogeneity. Conversely, computational results on Steiner bundles 
motivate investigations in $\repp(K_r,d)$, which elicit the conceptual sources of some recent work on the subject. 

From a purely representation-theoretic vantage point, our paper initiates the investigation of certain full subcategories of the, for $r\!\ge\!3$, wild category of $K_r$-representations. These may be characterized as being 
right Hom-orthogonal to certain algebraic families of elementary test modules. \end{abstract}

\maketitle

\section*{Introduction} \label{S:Intro}
Let $k$ be an algebraically closed field of characteristic $p\!>\!0$. In their groundbreaking article \cite{FPe11}, Friedlander and Pevtsova associated vector bundles to representations of infinitesimal group schemes by 
means of the so-called universal nilpotent operators. In subsequent work \cite{CFS11}, Carlson-Friedlander-Suslin computed some of these bundles for the second Frobenius kernel $\GG_{a(2)}$ of the additive group 
$\GG_a$. The bundles considered in \cite{CFS11} are kernels of the nilpotent operators associated to the so-called $W$-modules. These modules are graded with respect to the standard grading of the ``group algebra" 
$k\GG_{a(2)}$, and the grading induces a decomposition of the associated  vector bundles. 

The vector bundles studied in \cite{CFS11} are defined on the projective line $\PP^1$ and hence decompose into a direct sum of Serre shifts of the structure sheaf. For group schemes of type $\GG_{a(r)}$ there are natural 
generalizations of $W$-modules, whose bundles have base space $\PP^{r-1}$. When studying their graded pieces as mentioned above, one is naturally led to bundles on $\PP^{r-1}$ that are associated to certain 
representations of the $r$-Kronecker quiver $K_r$. The observation that their duals are examples of the so-called Steiner bundles motivated the investigations of this paper.

Since the introduction of the notion of Steiner bundles by Dolgachev-Kapranov \cite{DK93}, the original definition has been generalized to include vector bundles over projective varieties with short resolutions or
coresolutions given by certain exceptional pairs. We will confine our attention to Steiner bundles of Grassmannians, whose defining resolution is given by the universal bundle and trivial bundles, cf.\ \cite{AM15}. In the 
aforementioned context of group schemes, the study of the interplay between representations of elementary abelian $p$-groups and vector bundles on Grassmannians was initiated in \cite{CFP12}. Being based on 
morphisms between Serre shifts  of trivial vector bundles, the approach in loc.\ cit.\ differs from the point of view taken here, which can also be adapted to representations of elementary abelian groups.  

The by now numerous connections between vector bundles on $\PP^{r-1}$ and representations of the generalized Kronecker quiver $K_r$, were apparently first systematically exploited in Hulek's article 
\cite{Hu81}.\footnote{Hulek considers representations of dimension vectors $(n,n)$, whereas our approach is based on dimension vectors $(m,n)$ with $m\!<\!n$.} While in most cases, the authors are working over the 
complex numbers, the examples alluded to above guide us to work in greater generality: \textit{Throughout, $\KK$ is assumed to be an algebraically closed field of arbitrary characteristic. All vector spaces are assumed to be 
finite-dimensional over $\KK$.} 

Our paper is organized as follows. After recalling basic features of representations of Kronecker quivers, we introduce in Section \ref{S:Relproj} the full subcategories $\repp(K_r,d)$ of the category $\rep(K_r)$ of 
representations of the Kronecker quiver that turn out to correspond to Steiner bundles on the Grassmannian of $d$-planes of $\KK^r$. Their objects may be described as those having trivial rank varieties; they may also be 
thought of as relative projective modules of certain ring extensions. The category $\repp(K_r,1)$ coincides with the category $\EKP(K_r)$ of equal kernels representations, whose definition was inspired by \cite{CFS11}, and
which was studied in our context in a number of papers including \cite{Wo13,Bi20}. Our first main result, Theorem \ref{Fam5}, describes the objects $M \in \repp(K_r,d)$ via a family $(E(\fv))_{\fv \in \Gr_d(\KK^r)}$ of 
elementary test modules. This observation is the gateway for our applications of the shift functor $\sigma_{K_r} : \rep(K_r) \lra \rep(K_r)$ and Auslander-Reiten theory. We summarize some of our findings as follows: 

\bigskip

\begin{thm*} Suppose that $r\!\ge\!3$ and $d \in \{1,\ldots, r\!-\!1\}$. Then the following statements hold:
\begin{enumerate}
\item A representation $M \in \rep(K_r)$ belongs to $\repp(K_r,d)$ if and only if $\Hom_{K_r}(E(\fv),M)\!=\!(0)$ for every $\fv \in \Gr_d(\KK^r)$.
\item $\repp(K_r,d) \subseteq \EKP(K_r)$ is a torsion-free class containing all preprojective representations.
\item We have $\repp(K_r,r\!-\!1) \subseteq \repp(K_r,r\!-\!2) \subseteq \cdots \subseteq \repp(K_r,2) \subseteq \repp(K_r,1)\!=\!\EKP(K_r)$.
\item We have $\sigma_{K_r}^{-1}(\repp(K_r,d)) \subseteq \repp(K_r,r\!-\!1)$. \end{enumerate} \end{thm*}

\bigskip
\noindent
Generalizing a consequence of Westwick's fundamental result \cite{We87}, we introduce, for a given $M\!=\!(M_1,M_2, (M(\gamma_i))_{1\le i \le r}) \in \rep(K_r)$ and $d \in \{1,\ldots, r\}$, the invariant $\Delta_M(d)\!:=\!
\dim_\KK M_2\!-\!d\dim_\KK M_1 \in \ZZ$ and show that $\Delta_M(d)\!\ge \min\{\dim_\KK M_1,d\}(r\!-\!d)$ whenever $M \in \repp(K_r,d)$. Moreover, if $(V_1,V_2)$ is a pair of vector spaces such that 
$\Delta_{(V_1,V_2)}(d)\!\ge\!d(r\!-\!d)$, then there always exists $M \in \repp(K_r,d)$ such that $(M_1,M_2)\!=\!(V_1,V_2)$.  

Given $d \in \{1,\ldots, r\!-\!1\},$ we introduce in Section \ref{S:Coh} the functor $\TilTheta_d : \rep(K_r) \lra \Coh(\Gr_d(\KK^r))$ which assigns to a representation of $K_r$ a coherent sheaf on $\Gr_d(\KK^r)$. The essential 
image of $\repp(K_r,d)$ under $\TilTheta_d$ coincides with the full subcategory $\StVect(\Gr_d(\KK^r))$ of Steiner bundles. By definition (see \cite{AM15}), each Steiner bundle $\cF$ affords a short resolution
\[ (0) \lra \cU_{(r,d)}^s \lra \cO_{\Gr_d(\KK^r)}^t \lra \cF \lra (0),\]
where $\cU_{(r,d)}$ denotes the universal subbundle on $\Gr_d(\KK^r)$. Since $\repp(K_r,d) \subseteq \repp(K_r,e)$ for $e\!\le\!d$, one can in fact associate Steiner bundles $\TilTheta_e(M) \in 
\StVect(\Gr_e(\KK^r))$ to each $M \in \repp(K_,d)$ ($e\!\le\!d$). The functor $\TilTheta_e : \repp(K_r,d) \lra \StVect(\Gr_e(\KK^r))$ is full and faithful and $\TilTheta_d : \repp(K_r,d) \lra \StVect(\Gr_d(\KK^r))$ is an 
equivalence, cf.\ \cite{JP15}. We record a few basic properties of $\TilTheta_d$ and, by way of illustration, interpret the results of Section \ref{S:MinType} in the context of Steiner bundles.  

The present understanding of the wild category $\rep(K_r)$ mainly rests on Kac's Theorem \cite{Ka82} concerning the distribution of the dimension vectors of indecomposable representations and the Auslander-Reiten 
quiver $\Gamma(K_r)$, which describes $\rep(K_r)$ modulo the square of its radical. In Section \ref{S:KAR} we introduce these tools and provide first applications concerning exceptional Steiner bundles, which have also 
been referred to as Fibonacci bundles, cf.\ \cite{Bra08}. Much of our subsequent work on non-exceptional Steiner bundles builds on Ringel's determination \cite{Ri78} of the regular Auslander-Reiten 
components of $\Gamma(K_r)$. This enables us to identify the parts of AR-components belonging to $\repp(K_r,d)$ and thus yields a way to efficiently organize the corresponding Steiner bundles of $\Gr_d(\KK^r)$, see 
Section \ref{S:Strat}.\footnote{In case $r\!=\!2$, it is well-known that $\Coh(\PP^1)$ is a $\Hom$-finite abelian $\KK$-category with Serre duality. It therefore affords almost split sequences, see \cite[\S 2]{GL87}.} 

In the next two sections, we apply the aforementioned concepts in the context of homogeneity, uniformity and stability, with the latter two only being discussed in the classical setting of projective spaces. Since the
inhomogeneity of an indecomposable Steiner bundle is inherited by all other members of its component, one can easily construct families of inhomogeneous bundles that are uniform of certain type. There are various 
notions of equivariance for modules and vector bundles, which are equivalent to homogeneity in some cases, such as modules with trivial space endomorphisms or over fields of characteristic $0$. We intend to address 
these connections along with those involving rational representations of parabolic subgroups of $\SL(A_r)$ in future work, where it turns out to be expedient to interpret vector bundles as locally trivial vector space fibrations.  

Given $M \in \EKP(K_r)\!\smallsetminus\!\{(0)\}$ one defines the slope $\mu(M)\!:=\!\frac{\dim_\KK M_1}{\Delta_M(1)}$ and observes that $\mu(\TilTheta_1(M))\!=\!\mu(M)$ for the standard slope on $\Vect(\PP^{r-1})$. This 
motivates the investigation of stability for modules and Steiner bundles and our main result in Section \ref{S:StabE} shows that modules corresponding to stable Steiner bundles are stable. Non-projective modules $M \in 
\repp(K_r,d)$ of minimal type are stable and our results of Section \ref{S:Dist} concerning the distribution of slopes relative to AR-components imply that elementary $\EKP$-modules also enjoy this property. This is applied in 
Section \ref{S:Filtrat}, where we provide a sufficient condition on a pair $(a,b) \in \NN^2$ involving the Tits form for $\EKP$-modules of dimension vector $(a,b)$ to afford a filtration by elementary $\EKP$-modules. Specializing to $r\!=\!3$, we use methods by Drezet \cite{Dr87} to give a sufficient condition for the stability of a bundle corresponding to a stable Kronecker representation. In particular, we verify: 

\bigskip

\begin{thm**} Let $r\!\ge\!3$. 
\begin{enumerate} 
\item Suppose that $M \in \EKP(K_r)\!\smallsetminus\!\{(0)\}$.
\begin{enumerate}
\item If $\TilTheta_1(M)$ is (semi)stable, so is $M$.
\item Let $r\!=\!3$ and $\Delta_M(2)\!\ge\!0$ ($\Delta_M(2)\!>\!0$). If $M$ is semistable (stable), so is $\TilTheta_1(M)$. 
\item Let $r\!=\!3$. If $M$ is elementary, then $\TilTheta_1(M)$ is stable. \end{enumerate}
\item Let $\cC \subseteq \Gamma(K_3)$ be an Auslander-Reiten component containing an elementary module. Then every Steiner bundle $\cF \in \TilTheta_1(\cC\cap\EKP(K_3))$ possesses a Harder-Narasimhan filtration 
by Steiner bundles, whose filtration factors are  stable. \end{enumerate} \end{thm**} 

\bigskip
\noindent
The final section is concerned with an application pertaining to restrictions of Steiner bundles on $\PP^{r-1}(\CC)$ to linear hyperplanes $H \subseteq \PP^{r-1}(\CC)$. We provide conditions on the rank and the first Chern 
class ensuring that a generic Steiner bundle $\cF$ with these data is stable with none of its restrictions $\cF|_H$ being semistable.

\newpage

\begin{center}
\textbf{Acknowledgements}
\end{center}
The authors would like to thank Chiara Brambilla for patiently explaining her work to them, Adrian Langer for bringing reference \cite{Bis04} to their attention, and Enrique Arrondo and Simone Marchesi for correspondence. 

\bigskip

\section{Preliminaries} \label{S:Pre}
\subsection{Representations of $K_r$} \label{S:reps}
Let $r\!\ge\!2$ be a natural number. We denote by $K_r$ the (generalized) Kronecker quiver with two vertices $1$ and $2$ and $r$ arrows:
\[ \begin{tikzcd} 1 \arrow[r, draw=none, "\raisebox{+1.5ex}{\vdots}" description] \arrow[r, bend left, "\gamma_1"]  \arrow[r, bend right, swap,  "\gamma_r"] & 2. \end{tikzcd} \]
(For $r\!=\!1$ one obtains the quiver $K_1$ with underlying Dynkin diagram $A_1$.) 

It will be convenient to view representations of $K_r$ in the following equivalent ways: 
\begin{enumerate}
\item[(i)] A {\it representation} $M$ of the quiver $K_r$ is a triple $(M_1,M_2,(M(\gamma_i))_{1\le i \le r})$, where $M_1, M_2$ are 
finite-dimensional vector spaces and each $M(\gamma_i) : M_1 \lra M_2$ is a $\KK$-linear map. A homomorphism $f : M \lra N$ is a pair $(f_1,f_2)$ of $\KK$-linear maps $f_j : M_j \lra N_j \ (j \in \{1,2\})$ such that, 
for each $i \in \{1,\ldots,r\}$, the diagram
\[ \begin{tikzcd}   M_1 \arrow[r, "M(\gamma_i)"] \arrow[d,"f_1"] & M_2 \arrow[d,"f_2"] \\
                            N_1  \arrow[r,"N(\gamma_i)"] & N_2  
\end{tikzcd} \] 
commutes. 
\item[(ii)] Let $A_r\!:=\!\bigoplus_{i=1}^r\KK\gamma_i$ be the {\it arrow space} of $K_r$. We consider the triangular matrix algebra
\[ \KK K_r := \begin{pmatrix} \KK & 0 \\ A_r & \KK \end{pmatrix},\]
which is isomorphic to the path algebra of $K_r$. Accordingly, $\rep(K_r)$ is equivalent to the category $\modd \KK K_r$ of finite-dimensional $\KK K_r$-modules. In view of \cite[(III.2.2)]{ARS95}, the category $\modd 
\KK K_r$ is equivalent to the category $\cK_r$, whose objects are triples $M\!=\!(M_1,M_2, \psi_M)$, where $M_1,M_2$ are finite-dimensional $\KK$-vector spaces and $\psi_M : A_r\!\otimes_\KK\!M_1 \lra M_2$ is $\KK$-
linear. A morphism $M \lra N$ is a pair $(f_1,f_2)$ of $\KK$-linear maps $f_i: M_i \lra N_i$ such that the diagram
\[ \begin{tikzcd}   A_r\!\otimes_\KK\!M_1 \arrow[r, "\psi_M"] \arrow[d,"\id_{A_r}\otimes f_1"] & M_2 \arrow[d,"f_2"] \\
                           A_r\!\otimes_\KK\!N_1  \arrow[r,"\psi_N"] & N_2  
\end{tikzcd} \] 
commutes. 
\item[(iii)] One can equally well consider triples $(M_1,M_2,\psi_M)$, where $\psi_M : A_r \lra \Hom_\KK(M_1,M_2)$ is $\KK$-linear. The condition for a morphism $(f_1,f_2)$ then reads
\[ \psi_N(a)\circ f_1 = f_2\circ \psi_M(a) \ \ \ \ \forall \ a \in A_r.\]
\item[(iv)] For future reference we also record the interpretation as triples $(M_1,M_2,\varphi_M)$, where $\varphi_M : M_1 \lra \Hom_\KK(A_r,M_2)$ is given by $\varphi_M(m)(a)\!:=\!\psi_M(a \otimes m)$ for all 
$m \in M_1$ and $a \in A_r$. \end{enumerate} 

We denote by $\rep(K_r)$ the category of all (finite-dimensional) representations of $K_r$ and refer the reader to \cite{ARS95} and \cite{ASS06} for further details. 

Every subspace $\fv \subseteq A_r$ gives rise to a subalgebra 
\[ \KK.\fv := \begin{pmatrix} \KK & 0 \\ \fv & \KK \end{pmatrix}\] 
of $\KK K_r$. If $M$ is a $\KK K_r$-module, the restriction $M|_{\KK.\fv}$ corresponds to $(M_1,M_2,\psi_{M,\fv})$, where $\psi_{M,\fv}\!=\!\psi_M|_{\fv\otimes_\KK M}$. By choosing a basis of $\fv$ one obtains an 
equivalence $\modd \KK.\fv \cong \rep(K_{\dim_\KK\!\fv})$, which depends on the choice of the basis.  

Given a pair $(V_1,V_2)$ of $\KK$-vector spaces, we let $\udim(V_1,V_2)\!\:=\!(\dim_\KK V_1,\dim_\KK V_2)$ be its {\it dimension vector}. For $M \in \rep(K_r)$, we write $\udim M\!:=\!\udim(M_1,M_2)$. The category 
$\rep(K_r)$ has two simple objects (representations) $S_1,S_2$ of dimension vectors $(1,0)$ and $(0,1)$, respectively. The object $S_1$ is injective, while $S_2$ is projective. The projective cover $P_1(r)$ of
$S_1$ has dimension vector $\udim P_1(r)\!=\!(1,r)$. We will also write $P_0(r)\!:=\!S_2$.

Note that
\[ \omega_r : \KK K_r \lra \KK K_r \ \ ; \ \ \left(\begin{smallmatrix} \alpha & 0 \\ a & \beta \end{smallmatrix}\right) \mapsto \left(\begin{smallmatrix} \beta & 0 \\ a & \alpha \end{smallmatrix}\right)\]
is an anti-automorphism of order $2$. There results a duality $D_{K_r}: \modd \KK K_r \lra \modd \KK K_r  \   \  ; \  \   M  \mapsto M^\ast$ given by 
\[ (x.f)(m) := f(\omega_r(x).m) \ \ \ \ \ \ \forall \ f \in M^\ast, m \in M, x \in \KK K_r.\]
Let $M \in \rep(K_r)$. For $a \in A_r$, we denote by
\[ a_M : M_1 \lra M_2 \ \ ; \ \ m \mapsto a.m\]
the multiplication effected by $a$. Then the \textit{dual} $D_{K_r}(M)$ of $M \in \modd \KK K_r$ corresponds to $D_{K_r}(M)\!=\!(M_2^\ast, M_1^\ast, \psi_{D_{K_r}(M)})$, where
\[ \psi_{D_{K_r}(M)}(a\otimes f) = f\circ a_M \ \ \ \  \  \forall \ a \in A_r, f \in M_2^\ast.\] 
\noindent
We denote by
\[ \langle \, , \, \rangle_r : \ZZ^2\!\times\!\ZZ^2 \lra \ZZ \ \ ; \ \ (x,y) \mapsto x_1y_1\!+\!x_2y_2\!-\!rx_1y_2\]
the {\it Euler-Ringel bilinear form}. The corresponding quadratic form 
\[ q_r : \ZZ^2 \lra \ZZ \ \; \ \ x \mapsto \langle x,x\rangle_r\]
is referred to as the {\it Tits quadratic form} of $K_r$. 

The dimension vectors give rise to an isomorphism
\[ \udim : K_0(\rep(K_r)) \lra \ZZ^2,\]
which identifies the Grothendieck group $K_0(\rep(K_r))$ of $\rep(K_r)$ with $\ZZ^2$. Recall that
\[\langle \udim M,\udim N\rangle_r\!=\!\dim_\KK\Hom_{K_r}(M,N)\!-\!\dim_\KK\Ext^1_{K_r}(M,N)\]
for all $M,N \in \rep(K_r)$. Moreover, as the path algebra $\KK K_r$ is hereditary, we have $\Ext^i_{K_r}(-,-)\!=\!(0)$ for all $i\!\ge\!2$.

Let $M \in \modd \KK K_r$. Then $\Ext^1_{K_r}(M,\KK K_r)$ is a right $\KK K_r$-module, so that 
\[ \tau_{K_r}(M)\!:=\!\Ext^1_{K_r}(M,\KK K_r)^\ast\]
is a left $\KK K_r$-module. There results a (covariant) endofunctor
\[ \tau_{K_r} : \modd \KK K_r \lra \modd \KK K_r\]
which, by virtue of $\KK K_r$ being hereditary, is left exact. By the same token, $\tau_{K_r}$ is exact when restricted to the full subcategory of those modules which have no non-zero projective direct summands. 

The functor $\tau_{K_r}$ induces an endofunctor on $\rep(K_r)$, which we will also denote by $\tau_{K_r}$. General theory \cite[(VIII.2.2)]{ARS95} implies that 
\[ \udim \tau_{K_r}(M) = \Phi_r.\udim M\]
for every non-projective indecomposable $M \in \rep(K_r)$,\footnote{When multiplying matrices with vectors, we will always implicitly assume that the vector has the correct shape and thus dispense with transposing vectors.} 
where 
\[ \Phi_r\!=\! \begin{pmatrix} r^2\!-\!1 & -r \\ r & -1 \end{pmatrix} \in \GL_2(\ZZ)\] 
is the {\it Coxeter matrix} of the quiver $K_r$. An indecomposable representation $M \in \rep(K_r)$ is called {\it preprojective}, provided $\tau^n_{K_r}(M)\!=\!(0)$ for some $n \in \NN$. 

Similarly, we consider the functor
\[ \tau_{K_r}^{-1} : \modd \KK K_r \lra \modd \KK K_r \ \ ; \  \ M \mapsto \Ext^1_{K_r}(M^\ast, \KK K_r).\]
Then we have
\[ \udim \tau^{-1}_{K_r}(M) = \Phi^{-1}_r.\udim M\]
for every non-injective indecomposable representation $M \in \rep(K_r)$. An indecomposable representation $M \in \rep(K_r)$ is called {\it preinjective}, provided $\tau^{-n}_{K_r}(M)\!=\!(0)$ for some $n \in \NN$. 
Indecomposable representations that are neither preprojective nor preinjective are said to be {\it regular}.

The preprojective $\KK K_r$-modules are well understood. For $r\!\ge\!3$, we put $L_r\!:=\!\frac{r+\sqrt{r^2-4}}{2}$ and set
\[ a_i(r) := \left\{ \begin{array}{cc} \frac{(L_r)^i-(\frac{1}{L_r})^i}{\sqrt{r^2-4}} & r\!\ge\!3 \\ i & r\!=\!2 \end{array} \right.  \ \ \ \ \forall\ i \in \NN_0.\] 
For each $i \in \NN_0$ there is, up to isomorphism, exactly one preprojective module $P_i(r)$ such that 
\begin{enumerate}
\item[(a)] $P_0(r)\!=\!S_2$; $P_1(r)$ is the projective cover of $S_1$, and
\item[(b)] $\udim P_i(r)\!=\!(a_i(r),a_{i+1}(r))$ for all $i \in \NN_0$.\end{enumerate} 
For every $i \in \NN_0$ there is an almost split sequence
\[ (0) \lra P_i(r) \lra rP_{i+1}(r) \lra P_{i+2}(r) \lra (0).\]
In particular, $\tau^{-1}_{K_r}(P_i(r))\cong P_{i+2}(r)$. 

In case the module category to be considered is clear from the context, we will write $P_i$ instead of $P_i(r)$.

\bigskip

\subsection{Module varieties} \label{S:MV} In this section, we recall a number of general results from the representation theory of quivers, most of which can be found in \cite{Ka82}. 

Let $A$ be an $n$-dimensional associative $\KK$-algebra, $V$ be a $\KK$-vector space. We denote by $\modd(A;V)$ the affine variety of $A$-module structures on $V$. When 
convenient we will identify $M \in \modd(A;V)$ with its representation $\varrho_M : A \lra \End_\KK(V)$. If $\{a_1,\ldots, a_n\}$ is a basis of $A$, then
\[ \iota_V : \modd(A;V) \lra \End_\KK(V)^n \ \ ; \ \ M \mapsto (\varrho_M(a_1), \ldots, \varrho_M(a_n))\]
is a closed embedding. 

If $X$ is a variety and $x \in X$, then $\dim_xX$ denotes the {\it local dimension of $X$ at $x$}.

\bigskip

\begin{Lem} \label{MV1} Let $V$ and $W$ be $\KK$-vector spaces. The map
\[ d : \modd(A;V)\!\times\!\modd(A;W) \lra \NN_0 \ \ ; \ \ (M,N) \mapsto \dim_\KK\Hom_A(M,N)\]
is upper semicontinuous. \hfill $\square$ \end{Lem}

\bigskip
\noindent
In analogy with the above, one can consider the variety $\rep(K_r;V_1,V_2)$ of representations of $K_r$ on a pair $(V_1,V_2)$ of vector spaces. There is an isomorphism
\[  \rep(K_r;V_1,V_2) \stackrel{\sim}{\lra} \modd(\KK K_r;V_1\!\oplus\!V_2)^{\rightarrow}\]  
between $\rep(K_r;V_1,V_2)$ and the closed subset  $\modd(\KK K_r;V_1\!\oplus\!V_2)^{\rightarrow}$ of the module variety of the path algebra $\KK K_r$, given by conditions $(\begin{smallmatrix} 1 & 0 \\ 0 & 0 
\end{smallmatrix})\dact (v_1\!\oplus\!v_2)\!=\!v_1$ and $(\begin{smallmatrix} 0 & 0 \\ 0 & 1\end{smallmatrix})\dact (v_1\!+\!v_2)\!=\!v_2.$

\bigskip
\noindent
Recall that $M \in \rep(K_r)$ is referred to as a {\it brick}, provided $\End_{K_r}(M)\cong \KK$. Preprojective and preinjective representations are known to be bricks.

\bigskip

\begin{Cor} \label{MV2} Suppose that $r\!\ge\!3$ and let $V_1,V_2$ be vector spaces with $V_1\!\oplus\!V_2\!\ne\!(0)$ and such that $q_r(\udim (V_1,V_2))\!\le\!1$. Then
\[ \cB_{(V_1,V_2)} := \{ M \in \rep(K_r;V_1,V_2) \ ; \ M \ \text{is a brick} \}\]
is a dense open subset of $\rep(K_r;V_1,V_2)$. \end{Cor} 

\begin{proof} Since $r\!\ge\!3$, the assumption $q_r(a,b)\!=\!0$ implies $a\!=\!0\!=\!b$. 

If $q_r(\udim(V_1,V_2))\!\le\!0$, we may apply \cite[Prop.1]{Ka82} (see also \cite[(3.6)]{Ch13} for an elementary proof valid in our context) in conjunction with the observation above to see that the set 
$\cB_{(V_1,V_2)}$ is not empty. Alternatively, general theory provides $M \in \rep(K_r;V_1,V_2)$ that is preprojective or preinjective (see Theorem \ref{Kac1} below). In particular, $M$ is a brick. 

Lemma \ref{MV1} implies that the map $\rep(K_r;V_1,V_2) \lra \NN_0 \ ; \ M \mapsto \dim_\KK\End_{K_r}(M)$ is upper semicontinuous, so that $\cB_{(V_1,V_2)}$ is open. Since the variety $\rep(K_r;V_1,V_2) 
\cong \Hom_\KK(V_1,V_2)^r$ is irreducible, $\cB_{(V_1,V_2)}$ thus lies dense in $\rep(K_r;V_1,V_2)$. \end{proof} 

\bigskip
\noindent
Let $V_1,V_2$ be non-zero $\KK$-vector spaces. The algebraic group $G\!:=\!\GL(V_2)\!\times\!\GL(V_1)$ acts on the irreducible variety $\Hom_\KK(V_1,V_2)^r$ via
\[ (g_2,g_1)\dact (f_i)_{1\le i\le r} := (g_2\circ f_i\circ g_1^{-1})_{1\le i\le r}.\]
Upper semicontinuity of fiber dimension implies that the map
\[ \Hom_\KK(V_1,V_2)^r \lra \NN_0 \ \ ; \ \ x \mapsto \dim G\dact x\]
is lower semicontinuous. Thus, if $s\!:=\!\max\{ \dim G\dact x \ ; \ x \in \Hom_\KK(V_1,V_2)^r\}$, then the {\it open sheet} 
\[ \cO_{(V_1,V_2)} := \{ x \in \Hom_\KK(V_1,V_2)^r \ ; \ \dim G\dact x\!=\!s\}\]
for the action of $G$ on $\Hom_\KK(V_1,V_2)^r$ is a non-empty open subset of $\Hom_\KK(V_1,V_2)^r\!=\!\rep(K_r;V_1,V_2)$.

\bigskip
\noindent
The following result, which was first proved by Kac \cite[Prop.1]{Ka82}, actually shows that for those dimension vectors for which $\rep(K_r;V_1,V_2)$ contains indecomposable modules, the modules belonging to the open 
sheet are just the bricks.  

\bigskip

\begin{Cor} \label{MV3} Let $r\!\ge\!3$. Suppose that $V_1,V_2$ are vector spaces with $V_1\!\oplus\!V_2\!\ne\!(0)$ and such that $q_r(\udim(V_1,V_2))\!\le\!1$. Then we have $\cB_{(V_1,V_2)}\!=\!\cO_{(V_1,V_2)}$.
\end{Cor}

\begin{proof} Given any $M \in \rep(K_r;V_1,V_2)$, the stabilizer $G_M$ of $M$ coincides with $G\cap\End_{K_r}(M)$, and thus is a dense open subset of $\End_{K_r}(M)$.

By Corollary \ref{MV2}, there is $M_0 \in \cB_{(V_1,V_2)}\cap\cO_{(V_1,V_2)}$, so that
\[ s = \dim G\dact M_0 = \dim G\!-\!\dim G_{M_0} = \dim G\!-\!1.\]
Let $M \in \cB_{(V_1,V_2)}$. Then $\dim G_M\!=\!1$, so that $\dim G\dact M\!=\!s$. This implies $\cB_{(V_1,V_2)}\! \subseteq \!\cO_{(V_1,V_2)}$.

On the other hand, if $M \in \cO_{(V_1,V_2)}$, then $\dim_\KK\End_{K_r}(M)\!=\!\dim G_M\!=\!\dim G\!-\!s\!=\!1$, whence $M \in \cB_{(V_1,V_2)}$. \end{proof}

\bigskip
                                 
\subsection{Restrictions} \label{S:Rest} For $d \in \{1,\ldots, r\}$, we let $\Gr_d(A_r)$ be the Grassmannian of $d$-planes of $A_r$. This is an irreducible projective variety of dimension $\dim \Gr_d(A_r)\!=\!d(r\!-\!d)$.  
For $M \in \modd \KK K_r$ and $\fv \in \Gr_d(A_r)$, we denote by $M|_\fv$ the restriction of $M$ to $\KK.\fv \subseteq \KK K_r$.

Given $\KK$-vector spaces $V_1,V_2$, we consider the set
\[ \Iso(K_d;V_1,V_2)\!:=\!\rep(K_d;V_1,V_2)/(\GL(V_2)\!\times\!\GL(V_1))\] 
of isomorphism classes together with the projection map $\pi : \rep(K_d;V_1,V_2) \lra \Iso(K_d;V_1,V_2)$. We equip $\Iso(K_d;V_1,V_2)$ with the final topology relative to $\pi$. This renders $\pi$ a continuous map
such that $U \subseteq \Iso(K_d;V_1,V_2)$ is open if and only if $\pi^{-1}(U) \subseteq \rep(K_r;V_1,V_2)$ is open.

Let
\[ \Inj_\KK(A_d,A_r) := \{ \alpha \in \Hom_\KK(A_d,A_r) \ ; \ \alpha \ \text{is injective}\}.\]
This is a conical, open subset of the affine space $\Hom_\KK(A_d,A_r)$. In particular, $\Inj_\KK(A_d,A_r)$ is a quasi-affine variety. 

The algebraic groups $\GL(A_d)$ and $\GL(A_r)$ act on the affine variety $\Hom_\KK(A_d,A_r)$ via
\[ g.\alpha := \alpha\circ g^{-1} \ \ \ \forall \ g \in \GL(A_d) \ \ \text{and} \ \ h.\alpha := h\circ \alpha \ \ \ \forall \ h \in \GL(A_r),\]
respectively. The subvariety $\Inj_\KK(A_d,A_r) \subseteq \Hom_k(A_d,A_r)$ is stable with regard to these commuting actions.

Every $\alpha \in \Inj_\KK(A_d,A_r)$ defines an injective homomorphism
\[ \alpha : \KK K_d \lra \KK K_r \   \   ;  \   \   \left(\begin{smallmatrix} x & 0 \\ a & y \end{smallmatrix} \right) \lra \left(\begin{smallmatrix} x & 0 \\ \alpha(a) & y \end{smallmatrix}\right) \]
of $\KK$-algebras. In particular, $\GL(A_r)$ can be considered a group of automorphisms of $\KK K_r$.  

Let $M \in \modd \KK K_r$. Given $\alpha \in \Inj_\KK(A_d,A_r)$, we let $\alpha^\ast(M) \in \modd \KK K_d$ be the pull-back of $M$ along $\alpha$. Thus, we have
\[ u \dact m := \alpha(u)m \ \ \ \ \ \forall \ u \in \KK K_d, m \in M.\]
In particular, $\GL(A_r) \subseteq \Aut_\KK(\KK K_r)$ acts on $\modd \KK K_r$ via pull-back: If $M \in \modd \KK K_r$, and $g \in \GL(A_r)$, then
\[M^{(g)} := (g^{-1})^\ast(M).\]
We say that $M$ is {\it homogeneous}, provided $M^{(g)} \cong M$ for every $g \in \GL(A_r)$.

\bigskip

\begin{Example} Let $P_i\!=\!(k[X,Y]_{i-1},k[X,Y]_i, X\cdot, Y\cdot)$ be the preprojective $\KK K_2$-module of dimension vector $(i,i\!+\!1)$. Then $\GL(A_2)\!=\!\GL(\KK X\!\oplus\!\KK Y)$
acts of $\KK [X,Y]$ via algebra homomorphisms of degree zero in such a way that
\[ g(\gamma f) = g(\gamma) g(f) \ \ \ \ \ \ \ \ \ \forall \ g \in \GL(A_2), \gamma \in A_2, f \in \KK[X,Y]_{i-1}.\]
Consequently, $g$ induces a $\KK$-linear isomorphism $P_i^{(g)} \stackrel{\sim}{\lra} P_i$ such that
\[ g(\gamma \dact f) = g(g^{-1}(\gamma)f) = \gamma g(f).\]
As a result, the $\KK K_2$-module $P_i$ is homogeneous. (In fact, the $P_i$ are examples of equivariant $K_r$-representations.)\end{Example}

\bigskip
\noindent
We let $\GL(A_r)$ act on $\Gr_d(A_r)$ via $h\dact \fv\!:=\!h(\fv)$ for all $h \in \GL(A_r)$ and $\fv \in \Gr_d(A_r)$. This action is transitive.   

Following \cite{Wo13}, we denote by $\EKP(K_r)$ the full subcategory of $\rep(K_r)$ such that $a_M$ is injective for every $a \in A_r\!\smallsetminus\!\{0\}$.\footnote{The notation derives from the
correspondence with the category of equal kernels modules studied in \cite{CFS11}.} Note that $\EKP(K_r)$ is closed under taking subobjects.

The classification of indecomposable objects in $\rep(K_2)$ implies that every $N \in \EKP(K_2)$ decomposes into a direct sum
\[ N \cong \bigoplus_{i\ge 0} n_iP_i.\]

\bigskip 

\begin{Lem} \label{Rest1} Let $M \in \EKP(K_r)$. Then we have
\[ \alpha^\ast(M) \cong \beta^\ast(M)\]
for all $\alpha,\beta \in \Inj_\KK(A_2,A_r)$ such that $\im\alpha\!=\!\im\beta$. \end{Lem} 

\begin{proof} Let $\alpha \in \Inj_\KK(A_2,A_r)$. Since $\alpha$ is injective, it follows that $\alpha^\ast(M) \in \EKP(K_2)$. By virtue of the foregoing remark and the example above, we conclude that  $\alpha^\ast(M)$ 
is homogeneous.

If $\beta \in \Inj_\KK(A_2,A_r)$ is such that $\im \alpha\!=\!\im \beta$, then $g\!:=\!\beta^{-1}\circ \alpha \in \GL(A_2)$, so that for $m \in (\alpha^\ast(M))^{(g)}$ and $\gamma \in A_2$, we obtain
\[ \gamma\dact m = g^{-1}(\gamma).m = \alpha(g^{-1}(\gamma))m = \beta(\gamma)m.\]
By the example above, this shows that $\beta^\ast(M)\cong (\alpha^\ast(M))^{(g)} \cong \alpha^\ast(M)$. \end{proof} 

\bigskip
\noindent
Let  $M \in \EKP(K_r)$, $\fv \in \Gr_2(A_r)$. In view of the foregoing result, we may define the isomorphism class $[M|_\fv]$ via
\[ [M|_\fv] := [\alpha^\ast(M)] \ \ \ \ \ \alpha \in \Inj_\KK(A_2,A_r)\ ; \  \im \alpha\!=\!\fv.\]
Less formally, we will write
\[ M|_\fv \cong \bigoplus_{i\ge 0} n_i(M,\fv) P_i(2)\]
to indicate that the above isomorphism holds for $\alpha^\ast(M)$, where $\alpha \in \Inj_\KK(A_2,A_r)$ is such that $\im\alpha\!=\!\fv$.   

Given $X \in \rep(K_2;V_1,V_2)$, we write
\[ \dim_\KK\Hom_{K_2}(X,M|_\fv) := \dim_\KK\Hom_{K_2}(X,\alpha^\ast(M)).\]
for some $\alpha \in \Inj_\KK(A_2,A_r)$ with $\im\alpha\!=\!\fv$.  

\bigskip 

\begin{Prop} \label{Rest2} Let $r\!\ge\!2$, $M \in \EKP(K_r)\cap \rep(K_r;V_1,V_2)$. Then the following statements hold:
\begin{enumerate} 
\item The map
\[ \res_M : \Gr_2(A_r) \lra \Iso(K_2;V_1,V_2) \ \ ; \ \  \fv \mapsto [M|_\fv]\]
is continuous. 
\item If $X \in \rep(K_2;U_1,U_2)$, then the map
\[ d_{X,M} : \Gr_2(A_r) \lra \NN_0 \ \ ; \ \ \fv \mapsto \dim_\KK\Hom_{K_2}(X,M|_\fv)\]
is upper semicontinuous. \end{enumerate} \end{Prop} 

\begin{proof} (1) It follows from  \cite[(1.1)]{CFP15} that the map
\[ \msim :  \Inj_\KK(A_2,A_r) \lra \Gr_2(A_r) \ \ ; \ \ \alpha \mapsto \im\alpha\]
is a principal $\GL(A_2)$-bundle (a $\GL(A_2)$-torsor). In particular,
\[ \overline{\msim} : \Inj_\KK(A_2,A_r)/\GL(A_2) \lra \Gr_2(A_r) \ \ ; \ \ \bar{\alpha} \mapsto \im\alpha\]
is a homeomorphism. 

Consider the map
\[ \widetilde{\res}_M : \Inj_\KK(A_2,A_r) \lra  \rep(K_2;V_1,V_2) \ \ ; \ \ \alpha \mapsto \alpha^\ast(M).\]
We identify $\rep(K_2;V_1,V_2)$ with $\Hom_\KK(V_1,V_2)^2$ and let $\{\gamma_1,\gamma_2\}$ be a basis of $A_2$. Then
\[ \widehat{\res}_M : \Hom_\KK(A_2,A_r) \lra \Hom_\KK(V_1,V_2)^2 \ \ ; \ \ \alpha \mapsto (M(\alpha(\gamma_1)),M(\alpha(\gamma_2)))\]
is a linear map and hence in particular a morphism. (Here we have written $M(\alpha(\gamma_i))\!:=\!\psi_M(\alpha(\gamma_i)\!\otimes -)$.) Consequently, 
$\widetilde{\res}_M\!=\!\widehat{\res}_M|_{\Inj_\KK(A_2,A_r)}$ is also a morphism. 

In view of Lemma \ref{Rest1}, $\widetilde{\res}_M$ gives rise to  a continuous map
\[ \iso_M : \Inj_\KK(A_2,A_r)/\GL(A_2) \lra \Iso(K_2;V_1,V_2) \ \ ; \ \ \bar{\alpha} \mapsto [\alpha^\ast(M)].\]
Since
\[ \res_M \circ \overline{\msim} = \iso_M,\]
we obtain (1).

(2) Thanks to Lemma \ref{MV1}, the map
\[ d_X : \rep(K_2;V_1,V_2) \lra \NN_0 \ \ ; \ \ Y \mapsto \dim_\KK\Hom_{K_2}(X,Y)\]
is upper semicontinuous. Since images of $(\GL(V_2)\!\times\!\GL(V_1))$-stable open sets under $\pi : \rep(K_2;V_2,V_2)$ $ \lra \Iso(K_2;V_1,V_2)$ are open, the map
\[ \bar{d}_X : \Iso(K_2;V_1,V_2) \lra \NN_0 \ \ ; \ \ [Y] \mapsto \dim_\KK\Hom_{K_2}(X,Y)\]
enjoys the same property. In view of (1), $d_{X,M}\!=\! \bar{d}_X\circ \res_M$ is upper semicontinuous as well. \end{proof}

\bigskip

\begin{Cor} \label{Rest3} Let $M \in \EKP(K_r)$. Then there exists a sequence $(n_i(M))_{i \in \NN_0} \in \NN_0^{\NN_0}$ such that
\[ O_M := \{ \fv \in \Gr_2(A_r) \ ; \  n_i(M,\fv)\!=\!n_i(M) \ \ \ \forall \ i \in \NN_0\}\]
is a dense open subset of $\Gr_2(A_r)$. \end{Cor}

\begin{proof} For $i \in \NN_0$ we consider the map $d_i\!:=\!d_{P_i(2),M}$. Proposition \ref{Rest2} provides $c_i \in \NN_0$ such that
\[ O_i := \{ \fv \in \Gr_2(A_r) \ ; \ d_i(\fv) = c_i\}\]
is a dense open subset of $\Gr_2(A_r)$. Since $\dim_\KK\Hom_{K_2}(P_i(2),P_j(2))\!=\!\max\{j\!-\!i\!+\!1,0\}$, it follows that
\[ n_i(M,\fv) = d_i(\fv)\!-\!2d_{i+1}(\fv)\!+\!d_{i+2}(\fv)\]
for every $\fv \in \Gr_2(A_r)$ and $i \in \NN_0$. Consequently,
\[ n_i(M,\fv) = c_i\!-\!2c_{i+1}\!+\!c_{i+2}\]
for every $\fv \in O_i\cap O_{i+1}\cap O_{i+2}$. There exists $\ell \in \NN_0$ such that $n_i(M,\fv)\!=\!0\!=\! d_i(\fv)$ for all $i\!\ge\!\ell$ and $\fv \in \Gr_2(A_r)$. We define
\[ n_i(M):= c_i\!-\!2c_{i+1}\!+\!c_{i+2}.\]
for every $i \in \NN_0$ and obtain $n_i(M)\!=\!0\!=n_i(M,\fv)$ for all $i\!\ge\!\ell$. Consequently, $\bigcap_{i=0}^{\ell-1} O_i \subseteq O_M$. 

Suppose that $\fv \in O_M$. Then we have
\[ d_i(\fv)\!-\!2d_{i+1}(\fv)\!+\!d_{i+2}(\fv) = n_i(M,\fv) =  c_i\!-\!2c_{i+1}\!+\!c_{i+2} \ \ \ \ 0\!\le\!i\!\le\!\ell\!-\!1.\]
Since the $(\ell\!\times\!\ell)$-matrix
\[ A = \begin{pmatrix} 1 & -2 & 1 & 0 & \cdots & 0 \\ 
                                  0 &  1 & -2 & 1 & \cdots & 0 \\
                                  \vdots & \vdots & \vdots & \vdots & \vdots & \vdots \\
                                  0 & 0 &  \cdots & 1 & -2 & 1 \\
                                  0 & 0 & \cdots  & 0 & 1 & -2\\
                                  0 & 0 & \cdots & 0 & 0 & 1 \end{pmatrix}\]
 is invertible, it follows that $d_i(\fv)\!=\!c_i$ for all $i \in \NN_0$. Consequently, $\fv \in \bigcap_{i=0}^{\ell-1}O_i$, so that $O_M\!=\!\bigcap_{i=0}^{\ell-1}O_i$ is a dense, open subset of $\Gr_2(A_r)$. \end{proof}
 
\bigskip

\begin{Definition} Let $M \in \EKP(K_r)\!\smallsetminus\!\{(0)\}$. We call
\[ M_{\gen} :=  \bigoplus_{i\ge 0} n_i(M)P_i\]
the {\it generic decomposition} of $M$. \end{Definition}

\bigskip
\noindent 
We let $\PP(A_r)$ be the projective space of lines in $A_r$. The basis $\{\gamma_1,\ldots, \gamma_r\}$ defines an isomorphism $\PP(A_r) \cong \PP^{r-1}$ of projective varieties.  For $x \in \PP(A_r)$ the number 
$\rk(x_M)\!=\! \rk(a_M) \ \text{for some} \ a \in x\!\smallsetminus\!\{0\}$ is well-defined and we refer to
\[ \rk(M) := \max\{\rk(x_M) \ ; \ x \in \PP(A_r)\}\]
as the {\it generic rank} of $M$. We say that $M$ has {\it constant rank}, provided $\rk(x_M)\!=\!\rk(M)$ for all $x \in \PP(A_r)$. The full subcategory of $\rep(K_r)$, whose  objects have constant rank, will be denoted 
$\CR(K_r)$. It contains the full subcategory $\EIP(K_r)$, whose objects $M$ satisfy the condition $\rk(M)\!=\!\dim_\KK M_2$. We have a duality $D_{K_r} : \CR(K_r) \lra \CR(K_r)$ that sends $\EKP(K_r)$ to $\EIP(K_r)$ and 
vice versa. The aforementioned classification of indecomposable $K_2$-representations implies that $\CR(K_2)\!=\!\add(\EKP(K_2)\cup\EIP(K_2))$, so that every $M \in \CR(K_r)$ affords a generic decomposition
\[ M_{\gen} = \bigoplus_{i\ge 0} n_i(M)P_i(2)\!\oplus\!\bigoplus_{j\ge 0} m_j(M) I_j(2),\]
where $I_j(2)\!:=\!D_{K_r}(P_j(2))$ is indecomposable preinjective.  

\bigskip

\section{Categories of relative projective modules} \label{S:Relproj}
Let $d \in \{1,\ldots, r\!-\!1\}$. In this section we introduce the full subcategory $\repp(K_r,d)$ of $\EKP(K_r)$, that will turn out to be equivalent to the category of Steiner bundles on the Grassmannian $\Gr_d(A_r)$. 
The objects of $\repp(K_r,d)$ may be characterized via rank varieties, whose definition is analogous to those investigated in \cite{CFP12}. By contrast to the setting in \cite{CFP12}, their complements may also be 
defined as being the objects of the right $\Hom$-orthogonal category of a class of elementary $K_r$-modules.  

\bigskip

\subsection{Rank varieties and families of elementary modules} \label{S:Fam}
As noted above, the map
\[ \msim : \Inj_\KK(A_d,A_r) \lra \Gr_d(A_r) \  \  ;  \ \ \alpha \mapsto \im \alpha\]
is a $\GL(A_d)$-torsor with respect to the canonical right $\GL(A_d)$-action $\alpha.h\!:=\!\alpha\circ h$, cf.\ \cite[(1.1)]{CFP15}. 

\bigskip

\begin{Lem} \label{Fam1} Let $M \in \rep(K_r)$. Then the following statements hold:
\begin{enumerate} 
\item The subset
\[ \cP(K_r,d)_M := \{ \alpha \in \Inj_\KK(A_d,A_r) \ ; \ \alpha^\ast(M) \ \text{is not projective}\}\]
is closed and $\GL(A_d)$-stable. 
\item The set
\[\cV(K_r,d)_M := \{ \fv \in \Gr_d(A_r) \ ; \ \exists \ \alpha \in \msim^{-1}(\fv) \ \text{such that} \ \alpha^\ast(M) \ \text{is not projective}\}\]
is a closed subset of $\Gr_d(A_r)$. \end{enumerate} \end{Lem}

\begin{proof} (1) Recall that a $\KK K_d$-module $N$ is projective if and only if $\dim_\KK\Ext^1_{K_d}(N,S_2)\!=\!0$. In view of the upper semicontinuity of the map $\rep(K_d;M_1,M_2) \lra \NN_0 \ ; \ N \mapsto 
\dim_\KK\Ext^1_{K_d}(N,S_2)$,\footnote{This follows for instance from Lemma \ref{MV1} and $\dim_\KK\Ext^1_{K_d}(N,S_2)\!=\!\dim_\KK \Hom_{K_d}(N,S_2)\!-\!\langle \udim N,\udim S_2\rangle_d$.} we 
conclude that
\[ C_d := \{ X \in \rep(K_d;M_1,M_2) \ ; \ X \ \text{is not projective}\}\]
is a closed subset of $\rep(K_d;M_1,M_2)$. The arguments of the proof of Proposition \ref{Rest2} ensure the continuity of the map
\[ \Inj_\KK(A_d,A_r) \lra \rep(K_d;M_1,M_2) \ \ ; \ \ \alpha \mapsto \alpha^\ast(M).\]
Consequently, $\cP(K_r,d)_M$ is a closed subset of $\Inj_\KK(A_d,A_r)$. 

Recall that $P_1$ and $S_2$ are representatives of the isomorphism classes of the projective indecomposable $\KK K_d$-modules. Since their dimensions pairwise differ, it follows that $Q^{(h)} \cong Q$ for 
$Q \in \{P_1,S_2\}$ and $h \in \GL(A_d)$ and hence $P^{(h)} \cong P$ for every projective $\KK K_d$-module $P$. Let $\alpha \not \in \cP(K_r,d)_M$ and $h \in \GL(A_d)$. Then
\[ (h\dact\alpha)^\ast(M) \cong (\alpha^\ast(M))^{(h)},\]
is projective, so that $h\dact\alpha \not\in \cP(K_r,d)_M$. As a result, the variety $\cP(K_r,d)_M$ is $\GL(A_d)$-stable. 

(2) As before, we consider the space $\Inj_\KK(A_d,A_r)/\GL(A_d)$ together with the topology induced by the canonical projection
\[ \pi : \Inj_\KK(A_d,A_r) \lra \Inj_\KK(A_d,A_r)/\GL(A_d).\]
As observed in the proof of Proposition \ref{Rest2}, the map 
\[ \msim : \Inj_\KK(A_d,A_r) \lra \Gr_d(A_r) \ \ ; \ \ \alpha \mapsto \im \alpha \]
induces a homeomorphism
\[ \overline{\msim} :  \Inj_\KK(A_d,A_r)/\GL(A_d) \lra \Gr_d(A_r) \ \ ; \ \ [\alpha] \mapsto \im \alpha.\]
In view of (1), the open set $U_M\!:=\!\Inj_\KK(A_d,A_r)\!\smallsetminus\!\cP(K_r,d)_M$ is $\GL(A_d)$-stable. Thus, $\pi^{-1}(\pi(U_M))\!=\!U_M$, so that $\pi(U_M)\!=\!(\Inj_\KK(A_d,A_r)/\GL(A_d))\!\smallsetminus\!
\pi(\cP(K_r,d)_M)$ is an open subset of  $\Inj_\KK(A_d,A_r)/\GL(A_d)$. It follows that $\cV(K_r,d)_M\!=\!\msim(\cP(K_r,d)_M)\!=\!\overline{\msim}(\pi(\cP(K_r,d)_M))$ is closed. \end{proof}

\bigskip

\begin{Remark} Since $\GL(A_d)$ acts transitively on $\im^{-1}(\fv)$, it follows from (1) that
\[ \cV(K_r,d)_M = \{ \fv \in \Gr_d(A_r) \ ; \  \alpha^\ast(M) \ \text{is not projective for all}  \ \alpha \in \msim^{-1}(\fv) \}.\]
\end{Remark}

\bigskip

\begin{Lem} \label{Fam2} Let $(0) \lra M' \lra M \lra M'' \lra (0)$ be an exact sequence in $\rep(K_r)$.
\begin{enumerate}
\item  We have
\[ \cV(K_r,d)_{M'} \subseteq \cV(K_r,d)_M \subseteq \cV(K_r,d)_{M'}\cup\cV(K_r,d)_{M''}.\]
\item If the sequence splits, then
\[  \cV(K_r,d)_M = \cV(K_r,d)_{M'}\cup\cV(K_r,d)_{M''}. \]
\end{enumerate} \end{Lem} 

\begin{proof} (1) Since the algebra $\KK K_d$ is hereditary, we have $\cV(K_r,d)_{M'} \subseteq \cV(K_r,d)_M$. 

(2) This is a direct consequence of (1). \end{proof} 

\bigskip
\noindent
We consider the projective $K_r$-representations $(0,A_d)\cong P_0(r)^d$ and $(\KK,A_r)\cong P_1(r)$, where we have $\psi_{P_1(r)}(a\!\otimes\!t)\!=\!t a$. An element $\alpha \in \Hom_\KK(A_d,A_r)$ can be viewed as a 
morphism 
\[ \bar{\alpha} : (0,A_d) \lra (\KK,A_r) \ \ ; \ \  \bar{\alpha}_1\!:=\!0 \ \ \text{and} \ \ \bar{\alpha}_2\!:=\!\alpha.\] 
Let $M \in \rep(K_r)$. As noted in Section \ref{S:reps}, $M$ may be interpreted as a linear map
\[ \varphi_M : M_1 \lra \Hom_\KK(A_r,M_2) \ \ ; \ \  \varphi_M(m)(a)\!:=\!\psi_M(a\!\otimes m).\]
Given $\alpha \in \Hom_\KK(A_d,A_r)$, we obtain
\[ \varphi_{\alpha^\ast(M)} : M_1 \lra \Hom_\KK(A_d,M_2) \   \  ; \ \ m \mapsto \varphi_M(m)\circ \alpha.\]

\bigskip
\noindent
A representation $M \in \rep(K_r)$ is called {\it regular}, if all its indecomposable constituents are regular. We recall the notion of elementary representations, which are just the simple objects in the full subcategory $\reg(K_r) \subseteq \rep(K_r)$ of regular representations.

\bigskip

\begin{Definition} A non-zero regular representation $E \in \rep(K_r)$ is referred to as {\it elementary}, provided there does not exist a short exact sequence
\[ (0) \lra A \lra E \lra B \lra (0)\]
with $A,B$ non-zero and regular. \end{Definition} 

\bigskip

\begin{Prop} \label{Fam3} Let $\alpha \in \Inj_\KK(A_d,A_r)$. Then the following statements hold:
\begin{enumerate}
\item The module $\coker\bar{\alpha}$ is elementary with dimension vector $\udim \coker\bar{\alpha}\!=\!(1,r\!-\!d)$. 
\item Let $\beta \in \Inj_\KK(A_d,A_r)$. Then $\coker\bar{\alpha}\cong \coker\bar{\beta}$ if and only if $\im \alpha\!=\!\im\beta$.
\item If $N\in \rep(K_r)$ is indecomposable with $\udim N\!=\!(1,r\!-\!d)$, then there is $\zeta \in \Inj_\KK(A_d,A_r)$ such that $N \cong \coker\bar{\zeta}$.
\item Given $M \in \rep(K_r)$, we have an isomorphism of vector spaces $\coker\varphi_{\alpha^\ast(M)} \cong \Ext^1_{K_r}(\coker\bar{\alpha},M)$. 
\item Given $c \in \{1,\ldots, d\}$ and $\beta \in \Inj_\KK(A_c,A_d)$, there exists an epimorphism $\coker\overline{\alpha\circ \beta} \lra \coker\bar{\alpha}$. \end{enumerate}   \end{Prop}  

\begin{proof} (1)-(3) Recall that the projective cover $P_1(r)$ of the simple representation $S_1\!=\!(\KK,0)$ is given by $P_1(r)\!=\!(\KK,A_r,(P_1(r)(\gamma_i))_{1 \leq i \leq r})$ with $P_1(r)(\gamma_i)(\lambda)\!=\!\lambda 
\gamma_i$ for all $i \in \{1,\ldots,r\}$. For $\alpha \in \Inj_\KK(A_d,A_r)$ we consider the monomorphism of representations $x_\alpha\!=\!(0,(x_\alpha)_2) \colon P_0(r)^d$ $\lra P_1(r)$, given by
\[(x_\alpha)_2 \colon \KK^d \lra A_r \ \ ; \ \ \lambda \mapsto \sum^d_{i=1}  \lambda_i \alpha(\gamma_i).\]
We get a commutative diagram
\[ \xymatrix{
(0) \ar[r] & P_0(r)^d \ar^{x_\alpha}[r] \ar^{\cong}[d] & P_1(r) \ar[r] \ar^{\cong}[d]& \coker x_\alpha \ar[r] \ar[d] & (0)  \\
(0) \ar[r] & (0,A_d) \ar^{\overline{\alpha}}[r] & (\KK,A_r) \ar[r] & \coker \overline{\alpha} \ar[r] & (0) 
}\]
with exact rows. In particular,  $\udim\coker \overline{\alpha}\!=\!(1,r\!-\!d)$ and the 5-Lemma yields $\coker x_\alpha \cong \coker \overline{\alpha}$. Hence (1), (2) and (3) follow from  \cite[(2.1.4),(2.2.1),(2.2.2)]{Bi20}, 
respectively. 

(4) The statement follows from a small modification of the proof of \cite[(2.3.3)]{Bi20}. For the benefit of the reader we outline the changes. Since $\Ext^1_{K_r}((\KK,A_r),-)\!=\!(0)$, application of $\Hom_{K_r}(-,M)$ to the 
short exact sequence 
\[(0) \lra (0,A_d) \stackrel{\overline{\alpha}}{\lra} (\KK,A_r) \lra \coker \overline{\alpha} \lra (0)\] 
yields the following diagram with exact rows 
\[ 
\xymatrix{
\Hom_{K_r}((\KK,A_r),M) \ar^{\overline{\alpha}^\ast}[r] \ar^{f}[d]& \Hom_{K_r}((0,A_d),M) \ar[r] \ar^{g}[d] & \Ext^1_{K_r}(\coker \overline{\alpha},M) \ar[r] & (0)  \\
M_1 \ar^{\varphi_{\alpha^{\ast}(M)}}[r] &  \Hom_\KK(A_d,M_2) \ar[r] & \coker \varphi_{\alpha^{\ast}(M)} \ar[r] & (0), 
}\]
where $f(h)\!:=\!h_1(1_\KK)$ for all $h \in \Hom_{K_r}((\KK,A_r),M)$ and $g(\eta)\!:=\!\eta_2$ for all $\eta \colon (0,A_d) \lra M$. Let $(h_1,h_2)\!=\!h \in \Hom_{K_r}((\KK,A_r),M)$. We recall that $\psi_M \circ (\id_{A_r} 
\otimes h_1)\! =\! h_2 \circ \psi_{(\KK,A_r)}$ and $\psi_{(\KK,A_r)}(a \otimes \lambda)\! =\! \lambda a$ for all $\lambda \in \KK$ and $a \in A_r$. For $a \in A_d$ we obtain
\begin{align*}
 [(\varphi_{\alpha^\ast(M)} \circ f)(h)](a) &= [\varphi_M(h_1(1_\KK)) \circ \alpha](a) = [\psi_M \circ (\id_{A_r} \otimes h_1)](\alpha(a) \otimes 1_\KK)\\
 &= [h_2 \circ \psi_{(\KK,A_r)}](\alpha(a) \otimes 1_\KK) = h_2(\alpha(a)) =  (h \circ \overline{\alpha})_2(a)  \\
 &= [(g \circ \overline{\alpha}^{\ast})(h)](a).
\end{align*}
Hence the left-hand square of the diagram commutes and $f, g$ being isomorphisms gives us
\[\dim_\KK \coker \varphi_{\alpha^{\ast}(M)} = \dim_\KK \Ext^1_{K_r}(\coker \overline{\alpha},M).\]

(5) This follows from the proof of \cite[(2.2.2(c))]{Bi20}.  \end{proof} 

\bigskip
\noindent
Let $\fv \in \Gr_d(A_r)$. Thanks to Proposition \ref{Fam3}(2), we may define
\[ \coker \fv := \coker \bar{\alpha} \ \ \ \ \ \ \alpha \in \msim^{-1}(\fv).\]

\bigskip

\begin{Definition} For $\fv \in \Gr_d(A_r)$, we put 
\[ E(\fv) := D_{K_r}(\tau_{K_r}(\coker \bar{\alpha})) \ \ \ \ \ \ \alpha \in \msim^{-1}(\fv).\] 
\end{Definition}

\bigskip
\noindent
In view of \cite[(1.1)]{KL96} $E(\fv)$ is an elementary representation.

Given $M \in \rep(K_r)$ and $\fv \in \Gr_d(A_r)$, we recall the notation
\[ \psi_{M,\fv} := (\psi_M)|_{\fv\otimes_\KK M_1}.\]
Similarly, we define $\varphi_{M,\fv} : M_1\lra \Hom_\KK(\fv,M_2)$ via
\[ \varphi_{M,\fv}(m) := \varphi_M(m)|_\fv  \ \ \ \ \forall \ m \in M_1.\]

\bigskip

\begin{Prop} \label{Fam4} Let $M \in \rep(K_r)$, $\fv \in \Gr_d(A_r)$. 
\begin{enumerate}
\item Let $\alpha \in \msim^{-1}(\fv)$. Then 
\[\dim_\KK M_2\!-\!\rk(\psi_{M,\fv}) = \dim_\KK M_2\!-\!\dim_\KK\Rad(\alpha^{\ast}(M))\] 
is the multiplicity of $P_0(d)$ in the decomposition of $\alpha^{\ast}(M) \in \rep(K_d)$ into indecomposable direct summands.
\item We have $\dim_\KK \ker \psi_{M,\fv}\!=\! \dim_\KK \Hom_{K_r}(E(\fv),M)$.
\end{enumerate} \end{Prop}

\begin{proof} (1) This follows immediately from the definition.

(2) We consider the dual representation $D_{K_r}(M) = (M_2^{\ast},M_1^{\ast},\psi_{D_{K_r}(M)} )$ with structure map 
\[\psi_{D_{K_r}(M)} \colon A_r\!\otimes_\KK\!M_2^{\ast} \lra M_1^{\ast} \ \ ; \ \  a \otimes f  \mapsto f \circ a_M.\]
There results a commutative diagram
\[\xymatrixcolsep{6pc} \xymatrix{
 M_2^{\ast} \ar^{(\psi_{M,\fv})^{\ast}}[r] \ar@{=}[d]& (\fv\!\otimes_\KK\!M_1)^{\ast} \ar_{\cong}^{\eta}[d]\\
M_2^{\ast} \ar^{\varphi_{D_{K_r}(M),\fv}}[r] & \Hom_\KK(\fv,M_1^{\ast}), 
}\]
with $\eta(f)(a)(m_1)\!=\!f(a \otimes m_1)$ for all $a \in \fv, m_1 \in M_1$ and $f \in (\fv\!\otimes_\KK\!M_1)^{\ast}$: For $h \in M_2^\ast, a \in \fv$ and $m \in M_1$ we have
\begin{align*}
[ \varphi_{D_{K_r}(M),\fv}(h)(a)](m) &= [ \psi_{D_{K_r}(M)}(a \otimes h)](m) = (h \circ a_M)(m)\\ 
                                                        &= (h \circ \psi_M)(a \otimes m) = (h \circ \psi_{M,\fv})(a \otimes m) \\
                                                        &= [\eta(h \circ \psi_{M,\fv})(a)](m) = [(\eta \circ (\psi_{M,\fv})^{\ast})(h)(a)](m).
\end{align*}
Hence $\varphi_{D_{K_r}(M),\fv}\!=\!\eta \circ (\psi_{M,\fv})^{\ast}$. Consequently, $\coker \varphi_{D_{K_r}(M),\fv} \cong \ker \psi_{M,\fv}$. 

Let $\alpha \in \im^{-1}(\fv)$. We conclude with Proposition \ref{Fam3}(2),(4) that 
\[\dim_\KK\coker\varphi_{\alpha^\ast(D_{K_r}(M))} =\dim_\KK\coker\varphi_{D_{K_r}(M),\fv} = \dim_\KK\ker\psi_{M,\fv}.\]
Since $\coker\overline{\alpha}$ is regular, the Auslander-Reiten formula \cite[(IV.2.13)]{ASS06} in conjunction with Proposition \ref{Fam3}(4) now yields
\begin{align*}
\dim_\KK\ker\psi_{M,\fv} & = \dim_\KK \coker \varphi_{\alpha^{\ast}(D_{K_r}(M))} = \dim_\KK \Ext^1_{K_r}(\coker \overline{\alpha},D_{K_r}(M)) \\
&= \dim_\KK \Hom_{K_r}(D_{K_r}(M),\tau_{K_r}(\coker \overline{\alpha})) \\
&=\dim_\KK \Hom_{K_r}(E(\fv),M) \hfill \qedhere.
\end{align*}  \end{proof}

\bigskip
\noindent
For a pair  $(V_1,V_2)$ of $\KK$-vector spaces and $d\in \{1,\ldots,r\}$, we set
\[ \Delta_{(V_1,V_2)}(d) := \dim_\KK V_2\!-\!d\dim_\KK V_1.\]
If $M \in \rep(K_r)$, we write 
\[\Delta_M(d) := \Delta_{(M_1,M_2)}(d).\] 

\bigskip

\begin{Thm} \label{Fam5} Let $M \in \rep(K_r)$. Given $\fv \in \Gr_d(A_r)$, the following statements are equivalent:
\begin{enumerate}
\item $\Hom_{K_r}(E(\fv),M)\!=\!(0)$.
\item For every $\alpha \in \msim^{-1}(\fv)$, the map $\psi_{\alpha^\ast(M)}$ is injective.
\item There is $\alpha \in \msim^{-1}(\fv)$ such that $\psi_{\alpha^\ast(M)}$ is injective.
\item There is $\alpha \in \msim^{-1}(\fv)$ such that $\alpha^\ast(M) \cong \Delta_M(d)P_0(d)\!\oplus\!(\dim_kM_1)P_1(d)$.
\item $\fv \not\in \cV(K_r,d)_M$. \end{enumerate} \end{Thm} 

\begin{proof} (1) $\Rightarrow$ (2). Let $\alpha \in \im^{-1}(\fv)$. Proposition \ref{Fam4}(2) yields $0\!=\!\dim_\KK\Hom_{K_r}(E(\fv),M)\!=\!\dim_\KK \ker \psi_{\alpha^\ast(M)}$.

(2) $\Rightarrow$ (3). Trivial.

(3) $\Rightarrow$ (4). Let $\alpha \in \msim^{-1}(\fv)$ be such that $\psi_{\alpha^\ast(M)}$ is injective. Then we have 
\begin{align*}
 \dim_\KK M_2\! -\! \rk(\psi_{M,\fv}) & = \dim_\KK M_2\!-\!\dim_\KK (\fv\!\otimes_\KK\!M_1) \\
 &= \dim_\KK M_2\! -\! d \dim_\KK M_1 = \Delta_M(d).
\end{align*} 
By Proposition \ref{Fam4}(1), we may write $\alpha^{\ast}(M)\!=\!{\Delta_M(d)} P_0(d)\!\oplus\!N$, where $N \in \rep(K_d)$ does not have $P_0(d)$ as a direct summand. Therefore, a projective cover of $N$ is given by  
$(\dim_\KK N_1)P_1(d) \twoheadrightarrow N$. Since $\dim_\KK N_1\!=\!\dim_\KK M_1$ and 
\[ \dim_\KK N_2 = \dim_\KK M_2\!-\! \Delta_M(d) = d \dim_\KK M_1, \]
we conclude $\udim N\! =\! \udim [(\dim_\KK N_1) P_1(d)]$. Hence $\alpha^{\ast}(M) \cong \Delta_M(d) P_0(d)\!\oplus\! (\dim_\KK M_1)\!P_1(d)$.

(4) $\Rightarrow$ (5). This is a direct consequence of the remark following Lemma \ref{Fam1}. 

(5) $\Rightarrow$ (1). Let $\alpha \in \msim^{-1}(\fv)$. We write $\alpha^{\ast}(M) \cong a P_0(d)\! \oplus\!b P_1(d)$, so that 
\[ \dim_\KK \Rad(\alpha^{\ast}(M)) = \dim_\KK \Rad(a P_0(d)\!\oplus\! b P_1(d)) = b \dim_\KK \Rad(P_1(d)) = b d.\]
Since $\udim P_0(d)$ and $\udim P_1(d)$ are linearly independent, it follows that $a\!=\!\Delta_M(d)$ and $b\!=\! \dim_\KK M_1$. We obtain
$\dim_\KK \ker \psi_{M,\fv}\! = \! d \dim_\KK M_1\! -\! \rk(\psi_{M,\fv})\! =\! d \dim_\KK M_1\!  -\! \dim_\KK\Rad(\alpha^{\ast}(M))\!=\!0$, so that Proposition \ref{Fam4} yields (1).  \end{proof}

\bigskip
\noindent
Let $M \in \rep(K_r)$. Theorem \ref{Fam5} implies
\[ \cV(K_r,d)_M = \{\fv \in \Gr_d(A_r) \ ; \  \rk(\psi_{M,\fv})\!<\!d\dim_\KK M_1\},\]
so that we refer to $\cV(K_r,d)_M$ as the {\it $d$-th rank variety} of $M$. 

\bigskip

\begin{Remarks} (1) For $d\!=\!1$, we have $\Gr_1(A_r)\!=\!\PP(A_r)$ and 
\[ \cV(K_r,1)_M := \{ x \in \PP(A_r) \ ; \ \rk(x_M)\!<\!\dim_\KK M_1\}.\]
In particular, $\cV(K_r,1)_M\!=\!\emptyset$ if and only if $M \in \EKP(K_r)$. 

(2) Let $\alpha \in \Inj_\KK(A_d,A_r)$. The pull-back functor $\alpha^\ast : \rep(K_r) \lra \rep(K_d)$ takes projectives to projectives. Theorem \ref{Fam5} implies that
$\beta \in \GL(A_d).\alpha$ if and only if for every $M \in \rep(K_r)$ we have $\alpha^\ast(M) \ \text{projective} \ \Leftrightarrow \ \beta^\ast(M) \ \text{projective}$. Hence
$\Inj_\KK(A_d,A_r)/\GL(A_d)$ may be viewed as an analogue of the space of equivalence classes of $p$-points (in the context of abelian unipotent group schemes), cf.\ \cite{FPe05}. \end{Remarks}

\bigskip

\subsection{The categories $\repp(K_r,d)$}\label{S:CatRep}
As before, we fix $d \in \{1,\ldots, r\!-\!1\}$. In this section, we introduce the subcategory $\repp(K_r,d)$ of $\rep(K_r)$ that will be instrumental in our study of Steiner bundles on the Grassmannian $\Gr_d(A_r)$. 

\bigskip

\begin{Definition} We let $\repp(K_r,d)$ be the full subcategory of $\rep(K_r)$, whose objects $M$ are given by $\cV(K_r,d)_M\!=\!\emptyset$. \end{Definition}

\bigskip

\begin{Remarks} (1) Note that $\repp(K_r,1)\!=\!\EKP(K_r)$ is just the category of equal kernels representations, which usually has wild representation type, cf.\ \cite[(4.1.1)]{Bi20}. The roughly analogous concept of shifted 
cyclic subgroups for elementary abelian $p$-groups yields projective modules, see \cite[(11.8)]{Da}. 

(2) The category $\repp(K_r,d)$ is closed under taking submodules, cf.\ Lemma \ref{Fam2}.

(3) Suppose that $M \in \repp(K_r,d)$ is such that $\dim_\KK M_1\!\le\!d$. Then $A_r\!\otimes_\KK\!M_1\!=\!\bigcup_{\fv \in \Gr_d(A_r)}\fv\otimes_\KK\!M_1$, so that $\psi_M$ is injective. Consequently, the direct
summand $(M_1,\psi_M(M_1),\psi_M)$ of $M$ is isomorphic to $(\dim_\KK M_1)P_1(r)$ and $M\cong \Delta_M(r)P_0(r)\!\oplus\!(\dim_\KK M_1)P_1(r)$ is projective. 

(4) An object $M \in \repp(K_r,d)$ can be characterized by saying that $\alpha^\ast(M)$ is projective for every $\alpha \in \Inj_\KK(A_d,A_r)$. Equivalently, for $M \in \modd \KK K_r$ the restrictions
$M|_{\KK .\fv} \in \modd \KK .\fv$ are projective for every $\fv \in \Gr_d(A_r)$. The latter condition can be shown to be equivalent to $M$ being $(\KK K_r, \KK .\fv)$-projective in the sense of relative homological 
algebra, cf.\ \cite{Ho56}. \end{Remarks}

\bigskip

\begin{Prop}\label{CatRep1} The following statements hold:
\begin{enumerate}
\item $\repp(K_r,d) \subseteq \EKP(K_r)$ is a torsion free class that is closed under $\tau^{-1}_{K_r}$.
\item $\repp(K_r,d)$ contains all preprojective representations.
\item We have $\repp(K_r,r\!-\!1) \subseteq \repp(K_r,r\!-\!2) \subseteq \cdots \subseteq \repp(K_r,1) = \EKP(K_r)$.
\end{enumerate} \end{Prop}

\begin{proof} (1) and (2) follow from Theorem \ref{Fam5} and \cite[(2.1.3)]{Bi20}.

(3) Let $M \in \repp(K_r,b)$ for some $1\! <\! b\!<\! r$ and consider $\fv \in \Gr_{b-1}(A_r)$. Then there is $\fw \in \Gr_b(A_r)$ such that $\fv \subseteq \fw$. Theorem \ref{Fam5} yields
\[ \ker\psi_{M,\fv} \subseteq \ker\psi_{M,\fw} = (0),\]
so that $M \in \repp(K_r,b\!-\!1)$. \end{proof}

\bigskip

\begin{Remark} Let $X \subseteq \Gr_d(A_r)$ be a subset. Using Theorem \ref{Fam5} and \cite[(2.1.3)]{Bi20}, one can show that (1) and (2) of the foregoing result hold for the full subcategory $\rep_X(K_r,d)$, whose objects
$M$ satisfy $\cV(K_r,d)_M \subseteq X$. \end{Remark} 

\bigskip
\noindent
In the sequel the shift functors $\sigma_{K_r},\sigma^{-1}_{K_r}: \rep(K_r) \lra \rep(K_r)$ will be of major importance. These functors correspond to the BGP-reflection functors but take into account that the opposite 
quiver of $K_r$ is isomorphic to $K_r$, i.e.\ $D_{K_r} \circ \sigma_{K_r} \cong \sigma^{-1}_{K_r} \circ D_{K_r}$, where $D_{K_r} \colon \rep(K_r) \lra \rep(K_r)$ denotes the standard duality. 

Given a representation $M \in \rep(K_r)$, $\sigma_{K_r}(M)$ is by definition the representation 
\[ (\sigma_{K_r}(M)_1,\sigma_{K_r}(M)_2) = (\ker \psi_M,M_1), \]
where we identify $\psi_M$ by means of the basis $\{\gamma_1,\ldots, \gamma_r\}$ with the map $\psi_M : (M_1)^r \lra M_2, (m_i) \mapsto \sum^{r}_{i=1} M(\gamma_i)(m_i)$. By definition, $[\sigma_{K_r}(M)](\gamma_i) 
\colon \sigma_{K_r}(M)_1 \lra \sigma_{K_r}(M)_2 = \pi_{i}|_{\ker \psi_M}$, with $\pi_i \colon (M_1)^r \lra M_1$ being the projection onto the $i$-th component. If $f \in \Hom_{K_r}(M,N)$, then $\sigma_{K_r}(f)_1 : \sigma_{K_r}(M)_1 \lra \sigma_{K_r}(N)_1$ is the restriction to $\ker\psi_M$ of the map
\[ (m_i)_{i\le i\le r} \mapsto (f_1(m_i))_{1\le i \le r},\]
while $\sigma_{K_r}(f)_2\!:=\!f_1$. 

According to \cite[(VII.5.3(b))]{ASS06}, $\sigma_{K_r}$ is left exact, while $\sigma_{K_r}^{-1}$ is right exact and the functor $\sigma_{K_r}$ induces an equivalence
\[ \sigma_{K_r} : \rep_2(K_r) \lra \rep_1(K_r)\]
between the full subcategories $\rep_i(K_r)$ of $\rep(K_r)$, whose objects don't have any direct summands isomorphic to $S_i$. By the same token, $\sigma_{K_r}^{-1}$ is a quasi-inverse of $\sigma_{K_r}$. 
Moreover, $\sigma_{K_r}$ and $\sigma^{-1}_{K_r}$ induce quasi-inverse equivalences on the full subcategory $\reg(K_r) \subseteq \rep(K_r)$ of regular representations. We have $\sigma_{K_r} \circ \sigma_{K_r} \cong 
\tau_{K_r}$ as well as $\sigma_{K_r}(P_{i+1}(r)) \cong  P_{i}(r)$ for all $i\!\geq\!0$. 

Recall that $\{ \udim P_i(r) \ ; \  i \in \NN_0 \}$ consists exactly of those tuples $(a,b) \in \NN^2_0$ that satisfy $a\!<\!b$ and $q_r(a,b)\!=\!1$. Since all irreducible morphisms between preprojective representations in 
$\rep(K_r)$ are injective, it follows that $P_i(r)$ is isomorphic to a subrepresentation of $P_j(r)$ if and only if $i\! \leq\!j$. We 
also note that 
\[\udim \sigma_{K_r}(M)\! =\!(r \dim_\KK M_1\! -\! \dim_\KK M_2, \dim_\KK M_1)\] 
for $M$ indecomposable and not isomorphic to $P_0(r)$, while $\sigma_{K_r}(P_0(r))\!=\!(0)$.  In conjunction with the left exactness of $\sigma_{K_r}$ this implies that $\sigma_{K_r} \colon \rep_2(K_r) \lra \rep_1(K_r)$ is exact. By the same token, $\sigma_{K_r}^{-1} \colon \rep_1(K_r) \lra \rep_2(K_r)$ is exact.

\bigskip

\begin{Lem}\label{CatRep2} Let $1\! \leq \!d \!<\! r$. We have 
\[\{ [\sigma_{K_r}(E(\fv))] \ ; \  \fv \in \Gr_d(A_r) \} = \{  [\coker \fw] \ ;  \ \fw \in \Gr_{r-d}(A_r) \}.\] \end{Lem}

\begin{proof} Let $\fv \in \Gr_d(A_r)$. Since $E(\fv)$ is regular indecomposable, the same holds for $\sigma_{K_r}(E(\fv))$. As $\udim \sigma_{K_r}(E(\fv))\!=\! \udim \sigma^{-1}_{K_r}(D_{K_r}(\coker \fv))\!=\!(1,d)$, it
follows from Proposition \ref{Fam3}(3) that $\sigma_{K_r}(E(\fv))$ $\cong \coker \fw$ for some $\fw \in \Gr_{r-d}(A_r)$.

If $\fw \in \Gr_{r-d}(A_r)$, then the preceding argument provides $\fv \in \Gr_d(A_r)$ such that $\sigma_{K_r}(E(\fw)) \cong \coker \fv$. We therefore obtain
\[ \sigma_{K_r}(E(\fv)) \cong \sigma_{K_r}(D_{K_r}(\tau_{K_r}(\coker \fv))) \cong D_{K_r}(\sigma_{K_r}(\coker\fv)) \cong D_{K_r}(\tau_{K_r}(E(\fw))) \cong \coker \fw. \qedhere\]
\end{proof}

\bigskip
\noindent
Let $M,E \in \rep(K_r)$ with $E$ regular. We claim that
\begin{enumerate}
\item[(i)]  $\Hom_{K_r}(E,M) \cong \Hom_{K_r}(\sigma_{K_r}(E),\sigma_{K_r}(M))$, and
\item[(ii)] $\Hom_{K_r}(E,M) \cong \Hom_{K_r}(\tau_{K_r}(E),\tau_{K_r}(M))$.   \end{enumerate}
In order to verify (i), we write $M\!=\!S_2^l\!\oplus\!N$, with $S_2$ not being a direct summand of $N$. Since there are no non-zero homomorphisms from non-zero regular representations to 
projective representations (see \cite[(VIII.2.13)]{ASS06}), we have $\Hom_{K_r}(E,S_2)\!=\!(0)$. Moreover, $\sigma_{K_r}(S_2)\!=\!(0)$ implies  $\sigma_{K_r}(M) \cong \sigma_{K_r}(N)$ and we conclude
\begin{eqnarray*}  \Hom_{K_r}(E,M)  & \cong   & \Hom_{K_r}(E,S_2^l)\!\oplus\! \Hom_{K_r}(E,N) \cong  \Hom_{K_r}(\sigma_{K_r}(E),\sigma_{K_r}(N))\\ 
                                                           & \cong & \Hom_{K_r}(\sigma_{K_r}(E),\sigma_{K_r}(M)).
\end{eqnarray*} 
Since $\sigma_{K_r}^2 \cong \tau_{K_r}$, we also get (ii). 

\bigskip

\begin{Thm}\label{CatRep3} Let $1\! \leq\! d\! <\! r$ and $M \in \rep(K_r)$.
\begin{enumerate}
\item $M \in \repp(K_r,d)$ if and only if $\Hom_{K_r}(\coker\fw,\sigma_{K_r}(M))\!=\!(0)$ for all $\fw \in \Gr_{r-d}(A_r)$.
\item $M \in \repp(K_r,r\!-\!1)$ if and only if $\sigma_{K_r}(M) \in \EKP(K_r)$. 
\item We have $\sigma_{K_r}^{-1}(\repp(K_r,d)) \subseteq \repp(K_r,r\!-\!1)$. In particular, $\repp(K_r,d)$ is $\sigma^{-1}_{K_r}$-stable. \end{enumerate} \end{Thm}

\begin{proof} (1) This is a direct consequence of Theorem \ref{Fam5}, Lemma \ref{CatRep2}, and (i).  

(2) In view of (i), we have $M \in \rep(K_r,r\!-\!1)$ if and only if 
\[ (0) = \Hom_{K_r}(E(\fv),M) \cong \Hom_{K_r}(\sigma_{K_r}(E(\fv)),\sigma_{K_r}(M)) \cong  \Hom_{K_r}(D_{K_r}(\sigma_{K_r}(\coker \fv)),\sigma_{K_r}(M))\] 
for all $\fv \in \Gr_{r-1}(A_r)$. Since $\dim_\KK D_{K_r}(\sigma_{K_r}(\coker \fv)) = (1,r\!-\!1)$ for all $\fv \in \Gr_{r-1}(A_r)$, the statement follows from Lemma \ref{CatRep2} and Proposition \ref{Fam3}(3).

(3) Let $M\!=\!\sigma_{K_r}^{-1}(N)$ for some $N \in \repp(K_r,d)$. Let $\fv \in \Gr_d(A_r)$. Since $\dim_\KK \Hom_{K_r}(E(\fv),S_1)\!=\!\dim_\KK E(\fv)_1\!\neq\!0$, it follows that $N \in \rep_1(K_r)$. 
Then Proposition \ref{CatRep1}(3) in conjunction with the equivalence $\rep_2(K_r) \lra \rep_1(K_r)$ yields $\sigma_{K_r}(M) \cong N \in \EKP(K_r)$, so that (2) ensures that $M \in \rep(K_r,r\!-\!1)$. 
By Proposition  \ref{CatRep1}(3), this also implies that $\repp(K_r,d)$ is $\sigma^{-1}_{K_r}$-stable.\end{proof}

\bigskip

\begin{Example}  In view of $\sigma_{K_r}(P_{i+1}(r)) \cong  P_{i}(r)$ for all $i\!\geq\!0$, parts (2) and (3) of Theorem \ref{CatRep3} imply inductively that $P_i(r) \in \repp(K_r,r\!-\!1)$ for all $i\!\ge\!0$. \end{Example}

\bigskip
\noindent
We finally record a topological property of the set of relative projective modules of fixed dimension vector. The relevance of the technical condition (2) will be clarified in the following section.

\bigskip  

\begin{Prop} \label{CatRep4} Let $V_1,V_2$ be $\KK$-vector spaces, $d\in \{1,\ldots,r\!-\!1\}$. Then the following statements hold:
\begin{enumerate}
\item The set $\repp(K_r,d)\cap\rep(K_r;V_1,V_2) \!=\!\{M \in \rep(K_r;V_1,V_2) \ ; \ \cV(K_r,d)_M\!=\!\emptyset\}$ is open. 
\item If $\Delta_{(V_1,V_2)}(d)\!\ge\!d(r\!-\!d)$, then $\repp(K_r,d)\cap\rep(K_r;V_1,V_2)$ lies dense in $\rep(K_r;V_1,V_2)$.  \end{enumerate} \end{Prop} 

\begin{proof} (1) We interpret $\rep(K_r;V_1,V_2)$ as $\Hom_\KK(A_r\!\otimes_\KK\!V_1,V_2)$. In view of Theorem \ref{Fam5}, the relevant set is given by
\[ \cO_{(r,d)} := \{ \psi \in \Hom_\KK(A_r\!\otimes_\KK\!V_1,V_2) \ ; \ \psi\circ (\alpha\!\otimes \id_{V_1}) \in \Inj_\KK(A_d\!\otimes_\KK\!V_1,V_2) \ \ \ \forall \ \alpha \in \Inj_\KK(A_d,A_r)\} .\]
We consider the canonical map
\[ \msim : \Inj_\KK(A_d,A_r) \lra \Gr_d(A_r) \ \ ; \ \ \alpha \mapsto \im\alpha.\]
As noted before, this map defines a principal $\GL(A_d)$-bundle, and hence so does
\[ \kappa : \Hom_\KK(A_r\!\otimes_\KK\!V_1,V_2)\!\times\!\Inj_\KK(A_d,A_r) \lra \Hom_\KK(A_r\!\otimes_\KK\!V_1,V_2)\!\times\!\Gr_d(A_r) \ \ ; \ \ (\psi, \alpha) \mapsto (\psi,\im\alpha).\]
In particular, $\kappa$ is an open morphism. 

For $\psi \in \Hom_\KK(A_r\!\otimes_\KK\!V_1,V_2)$ and $\fv \in \Gr_d(A_r)$, we put $\psi_\fv\!:=\!\psi|_{\fv\otimes_\KK V_1}$. We consider the sets
\[ \cO_1\!:=\!\{(\psi, \alpha) \in \Hom_\KK(A_r\!\otimes_\KK\!V_1,V_2)\!\times\!\Inj_\KK(A_d,A_r) \ ; \ \rk(\psi\circ (\alpha\otimes\id_{V_1}))\!=\!d\dim_\KK V_1\}\] 
and 
\[ \cO_2 \!:= \! \{(\psi,\fv) \in \Hom_\KK(A_r\!\otimes_\KK\!V_1,V_2)\!\times\!\Gr_d(A_r) \ ; \ \rk(\psi_\fv)\!=\!d\dim_\KK V_1\}.\]
Since $\Hom_\KK(A_r\!\otimes_\KK\!V_1,V_2)\!\times\!\Inj_\KK(A_d,A_r) \lra \Hom_\KK(A_d\!\otimes_\KK\!V_1,V_2) \ ; \ (\psi,\alpha) \mapsto \psi\circ(\alpha\otimes\id_{V_1})$ is a morphism, lower semicontinuity of ranks 
ensures that $\cO_1$ is an open subset, so that 
\[\cO_2 = \kappa(\cO_1)\]
is open as well. As a result, 
\[ \cC_{(r,d)}\!:=\!\{(\psi,\fv) \in \Hom_\KK(A_r\!\otimes_\KK\!V_1,V_2)\!\times\!\Gr_d(A_r) \ ; \ \rk(\psi_\fv)\!<\!d\dim_\KK V_1\}.\]
is closed. 

As the projective variety $\Gr_d(A_r)$ is complete, the morphism $\pr : \Hom_\KK(A_r\!\otimes_\KK\!V_1,V_2)\!\times\!\Gr_d(A_r) \lra \Hom_\KK(A_r\!\otimes_\KK\!V_1,V_2) \ ; \ (\psi,\fv) \mapsto \psi$ is closed. 
Consequently, 
\[\cX_{(r,d)}:=\pr(\cC_{(r,d)})\] 
is a closed subset of the affine space $\Hom_\KK(A_r\!\otimes_\KK\!V_1,V_2)$. 

Suppose that $\psi \not \in \cO_{(r,d)}$. Then there is $\fv \in \Gr_d(A_r)$ such that $\rk(\psi_\fv)\!<\!d\dim_\KK V_1$, so that $(\psi,\fv) \in \cC_{(r,d)}$. It follows that $\psi \in \cX_{(r,d)}$. Conversely, if $\psi \in \cX_{(r,d)}$, 
then there is $\fv \in \Gr_d(A_r)$ such that $(\psi,\fv) \in \cC_{(r,d)}$. Thus, there is $\alpha \in \Inj_\KK(A_d,A_r)$ such that $\rk(\psi \circ (\alpha\otimes\id_{V_1}))\!<\!d\dim_\KK V_1$, whence $\psi\not \in \cO_{(r,d)}$.

As an upshot of the above, we obtain that 
\[ \cO_{(r,d)} = \Hom_\KK(A_r\!\otimes_\KK\!V_1,V_2)\!\smallsetminus\!\cX_{(r,d)}\] 
is open. 

(2) In view of (1), it suffices to show that $\repp(K_r,d)\cap\rep(K_r;V_1,V_2)\!\ne\!\emptyset$. Setting $n\!:=\!\dim_\KK V_1$ and $m\!:=\!\dim_\KK V_2$, we define for $\ell \in \{0,\ldots,n\}$
\[ \Hom_\KK(V_1,V_2)_{ \le \ell} :=  \{ f \in \Hom_\KK(V_1,V_2) \ ; \ \rk(f)\!\le\!\ell\}.\]
We proceed in several steps, beginning by recalling that (see \cite[Chap.IV,exerc. 3]{Pe})

(i) \textit{$\Hom_\KK (V_1,V_2)_{\le \ell}$ is a closed, irreducible subspace of dimension $\ell(n\!+\!m\!-\!\ell)$}.

\medskip
\noindent
Next, we verify the following claim 

(ii) \textit{The closed subset $\cC_{(r,d)}$ defined in (1) is an irreducible variety of dimension}
\[ \dim \cC_{(r,d)} = rmn\!-\!1\!-\!\Delta_{(V_1,V_2)}(d)\!+\!d(r\!-\!d).\] 

\smallskip
\noindent
In what follows, we let $\ell\!:=\!dn\!-\!1$. The algebraic group $\KK^\times$ acts on $\Hom_\KK(A_r\!\otimes_\KK\!V_1,V_2)\!\times\!\Gr_d(A_r)$ via
\[ \alpha\dact(\psi,\fv) := (\alpha\psi,\fv) \ \ \ \ \forall \ (\psi,\fv) \in \Hom_\KK(A_r\!\otimes_\KK\!V_1,V_2)\!\times\!\Gr_d(A_r).\]
Note that $\cC_{(r,d)}$ is $\KK^\times$-stable, so that every irreducible component $Z$ of $\cC_{(r,d)}$ enjoys the same property. 

We consider the surjective morphism
\[ q : \cC_{(r,d)} \lra \Gr_d(A_r) \  \  ; \  \  (\psi,\fv) \mapsto \fv\]
as well as 
\[ \iota_0 : \Gr_d(A_r) \lra \Hom_\KK(A_r\!\otimes_\KK\!V_1,V_2)\!\times\!\Gr_d(A_r) \ \ ; \ \  \fv \mapsto (0,\fv) \]
Let $Z$ be an irreducible component of $\cC_{(r,d)}$. Then $Z$ is a closed subset of $\Hom_\KK(A_r\!\otimes_\KK\!V_1,V_2)\!\times\!\Gr_d(A_r)$. If $\fv \in q(Z)$, then there is $\psi \in 
\Hom_\KK(A_r\!\otimes_\KK\!V_1,V_2)$ such that $(\psi,\fv) \in Z$, and the $\KK^\times$-stability mentioned above implies $\iota_0(\fv)\!=\!(0,\fv) \in \overline{\KK^\times\dact(\psi,\fv)} \subseteq Z$. Consequently, $q(Z)\!=\!
\iota_0^{-1}(Z)$ is closed.

Let $\fv \in \Gr_d(A_r)$. Writing $A_r\!=\!\fv\!\oplus\!\fw$, we obtain  
\[ q^{-1}(\fv) \cong\Hom_\KK(\fv\!\otimes_\KK\!V_1,V_2)_{\le \ell}\!\times\! \Hom_\KK(\fw\!\otimes_\KK\!V_1,V_2),\]
so that (i) ensures that $q^{-1}(\fv)$ is irreducible of dimension
\[ \dim q^{-1}(\fv) = \ell(m\!+\!dn\!-\!\ell)\!+\!mn(r\!-\!d) = \ell(m\!+\!1)\!+\!mn(r\!-\!d).\]
In view of \cite[(1.5)]{Fa04}, the variety $\cC_{(r,d)}$ is irreducible and the fiber dimension theorem \cite[(\S I.8, Thm.3)]{Mu} yields
\begin{eqnarray*} 
\dim \cC_{(r,d)} & = & \ell(m\!+\!1)\!+\!mn(r\!-\!d) \!+\!d(r\!-\!d)\\
              & = & (dn\!-\!1)(m\!+\!1)\!+\!mn(r\!-\!d)\!+\!d(r\!-\!d)\\ 
              & = & dmn\!+\!\!dn\!-\!m\!-\!1\!+\!mn(r\!-\!d)\!+\!d(r\!-\!d)\\  
              & = & rmn\!-\!1\!-\!\Delta_{(V_1,V_2)}(d)\!+\!d(r\!-\!d),
\end{eqnarray*} 
as desired. \hfill $\diamond$

\medskip
\noindent
By virtue of (ii) and our current assumption, the fiber dimension theorem implies
\[ \dim \cX_{(r,d)} \le \dim \cC_{(r,d)}  \le rmn\!-\!1 < \dim \Hom_\KK(A_r\!\otimes_\KK V_1,V_2),\]
so that $\cO_{(r,d)}\!=\!\Hom_\KK(A_r\!\otimes_\KK\! V_1,V_2)\!\smallsetminus\!\cX_{(r,d)}\!\ne\!\emptyset$. Hence there is $M \in \rep(K_r;V_1,V_2)$ such that $\cV(K_r,d)_M\!=\!\emptyset$. \end{proof}

\bigskip

\subsection{Modules of minimal type and applications} \label{S:MinType}
It is a consequence of work of Westwick \cite{We87} that  for ``most" objects $M \in \EKP(K_r)$ the inequality 
\[ \Delta_M(1) \ge r\!-\!1\]
holds. As we shall show below, the following definition naturally generalizes those $\EKP$-representations for which we have equality. 

\bigskip

\begin{Definition} Let $d \in \{1,\ldots, r\!-\!1\}$. We say that $M \in \repp(K_r,d)$ has {\it minimal type}, provided $\Delta_M(d)\!=\!d(r\!-\!d)$. \end{Definition} 

\bigskip
\noindent
Our approach rests on the following technical Lemma: 

\bigskip

\begin{Lem} \label{MinType1} Let $M \in \repp(K_r,d)$. Then
\[ C_M :=\{(m,\fv) \in M_2\!\times\!\Gr_d(A_r) \ ; \ m \in \im \psi_{M,\fv}\} \subseteq M_2\!\times\!\Gr_d(A_r)\] 
is a closed, irreducible subset of dimension $\dim C_M\!=\!d(r\!-\!d)\!+\!d\dim_\KK M_1$.
\end{Lem}

\begin{proof} We proceed in several steps.

\medskip
(a) \ \textit{Let $f : V \lra W$ be a $\KK$-linear map, $C \subseteq \Gr_q(V)$ be closed such that $\dim_\KK f(\fv)\!=\!q$ for all $\fv \in C$. Then 
     \[ \msim_{f,C} : C \lra \Gr_q(W) \ \ ; \   \   \fv \mapsto f(\fv)\]
     is a morphism of projective varieties.}
     
\smallskip
\noindent
We denote by $\mspl_W : \Gr_q(W) \lra \PP(\bigwedge^q(W))$ the Pl\"ucker embedding and observe that $(\mspl_W\circ\msim_{f,C})(\fv)\!=\!\bigwedge^q(f(\fv))\!=\!\bigwedge^q(f)(\bigwedge^q(\fv))$ for every $\fv \in C$. The 
map $\bigwedge^q(f) : \bigwedge^q(V) \lra \bigwedge^q(W)$ is $\KK$-linear and
$O_f\!:=\!\{ [x] \in \PP(\bigwedge^q(V)) \ ; \ \bigwedge^q(f)(x)\!\ne\!0\}$ is open and such that
\[ \varphi : O_f \lra \PP(\bigwedge^q(W)) \  ;  \ [x] \mapsto [\bigwedge^q(f)(x)]\]
is a morphism. Since $\mspl_V(C) \subseteq O_f$, it follows that $\mspl_W\circ\msim_{f,C}\!=\!\varphi\circ \mspl_V|_C$ is a morphism. Hence the same holds for $\msim_{f,C}$. \hfill $\diamond$

\medskip
(b) \ \textit{The map
         \[ \zeta_M : \Gr_d(A_r) \lra \Gr_{d\dim_\KK M_1}(M_2) \  \  ;   \  \   \fv \mapsto \psi_M(\fv\!\otimes_\KK\!M_1)\]
         is a morphism.}
         
\smallskip
\noindent
We first consider the map
\[ \eta_M : \Gr_d(A_r) \lra \Gr_{d\dim_\KK M_1}(A_r\!\otimes_\KK\!M_1) \ \ ; \ \ \fv \mapsto \fv\!\otimes_\KK\!M_1.\]
Then we have 
\[ (\mspl_{A_r\otimes_\KK M_1}\circ \eta_M)(\fv) = \bigwedge^{d\dim_\KK M_1}(\fv\!\otimes_\KK\!M_1) \cong \bigwedge^d(\fv)\!\otimes_\KK\!\bigwedge^{\dim_\KK M_1}(M_1) \cong \mspl_{A_r}(\fv)\!\otimes_\KK\!
\bigwedge^{\dim_\KK M_1}(M_1),\]
where the second to last isomorphism holds for dimension reasons. Hence $\mspl_{A_r\otimes_\KK M_1}\circ \eta_M$ is a morphism, so that $\eta_M$ enjoys the same property.

Consequently, $C\!:=\!\im \eta_M$ is closed. Since $\zeta_M\!=\!\msim_{\psi_M,C}\circ \eta_M$, our assertion follows from (a). \hfill $\diamond$

\smallskip
\noindent
The incidence variety $\cI_M\!:=\!\{(m,\fw) \in M_2\!\times\!\Gr_{d\dim_\KK M_1}(M_2) \ ; \ m \in \fw\}$ is known to be closed. By (b), the map
\[ \id_{M_2}\!\times\zeta_M : M_2\!\times\!\Gr_d(A_r) \lra M_2\!\times\!\Gr_{d\dim_\KK M_1}(M_2)\]
is a morphism. Hence $C_M\!=\!(\id_{M_2}\!\times \zeta_M)^{-1}(\cI_M)$ is closed as well. 

We consider the projection 
\[ \pr_2 : C_M \lra \Gr_d(A_r) \ \ ; \ \ (m,\fv) \mapsto \fv.\]
For every $\fv \in \Gr_d(A_r)$, the fiber $\pr_2^{-1}(\fv)\!\cong\!\im \psi_{M,v}$ is irreducible of dimension $d\dim_\KK M_1$. Note that $\KK^\times$ acts on $C_M$ via
\[ \alpha. (m,\fv) = (\alpha m,\fv) \  \  \  \  \forall \ \alpha \in \KK^\times, (m,\fv) \in C_M.\]
It follows that every irreducible component $C \subseteq C_M$ is $\KK^\times$-stable, so that $(m,\fv) \in C$ implies $(0,\fv) \in C$. The morphism
\[ \iota : \Gr_d(A_r) \lra C_M \  \  ;  \  \  \fv \mapsto (0,\fv)\]
thus yields $\pr_2(C)\!=\!\iota^{-1}(C)$, showing that $\pr_2(C)$ is closed. The irreducibility of $C_M$ now follows from \cite[(1.5)]{Fa04}, while its dimension is given by the fiber dimension theorem. \end{proof}

\bigskip
\noindent
For $M \in \rep(K_r)$ we put
\[ I_M := \bigcup_{\fv \in \Gr_d(A_r)}\im\psi_{M,\fv} \subseteq M_2.\]

\bigskip

\begin{Thm} \label{MinType2} Let $M \in \repp(K_r,d)$. Then we have
\[ \dim_\KK \ker\psi_M \le (r\!-\!d)(\dim_\KK M_1\!-\!\min\{\dim_\KK M_1,d\}).\]
In particular,
\[ \Delta_M(d) \ge (r\!-\!d)\min\{\dim_\KK M_1,d\}.\]
\end{Thm}

\begin{proof} We consider the representation
\[ P(M)\!:=\!(M_1,A_r\!\otimes_\KK\!M_1,\id_{A_r\otimes_\KK M_1}).\]
Since $P(M)\cong (\dim_\KK M_1)P_1(r)$ is projective, Lemma \ref{MinType1} shows that the set
\[ C_{P(M)} = \{ (x,\fv) \in (A_r\!\otimes_\KK\!M_1)\!\times\!\Gr_d(A_r) \ ; \ x \in \fv\!\otimes_\KK\!M_1\}\]
is closed, irreducible and of dimension $\dim C_{P(M)}\!=\!d\dim_\KK M_1\!+\!d(r\!-\!d)$. The projection
\[ \pr_1 : C_{P(M)} \lra A_r\!\otimes_\KK\!M_1\]
has image $I_{P(M)}\!=\!\bigcup_{\fv \in \Gr_d(A_r)}\fv\!\otimes_\KK\!M_1 \subseteq A_r\!\otimes_\KK\!M_1$. Since $\Gr_d(A_r)$ is complete and $C_{P(M)}$ is closed and irreducible, $I_{P(M)}$ is a closed, irreducible 
subset of $A_r\!\otimes_\KK\!M$. Moreover, given $x \in I_{P(M)}$, we have $\pr_1^{-1}(x)\cong \{\fv \in \Gr_d(A_r) \ ; \ x \in \fv\!\otimes_\KK\!M_1\}$. 

Let $\{m_1,\ldots, m_s\}$ be a basis of $M_1$, so that $s\!=\!\dim_\KK M_1$. Every $x \in A_r\!\otimes_\KK\!M_1$ is of the form $x\!=\!\sum_{i=1}^s v_i(x)\otimes m_i$. We let $\fc(x)\!:=\!\langle v_i(x) \ , \ i \in \{1,\ldots, s\}
\rangle \subseteq A_r$ be the coefficient space of $x$ and put $\ell(x)\!=\!\dim_\KK\fc(x)$. Writing $v_i(x)\!:=\!\sum_{j=1}^ra_{ij}(x)\gamma_j$ and $A(x)\!:=\!(a_{ij}(x)) \in \Mat_{s\times r}(k)$, we obtain $\dim_\KK\fc(x)\!=\!
\rk(A(x))$. Since $\rk(A(x))\!\le\!\min\{s,d\}$ for all $x \in I_{P(M)}$, it follows from $\sum_{i=1}^{\min\{s,d\}}\gamma_i\otimes m_i \in I_{P(M)}$ and the lower semicontinuity of $x \mapsto \rk(A(x))$ that   
\[ O_M := \{ x \in I_{P(M)} \ ; \ \ell(x)\!=\!\min\{s,d\}\}\]
is a dense open subset of $I_{P(M)}$. 

Let $x \in O_M$. Then $(x,\fv )\in \pr_1^{-1}(x)$ if and only if $\fc(x) \subseteq \fv$. If $s\!\ge\!d$, then $\pr^{-1}_1(x)\!=\!\{(x\,\fc(x))\}$. Alternatively, $\pr_1^{-1}(x) \cong \Gr_{d-s}(A_r/\fc(x))$,
so that $\dim \pr^{-1}_1(x)\!=\!(d\!-\!s)(r\!-\!d)$. The fiber dimension theorem now shows that 
\[\dim I_{P(M)} = ds\!+\!(r\!-\!d)\min\{s,d\}.\] 
Note that $M \in \repp(K_r,d)$ implies that $\ker\psi_M\cap I_{P(M)}\!= \{0\}$.  The affine dimension theorem \cite[(7.1)]{Ha77} now ensures that
\[ 0 = \dim \ker\psi_M\cap I_{P(M)} \ge \dim_\KK \ker \psi_M\!+\!\dim I_{P(M)}\!-\!\dim_\KK(A_r\!\otimes_\KK\!M_1),\]
whence
\[ \dim_\KK\ker\psi_M\le rs\!-\!\dim I_{P(M)} = (r\!-\!d)(s\!-\!\min\{s,d\}).\]
Finally, 
\[ \dim_\KK M_2 \ge \rk(\psi_M) = rs\!-\!\dim_\KK\ker \psi_M \ge rs\!-\!(r\!-\!d)(s\!-\!\min\{s,d\}) = ds\!+\!(r\!-\!d)\min\{s,d\},\]
so that $\Delta_M(d)\!\ge\!(r\!-\!d)\min\{\dim_\KK M_1,d\}$. \end{proof} 

\bigskip
\noindent
Theorem \ref{MinType2} has been used by the first named author \cite{Bi22} to determine the orbit representatives of the dimension vectors of the elementary $K_r$-modules relative to the actions given by 
$\sigma_{K_r}$ and the duality. This extends work by Ringel \cite{Ri16} concerning $r\!=\!3$, where module representatives were also identified. In this regard, the case $r\!=\!3$ appears to be special.  

We shall see later that the following result reflects the scarcity of indecomposable vector bundles $\cF$ on $\Gr_d(A_r)$ of rank $\rk(\cF)\!<\!d(r\!-\!d)$. 

\bigskip

\begin{Cor} \label{MinType3}  Let $M,N \in \repp(K_r,d)$ be representations. 
\begin{enumerate}
\item Suppose that $\Delta_M(d)\!<\!d(r\!-\!d)$. 
\begin{enumerate}
\item If $M$ is indecomposable, then $M\cong P_0(r)$, or $d\!\ne\!1$ and $M\!\cong P_1(r)$.
\item $M \cong \Delta_M(r)P_0(r)\!\oplus\!(\dim_\KK M_1)P_1(r)$ is projective. \end{enumerate}
\item Suppose that $\Delta_M(d)\!=\!d(r\!-\!d)$. 
\begin{enumerate} 
\item If $M$ is not projective, then every $f \in \Hom_{K_r}(M,N)\!\smallsetminus\!\{0\}$ is injective. 
\item $M$ is a brick or projective.
\item If $M$ is not projective, then $\dim_\KK M_1\!\ge\!d\!+\!1$. 
\item If $\dim_\KK M_1\!\ge\!d\!+\!1$, then $M$ is not projective. \end{enumerate} \end{enumerate} \end{Cor}

\begin{proof} (1a) Suppose first that $M\not\cong P_0(r)$ is indecomposable. Then $\Delta_M(r)\!\le\!0$ and 
\[ (\ast) \ \ \ \ (0) \lra -\Delta_M(r)P_0(r) \lra (\dim_\KK M_1)P_1(r) \lra M \lra (0)\]
is a projective resolution of $M$. In view of
\[ (\dim_\KK M_1)(r\!-\!d) = \Delta_M(d)\!-\!\Delta_M(r),\]
our current assumption in conjunction with Theorem \ref{MinType2} implies $d\!>\!\dim_\KK M_1$ and $\Delta_M(d)\!\ge\!(\dim_\KK M_1)(r\!-\!d)$, whence $\Delta_M(r)\!\ge\!0$ and $d\!\ne\!1$. Consequently, 
$\Delta_M(r)\!=\!0$, so that ($\ast$) and the indecomposability of $M$ show that $M\cong P_1(r)$. 

(1b) Since $\repp(K_r,d)$ is closed under taking subrepresentations, the assertion follows from the Theorem of Krull-Remak-Schmidt.

(2a) Let $f \in \Hom_{K_r}(M,N)\!\smallsetminus\!\{0\}$. Then the terms of the short exact sequence
\[ (0) \lra \ker f \lra M \stackrel{f}{\lra} \im f \lra (0)\]
belong to $\repp(K_r,d)$ and
\[ \Delta_M(d) = \Delta_{\ker f}(d)\!+\!\Delta_{\im f}(d).\]
If $\Delta_{\ker f}(d)\!\ne\!0$, then $\Delta_{\im f}(d)\!<\!d(r\!-\!d)$, so that (1) shows that $\im f$ is projective. In view of Theorem \ref{MinType2}, the assumption $\Delta_{\im f}(d)\!=\!0$ implies $(\im f)_1\!=\!(0)$, whence
$0\!=\!-d\dim_\KK (\im f)_2$, which contradicts $f\!\ne\!0$. Hence $\ker f$ is also projective, and so is $M\cong \ker f\!\oplus\!\im f$. As $M$ is not projective, we conclude that $\Delta_{\ker f}(d)\!=\!0$, whence 
$\ker f\!=\!(0)$. 

(2b) Suppose that $M$ is not projective and let $f \in \End_{K_r}(M)\!\smallsetminus\!\{0\}$. In view of (2a), $f$ is bijective. Hence there is a non-zero eigenvalue $\alpha \in \KK$ and another application of (2a) yields
$f\!=\!\alpha\id_M$. 

(2c) As $M$ is a non-projective brick, the exact sequence ($\ast$) yields $\Delta_M(r)\!<\!0$ and
\[ (\dim_\KK M_1)(r\!-\! d) = \Delta_{(\dim_\KK M_1) P_1(r)}(d) > \Delta_M(d) = d(r\!-\!d),\]
so that $\dim_\KK M_1\!\ge\!d\!+\!1$. 

(2d) Suppose that $M$ is projective. Then $M \cong \Delta_M(r)P_0(r)\!\oplus\!(\dim_\KK M_1)P_1(r)$, and we have
\[ 0 \le \Delta_M(r) = \Delta_M(d)\!-\!(\dim_\KK M_1)(r\!-\!d) \le -(r\!-\!d),\]
a contradiction. Hence $M$ is not projective.  \end{proof}

\bigskip

\begin{Cor} \label{MinType4} Suppose that $M \in \repp(K_r,d)$ is non-projective and of minimal type. Then $M \not \in \repp(K_r,d\!+\!1)$. \end{Cor}

\begin{proof} Since $M$ is non-projective and of minimal type, Corollary \ref{MinType3}(2b),(2c) provides $x\!\ge\!d\!+\!1$ such that $\udim M\!=\!(x,d(r\!-\!d\!+\!x))$. We thus have
\[ (\ast) \ \ \ \ \Delta_M(d\!+\!1) = d(r\!-\!d)\!-\!x.\]
Suppose that $M \in \repp(K_r,d\!+\!1)$. Since $x\!\ge\!d\!+\!1$, Theorem \ref{MinType2} in conjunction with ($\ast$) gives $d(r\!-\!d)\!-\!x\!\ge\!(d\!+\!1)(r\!-\!d\!-\!1)$, whence $d\!+\!1\! \le\! x\!\le\! 2d\!+\!1\!-\!r$. As this 
contradicts $d\!\le\!r\!-\!1$, we conclude that $M \not \in \rep(K_r,d\!+\!1)$. \end{proof} 

\bigskip

\begin{Remarks} (1) It follows that every regular $M \in \repp(K_r,d)$ of minimal type is a brick. 

(2) Let $(V_1,V_2)$ be a pair of $\KK$-vector spaces. If $\Delta_{(V_1,V_2)}(d)\!\ge\!d(r\!-\!d)$, then Proposition \ref{Fam5} ensures that $\repp(K_r,d)\cap\rep(K_r;V_1,V_2)$ is not empty. Alternatively, this set is either empty 
(namely, if and only if $\Delta_M(r)\!<\!0$), or it consists of one $(\GL(V_2)\!\times\!\GL(V_1))$-orbit of projective representations. \end{Remarks}

\bigskip
\noindent
Since the category $\repp(K_r,d)$ is closed under taking extensions (see Theorem \ref{Fam5}), our next result provides a first characterization of objects in $\repp(K_r,d)$ in terms of those of minimal type.

\bigskip 

\begin{Prop} \label{MinType5} Let $M \in \repp(K_r,d)$ be such that $\Delta_M(d)\!>\!d(r\!-\!d)$. There exists an exact sequence
\[ (0) \lra a P_0(r) \lra M \stackrel{\pi}{\lra} M_{\min} \lra (0),\]
where $M_{\min}$ has minimal type and $a \in \NN$. \end{Prop} 

\begin{proof} By Lemma \ref{MinType1}, the set
\[ C_M =\{(m,\fv) \in M_2\!\times\!\Gr_d(A_r) \ ; \ m \in \im \psi_{M,\fv}\} \subseteq M_2\!\times\!\Gr_d(A_r)\] 
is a closed, irreducible subset of dimension $\dim C_M\!=\!d(r\!-\!d)\!+\!d\dim_\KK M_1$. Note that the morphism
\[ \pr_1 : C_M \lra M_2 \ \ ; \  \  (m,\fv) \mapsto m\]
has image $I_M\! =\!\bigcup_{\fv \in \Gr_d(A_r)}\im\psi_{M,\fv} \subseteq M_2$, so that $\Gr_d(A_r)$ being complete implies that $I_M$ is closed.\footnote{Observe that $\pr_1$ is the restriction of the projection 
$M_2\!\times\!\Gr_d(A_r)\lra M_2$, which is closed.} Moreover, $I_M$ is irreducible, conical and of dimension
\[ \dim I_M \le d(r\!-\!d)\!+\!d\dim_\KK M_1\!<\!\dim_\KK M_2.\]
The Noether Normalization Theorem thus implies the existence of a linear subspace $X \subseteq M_2$ of dimension $\dim_\KK M_2\!-\!\dim I_M$ and such that
\[ X\cap I_M = \{0\},\]
cf.\ \cite[(II.3.14)]{Ku}. Since $\dim_\KK X\!\ge\!\dim_\KK M_2\!-\!d\dim_\KK M_1\!-\!d(r\!-\!d)\!=\!\Delta_M(d)\!-\!d(r\!-\!d)\!>\!0$, we can find a subspace $(0) \subsetneq Y \subseteq X$ such that
\begin{enumerate}
\item[(i)] $\dim_\KK Y\!=\!\Delta_M(d)\!-\!d(r\!-\!d)$, and
\item[(ii)] $Y\cap I_M\!=\!\{0\}$. \end{enumerate}
We now consider $M_{\min}\!:=\!N\!:=\!(M_1,M_2/Y)$ together with the canonical projection $\pi_2 : M_2 \lra N_2$ and put $\psi_N\!:=\!\pi_2\circ \psi_M$. Then $\pi\!:=\!(\id_{M_1},\pi_2)$ defines a surjective morphism
$\pi : M \lra N$ such that $\ker\pi \cong (\dim_\KK Y)P_0(r)$. Given $\fv \in \Gr_d(A_r)$, (ii) yields $\im\psi_{M,\fv}\cap Y\!=\!(0)$, whence
\[ \ker \psi_{N,\fv} = \psi_{M,\fv}^{-1}(Y) = (0).\] 
Consequently, $M_{\min} \in \repp(K_r,d)$, while (i) yields $\Delta_{M_{\min}}(d)\!=\!\Delta_M(d)\!-\!\Delta_{(\dim_\KK Y)P_0(r)}(d)\!=\!\Delta_M(d)\!-\!\dim_\KK Y\!=\!d(r\!-\!d)$. \end{proof}

\bigskip
\noindent
We turn to the problem of embedding $K_r$-representations of minimal type into $K_{r+s}$-representations of the same type. 

\bigskip

\begin{Lem} \label{MinType6} Given $M \in \rep(K_r)$ such that $\dim_\KK M_1\!\ge\!d$, we have
\[ I_M = \bigcup_{W \in \Gr_d(M_1)} \psi_M(A_r\!\otimes_\KK\!W)\!=:\!J_M.\]
\end{Lem}

\begin{proof} Recall that $P(M)\!=\!(M_1,A_r\!\otimes_\KK\!M_1,\id_{A_r\!\otimes_\KK M_1})$. We first show that $I_{P(M)}\!=\!J_{P(M)}$. Let $x \in I_{P(M)}$. Then there is $\fv \in \Gr_d(A_r)$ such that 
$x \in \fv\!\otimes_\KK\!M_1$. Given a basis $\{a_1,\ldots, a_d\} \subseteq \fv$, we write $x\!=\!\sum_{i=1}^d a_i\otimes m_i$, so that there is $W \in \Gr_d(M_1)$ such that $\langle\{m_1,\ldots, m_d\}\rangle \subseteq W$. 
Consequently, $x \in A_r\!\otimes_\KK\!W$. As a result, $I_{P(M)} \subseteq J_{P(M)}$. The proof of the reverse inclusion is analogous. 

For general $M$, we note that $I_M\!=\!\psi_M(I_{P(M)})\!=\!\psi_M(J_{P(M)})\!=\!J_M$. \end{proof} 

\bigskip

\begin{Prop} \label{MinType7} Let $M \in \repp(K_r,d)$ be of minimal type such that $\dim_\KK M_1\!\ge\!d\!+\!1$. Then there exists $N \in \repp(K_{r+1},d)$ of minimal type such that
\[ N|_{K_r} \cong M\!\oplus\!dP_0(r).\]
\end{Prop}

\begin{proof} Writing $A_{r+1}\!=\!A_r\!\oplus\!\KK a$, we first construct $Q \in \repp(K_{r+1},d)$ such that
\[ Q|_{K_r} \cong M\!\oplus (\dim_\KK M_1)P_0(r).\]
We put $Q_1\!:=\!M_1$, $Q_2\!:=\!M_2\!\oplus\! Y_2$, with $\dim_\KK Y_2\!=\!\dim_\KK M_1$ and pick an isomorphism
\[ \lambda : \KK a\!\otimes_\KK\!M_1 \stackrel{\sim}{\lra} Y_2.\]
Interpreting $\psi_M$ and $\lambda$ as elements of $\Hom_\KK(A_{r+1}\!\otimes_\KK\!Q_1,Q_2)$, we have
\[ (\ast) \ \ \ \ \ \ \ \ \  \psi_M(A_r\!\otimes_\KK\!W)\cap \lambda(\KK a\!\otimes_\KK\!W) = (0) \ \  \text{for all} \ \  W \in \Gr_d(M_1).\] 
We consider the representation $Q\!:=\!(M_1,Q_2,\psi_M\!+\!\lambda) \in \rep(K_{r+1})$.  Let $x \in \ker\psi_Q\cap I_{P(Q)}$. Lemma \ref{MinType6} provides $W \in \Gr_d(M_1)$ such that
$x \in \ker\psi_Q\cap (A_{r+1}\!\otimes_\KK W)$. Writing $x\!=\!y\!+\!z$, with $y \in A_r\!\otimes_\KK\!W$ and $z \in \KK a\!\otimes_\KK\!W$ we obtain $\lambda(z)\!=\!-\psi_M(y) \in \psi_M(A_r\!\otimes_\KK\!W)\cap 
\lambda(ka\!\otimes_\KK\!W)$, so that ($\ast$) implies $\lambda(z)\!=\!0\!=\!\psi_M(y)$. As $\lambda$ is injective, we conclude that $z\!=\!0$, while Lemma \ref{MinType6} in conjunction with $M \in \repp(K_r,d)$ yields 
$y \in \ker\psi_M\cap J_{P(M)}\!=\!\ker\psi_M\cap I_{P(M)}\!=\!(0)$. As a result, $x\!=\!0$, so that $Q \in \repp(K_{r+1},d)$. By construction, we have $Q|_{K_r} \cong M\!\oplus (\dim_\KK M_1)P_0(r)$.

Since $\Delta_Q(d)\!=\!\Delta_M(d)\!+\!\dim_\KK M_1\!=\!\dim_\KK M_1\!-\!d\!+\!d(r\!+\!1\!-\!d)$ and $\dim_\KK M_1\!>\!d$, Proposition \ref{MinType5} provides a short exact sequence
\[ (0) \lra (\dim_\KK M_1\!-\!d)P_0(r\!+\!1) \lra Q \stackrel{\pi}{\lra} Q_{\min} \lra (0),\]
where $N\!:=\!Q_{\min}$ has minimal type. Let $\iota : M \lra Q|_{K_r}$ be the given injection. Since $\dim_\KK M\!\ge\!d\!+\!1$, Corollary \ref{MinType3}(2d) ensures that $M$ is not projective. As $\ker \pi|_{K_r} \cong  
(\dim_\KK M_1\!-\!d)P_0(r)$, the map $\pi \circ \iota : M \lra  N|_{K_r}$ is not zero, and hence by Corollary \ref{MinType3}(2a) injective. The restriction of the sequence above yields
\[ (0) \lra (\dim_\KK M_1\!-\!d)P_0(r) \stackrel{\binom{\alpha}{\beta}}{\lra} (\dim_\KK M_1)P_0(r)\!\oplus\!M \stackrel{(\zeta,\pi\circ \iota)}{\lra} N|_{K_r} \lra (0),\]
and thanks to  \cite[(I.5.6)]{ARS95} there results a push-out and pull-back diagram 
\[ \begin{tikzcd}   (\dim_\KK M_1\!-\!d)P_0(r)  \arrow[r, "\alpha"] \arrow[d,"\beta"] & (\dim_\KK M_1)P_0(r) \arrow[d,"-\zeta"] \\
                           M \arrow[r,"\pi\circ \iota"] & N|_{K_r}.  
\end{tikzcd} \] 
Since $\pi\circ\iota$ is injective, so is $\alpha$. It follows that $dP_0(r) \cong \coker\alpha \cong \coker(\pi\circ \iota)$ (cf.\ \cite[(I.5.6)]{ARS95}), so that there is an exact sequence
\[ (0) \lra M \stackrel{\pi\circ \iota}{\lra} N|_{K_r} \lra dP_0(r) \lra (0).\]
Consequently, $N|_{K_r} \cong M\!\oplus\! dP_0(r)$. \end{proof}

\bigskip

\begin{Remark} Suppose that $M \in \repp(K_r,d)$ is projective and of minimal type. Since $\Delta_{P_1(r)}(d)\!=\!(r\!-\!d)$, there is $\ell \in \{0,\ldots, d\}$ such that $M \cong \ell P_1(r)\!\oplus\!(d\!-\!\ell)(r\!-\!d)P_0(r)$.
We consider $N\!:= \ell P_1(r\!+\!1)\!\oplus\!(d\!-\!\ell)(r\!+\!1\!-\!d)P_0(r\!+\!1) \in \repp(K_{r+1},d)$. Then we have
\[ \Delta_N(d) = \ell (r\!+\!1\!-\!d)\!+\!(d\!-\!\ell)(r\!+\!1\!-\!d) = d(r\!+\!1\!-\!d),\]
so that $N$ has minimal type. Moreover,
\[ N|_{K_r} \cong \ell P_1(r)\!\oplus\!\ell P_0(r)\!\oplus\!(d\!-\!\ell)(r\!+\!1\!-\!d)P_0(r),\]
showing that $N|_{K_r} \cong M\!\oplus\! dP_0(r)$. By the remark at the beginning of Section \ref{S:CatRep}, Proposition \ref{MinType7} thus also holds in case $\dim_\KK M_1\!\le\!d$. \end{Remark} 

\bigskip

\subsection{Examples: Representations and hyperplane arrangements} \label{S:Ha} The representations of $\EKP(K_r)\!=\!\repp(K_r,1)$ discussed in this section correspond to the logarithmic bundles \cite{DK93} and 
Schwarzenberger bundles \cite{Sc61b}. Our notation is meant to reflect this relationship  

Let $r\!\ge\!2$, $V$ be an $r$-dimensional vector space and $m\!\ge\!r$. We say that $\ff\!:=\!(f_1,\ldots, f_m) \in (V^\ast)^m$ is in \textit{general position}, provided $V^\ast\!=\!\langle f_i \ ; \ i \in J \rangle$ for every subset $J 
\subseteq \{1,\ldots, m\}$ such that $|J|\!=\!r$. We consider the projective space $\PP^{r-1}\!=\!\PP(V)$. A {\it hyperplane arrangement} of $V$ is an $m$-tuple $\cH\!:=\!(H_1,\ldots, H_m)$ of linear hyperplanes of 
$\PP(V)$. We say that $\cH$ is {\it in general position}, provided for every $\ell\! \le\! r$ and every $\ell$-element subset $J_\ell\subseteq \{1,\ldots, m\}$, we have $\dim_\KK \bigcap_{i\in J_\ell}H_i\!=\!r\!-\!1\!-\!\ell$. 
(Here we put $\dim \emptyset\!=\!-1$.) By definition, each $H_i\:=\!Z(f_i)$ is the set of zeros for some $f_i \in V^\ast\!\smallsetminus\!\{0\}$. It follows that $\cH$ is in general position if and only if $(f_1,\ldots, f_m)$ enjoys this 
property. 

Let $\ff\!:=\!(f_1,\ldots, f_m) \in (V^\ast)^m$. Following \cite[\S1]{DK93}, we set 
\[ I_\ff\!:=\{\lambda \in \KK^m \ ; \ \sum_{i=1}^m\lambda_if_i\!=\!0\} \ \  \text{and} \ \ W_m\!:=\!\{\lambda \in \KK^m \ ; \ \sum_{i=1}^m\lambda_i\!=\!0\},\] 
and refer to the linear map
\[ t_\ff : V\!\otimes_\KK\!I_\ff  \lra W_m \ \ ; \  \  v \otimes \lambda \mapsto (\lambda_1f_1(v), \ldots, \lambda_mf_m(v))\]
as the {\it fundamental tensor} of $\ff$ (viewed as an element of $V^\ast\!\otimes_\KK\!I_\ff^\ast\!\otimes_\KK\!W_m$). 

Given $\ff\!:=\!(f_1,\ldots, f_m) \in (A_r^\ast)^m$, we define $Q_\ff\!:=\!(\KK^m,\KK^m,(Q_\ff(\gamma_i))_{1\le i \le r}) \in \rep(K_r)$ via
\[ Q_\ff(\gamma_i)(\lambda) = (\lambda_1f_1(\gamma_i), \ldots, \lambda_mf_m(\gamma_i)) \ \ \ \ \ \forall \ i \in \{1,\ldots, r\}, \lambda \in \KK^m.\]
Then, setting $V\!:=\!A_r$, we see that
\[ M_\ff\!:=\!(I_\ff,W_m, (Q_\ff(\gamma_i)|_{I_\ff})_{1\le i \le r})\] 
is a subrepresentation of $Q_\ff$ such that $\psi_{M_\ff}\!=\!t_\ff$. If $\cH$ is a hyperplane arrangement given by $\ff \in ((A_r)^\ast\!\smallsetminus\!\{0\})^m$, we abuse notation and  write $M_\cH\!=\!M_\ff$
as well as $t_\cH\!=\!t_\ff$. 

\bigskip

\begin{Lem}\label{Ha1} Suppose that $m\!\ge\!r\!\ge\!3$, and let $\cH\!=\!(H_1,\ldots, H_m)$ be a hyperplane arrangement in general position. 
\begin{enumerate}
\item $M_\cH \in \EKP(K_r)$ has dimension vector $\udim M_\cH\!=\!(m\!-\!r,m\!-\!1)$. In particular, $M_\cH$ has minimal type. 
\item If $m\!\ge\! r\!+\!2$, then $M_\cH \not \in \repp(K_r,2)$. \end{enumerate} \end{Lem}

\begin{proof} (1) Let $a \in A_r\!\smallsetminus\!\{0\}$. Then we have
\[ a_{M_\cH}(\lambda) = \psi_{M_\cH}(a\otimes \lambda) = t_\cH(a\!\otimes \lambda)\]
for all $\lambda \in I_\ff$. We may now apply \cite[(1.5)]{DK93} (which holds for arbitrary ground fields) to see that the linear map $a_{M_\cH}$ is injective. Consequently, $M_\cH \in \EKP(K_r)$. 

We write $\cH\!:=\!(Z(f_1),\ldots, Z(f_m))$, $\ff\!:=\!(f_1,\ldots, f_m)$. By construction, we have $\dim_\KK(M_\cH)_2\!=\!\dim_\KK W_m\!=\!m\!-\!1$. Since $m\!\ge\!r$ and $\cH$ is in general position, we have 
$A_r^\ast\!=\!\langle f_1,\ldots, f_r\rangle$, so that $\{f_1,\ldots, f_r\} \subseteq A_r^\ast$ is a basis of $A_r^\ast$. Hence the linear map
\[\alpha_\ff : \KK^m \lra A_r^\ast \ \ ; \  \  \lambda \mapsto  \sum_{i=1}^m\lambda_if_i\] 
is surjective, whence $\dim_\KK I_\ff\!=\!\dim_\KK\ker\alpha_\ff\!=\!m\!-\!r$. 

(2) This is a direct consequence of (1), Corollary \ref{MinType3}(2d) and Corollary \ref{MinType4}. \end{proof}

\bigskip
\noindent
A subclass of the modules $M_\cH$ considered above is given by the family $(M_\cS[m])_{m\ge r+1}$ of {\it Schwarzenberger modules}. Let $X_1,X_2$ be indeterminates over $\KK$. For $m\!\ge\!r\!+\!1$, we consider the 
representation $M_\cS[m] \in \rep(K_r)$, defined via
\[ M_\cS[m]_1 := \KK[X_1,X_2]_{m-r-1} \ \ ; \ \ M_\cS[m]_2 := \KK[X_1,X_2]_{m-2} \ \ ; \ \ M_\cS[m](\gamma_i)(a) := X_1^{i-1}X_2^{r-i}a \ \ \ \forall \ a \in M[m]_1\]
for $1\!\le\!i\!\le\!r$. Then we have $M_\cS[m] \in \EKP(K_r)$, while $\udim M_\cS[m]\!=\!(m\!-\!r,m\!-\!1)$. The modules $M_\cS[m]$ turn out to correspond to vector bundles on $\PP^{r-1}$ that were introduced by 
Schwarzenberger, cf.\ \cite{Sc61b}.

\bigskip

\subsection{The generic canonical decomposition} \label{S:GCD} Our first result provides the {\it canonical decomposition} \cite[p.85]{Ka80} of certain dimension vectors. For ease of notation, we write 
$\Delta_{(V_1,V_2)}\!:=\!\Delta_{(V_1,V_2)}(1)$.

Suppose that $V_1,V_2$ are vector spaces such that $\repp(K_r,d)\cap\rep(K_r;V_1,V_2)\!\ne\!\emptyset$. According to Theorem \ref{MinType2}, the assumption $\Delta_{(V_1,V_2)}(d)\!=\!0$ implies
$\dim_\KK V_1\!=\!0$ and hence $\dim_\KK V_2\!=\!\Delta_{(V_1,V_2)}(d)\!=\!0$. Hence the technical condition $\Delta_{(V_1,V_2)}\!\ge\!1$ of the ensuing results only excludes the trivial case. 

\bigskip

\begin{Lem} \label{GCD1} Let $V_1,V_2$ be vector spaces such that $\Delta_{(V_1,V_2)}\!\ge\!1$ and write $\dim_\KK V_1\!=\!j\Delta_{(V_1,V_2)}\!+\!b$ for $b \in \{0,\ldots,\Delta_{(V_1,V_2)}\!-\!1\}$. 
Then the set
\[ \rep(K_2;V_1,V_2)_0 := \{M \in \rep(K_2;V_1,V_2) \ ; \ M \cong (\Delta_{(V_1,V_2)}\!-\!b)P_j(2)\!\oplus\! bP_{j+1}(2)\}\]
is $(\GL(V_2)\!\times\!\GL(V_1))$-stable, dense and open. \end{Lem}

\begin{proof} By definition, $ \rep(K_2;V_1,V_2)_0$ is $(\GL(V_2)\!\times\!\GL(V_1))$-stable.

Since $P_j(2),P_{j+1}(2)$ are bricks, Corollary \ref{MV3} implies that these representations belong to their respective open sheets. By virtue of $\dim_\KK V_i\!=\!(j\!+\!i\!-\!1)\Delta_{(V_1,V_2)}\!+\!b$ and 
$\Ext^1_{K_2}(P_\ell(2),P_k(2))\!=\!(0)$ for $\{\ell,k\}\!=\!\{j,j\!+\!1\}$, the conditions of \cite[Prop.3(a)]{Ka82} are met, so that our assertion follows from Kac's Theorem. \end{proof}

\bigskip
\noindent
Let $r\!\ge\!3$. We consider the map
\[ \Res : \rep(K_r;V_1,V_2)\!\times\!\Inj_\KK(A_2,A_r) \lra \rep(K_2;V_1,V_2) \ \ ; \ \ (M,\alpha) \mapsto \alpha^\ast(M).\]

\bigskip

\begin{Prop} \label{GCD2}  Let $V_1,V_2$ be $\KK$-vector spaces such that $\Delta_{(V_1,V_2)}\!\ge\!1$, and write 
\[\dim_\KK V_1= j\Delta_{(V_1,V_2)}\!+\!b\] 
for some $b \in \{0,\ldots,\Delta_{(V_1,V_2)}\!-\!1\}$. The following statements hold:
\begin{enumerate}
\item $\Res$ is a morphism.
\item The set $\Res((\EKP(K_r)\cap\rep(K_r;V_1,V_2))\!\times\!\Inj_\KK(A_2,A_r))$ is open.
\item Suppose that $\EKP(K_r)\cap \rep(K_r;V_1,V_2)\!\ne\!\emptyset$. Then 
\[ U := \{(M,\alpha) \in \rep(K_r;V_1,V_2)\!\times\! \Inj_\KK(A_2,A_r) \ ; \  \alpha^\ast(M) \cong (\Delta_{(V_1,V_2)}\!-\!b)P_j(2)\!\oplus\!bP_{j+1}(2)\}\]
is a dense, open subset of $\rep(K_r;V_1,V_2)\!\times\! \Inj_\KK(A_2,A_r)$.   \end{enumerate} \end{Prop}

\begin{proof} (1) We interpret $\rep(K_r;V_1,V_2)$ as $\Hom_\KK(A_r, \Hom_\KK(V_1,V_2))$, so that 
\[ \Res(\varrho,\alpha) = \varrho\circ \alpha\]
corresponds to matrix multiplication. Consequently, $\Res$ is a morphism.

(2) Note that 
\[ \Res((\EKP(K_r)\cap\rep(K_r;V_1,V_2))\!\times\!\Inj_\KK(A_2,A_r)) = \bigcup_{\alpha \in \Inj_\KK(A_2,A_r)} \alpha^\ast(\EKP(K_r)\cap \rep(K_r;V_1,V_2)).\]
Given $\alpha \in \Inj_\KK(A_2,A_r)$, we write $A_r\!=\!\im\alpha\!\oplus\!\fw \cong  A_2\!\oplus\!\fw$ for some $\fw \in \Gr_{r-2}(A_r)$. Then we have
\[ \Hom_\KK(A_r,\Hom_\KK(V_1,V_2)) \cong \Hom_\KK(A_2,\Hom_\KK(V_1,V_2))\!\oplus \Hom_\KK(\fw,\Hom_\KK(V_1,V_2)),\]
and $\alpha^\ast$ is the projection onto the first summand.

For varieties $X,Y,$ the projection $\pr_X : X\!\times\!Y \lra X$ is easily seen to be open. We apply this to the situation above to see that $\alpha^\ast : \rep(K_r;V_1,V_2) \lra \rep(K_2;V_1,V_2)$ is open. 
Since Proposition \ref{CatRep4} ensures that $\EKP(K_r)\cap \rep(K_r;V_1,V_2)$ is open, our assertion follows.

(3) In view of (2) and Lemma \ref{GCD1}, our current assumption implies that
\[ \Res((\EKP(K_r)\cap\rep(K_r;V_1,V_2))\!\times\!\Inj_\KK(A_2,A_r))\cap \rep(K_2;V_1,V_2)_0\]
is a non-empty, open subset of $\rep(K_2;V_1,V_2)$. By (1), 
\[  U = \Res^{-1}(\rep(K_2;V_1,V_2)_0)\]
is thus a non-empty, open subset of $\rep(K_r;V_1,V_2)\!\times\! \Inj_\KK(A_2,A_r)$. Being an open subset of the irreducible variety $\rep(K_r;V_1,V_2)\!\times\! \Hom_\KK(A_2,A_r)$, the variety $\rep(K_r;V_1,V_2)\!\times\! 
\Inj_\KK(A_2,A_r)$ is irreducible, so that $U$ lies dense in $\rep(K_r;V_1,V_2)\!\times\! \Inj_\KK(A_2,A_r)$. \end{proof}

\bigskip
\noindent
Recall the notion of the generic decomposition $M_{\gen}$ of $M \in \repp(K_r,d)$, which we have discussed at the end of Section \ref{S:Rest}.

\bigskip

\begin{Cor} \label{GCD3} Suppose that $V_1,V_2$ are $\KK$-vector spaces such that $\Delta_{(V_1,V_2)}\!\ge\!1$ and $\EKP(K_r)\cap\rep(K_r;V_1,V_2)\!\ne\!\emptyset$. Let $(j,b) \in \NN_0\!\times\!\{0,\ldots,
\Delta_{(V_1,V_2)}\!-\!1\}$ be given by $\dim_\KK V_1\!=\!j\Delta_{(V_1,V_2)}\!+\!b$. Then 
\[ O_{\gen} := \{ M \in \EKP(K_r)\cap\rep(K_r;V_1,V_2) \ ; \ M_{\gen}\! =\! (\Delta_{(V_1,V_2)}\!-\!b)P_j(2)\!\oplus\!bP_{j+1}(2)\}\] 
is a dense open subset of $\rep(K_r;V_1,V_2)$. \end{Cor} 

\begin{proof} According to Proposition \ref{GCD2}
\begin{eqnarray*}  
\{(M,\alpha) \in \rep(K_r;V_1,V_2)\!\times\! \Inj_\KK(A_2,A_r) \ ; \  \alpha^\ast(M)\! & \cong & \!(\Delta_{(V_1,V_2)}\!-\!b)P_j(2)\!\oplus\!bP_{j+1}(2)\} \\
                                                                                                                                & = & U= \Res^{-1}(\rep(K_2;V_1,V_2)_0)
\end{eqnarray*}
is a dense, open subset of $\rep(K_r;V_1,V_2)\!\times\!\Inj_\KK(A_2,A_r)$. The projection $\pr_1 : \rep(K_r;V_1,V_2)\!\times\! \Inj_\KK(A_2,A_r) \lra \rep(K_r;V_1,V_2)$ is open, so that $\pr_1(U)$ is open and  
dense in $\rep(K_r;V_1,V_2)$. In view of Proposition \ref{CatRep4}(1), the subset $\pr_1(U)\cap\EKP(K_r)$ enjoys the same properties. 

If $M \in O_{\gen}$, then there is $\alpha \in \Inj_\KK(A_2,A_r)$ such that 
\[ \alpha^\ast(M) \cong (\Delta_{(V_1,V_2)}\!-\!b)P_j(2)\!\oplus\!bP_{j+1}(2),\]
whence $(M,\alpha) \in U$ and $M \in \pr_1(U)\cap\EKP(K_r)$.

Conversely, let $M \in \pr_1(U)\cap\EKP(K_r)$. Then there is $\alpha \in \Inj_\KK(A_2,A_r)$ such that 
\[ (\ast) \ \ \ \ \alpha^\ast(M) \cong (\Delta_{(V_1,V_2)}\!-\!b)P_j(2)\!\oplus\!bP_{j+1}(2).\]  
Let $\pi : \rep(K_2;V_1,V_2) \lra \Iso(K_2;V_1,V_2)$ be the canonical projection. Thanks to Lemma \ref{GCD1}, $\Iso(K_2;V_1,V_2)_0\!:=\!\pi(\rep(K_2;V_1,V_2)_0)$ is a non-empty open subset of $\Iso(K_2;V_1,V_2)$.
Owing to ($\ast$), Proposition \ref{Rest2} implies that
\[ U_M\!:=\! \res_M^{-1}(\Iso(K_2;V_1,V_2)_0) = \{\fv \in \Gr_2(A_r) \ ; \ [M|_\fv]\!=\![(\Delta_{(V_1,V_2)}\!-\!b)P_j(2)\!\oplus\!bP_{j+1}(2)]\}\]
is a dense open subset of $\Gr_2(A_r)$. In view of Corollary \ref{Rest3}, $U_M$ therefore intersects the set
\[ O_M := \{ \fv \in \Gr_2(A_r) \ ; \  n_i(M,\fv) = n_i(M) \ \ \ \forall \ i \in \NN_0\}.\]
Consequently, 
\[ n_i(M) = \left\{ \begin{array}{cc} 0 & i \not\in \{j,j\!+\!1\} \\ \Delta_{(V_1,V_1)}\!-\!b & i\!=\!j \\ b & i\!=\!j\!+\!1, \end{array} \right.\]
so that $M \in O_{\gen}$. It follows that $O_{\gen}\!=\!\pr_1(U)\cap\EKP(K_r)$ has the asserted properties. \end{proof} 

\bigskip

\begin{Cor} \label{GCD4} Let $V_1,V_2$ be $\KK$-vector spaces such that $\Delta_{(V_1,V_2)}\!\ge\!\max\{\dim_\KK V_1,r\!-\!1\}$. Then we have 
\[ O_{\gen} = \{ M \in \EKP(K_r)\cap\rep(K_r;V_1,V_2) \ ; \ M_{\gen}\! =\! (\Delta_{(V_1,V_2)}\!-\!\dim_\KK V_1)P_0(2)\!\oplus\!(\dim_\KK V_1)P_1(2)\}.\]
\end{Cor}

\begin{proof} Since $\Delta_{(V_1,V_2)}\!\ge\!r\!-\!1$, Proposition \ref{CatRep4} shows that $\EKP(K_r)\cap\rep(K_r;V_1,V_2)\!\ne\!\emptyset$. If $\Delta_{(V_1,V_2)}\!\ge\!\dim_\KK V_1\!+\!1$, we obtain
\[ \dim_\KK V_1 = 0\Delta_{(V_1,V_2)}\!+\!\dim_\KK V_1,\]
where $\dim_\KK V_1\le\!\Delta_{(V_1,V_2)}\!-\!1$. The result thus follows from Corollary $\ref{GCD3}$. Alternatively, $\dim_\KK V_1\!=\!\Delta_{(V_1,V_2)}$ and Corollary \ref{GCD3} implies
$M_{\gen}\!=\!(\dim_\KK V_1)P_1(2)$. \end{proof}

\bigskip

\begin{Cor} \label{GCD5} Let $V_1,V_2$ be $\KK$-vector spaces such that $\Delta_{(V_1,V_2)}\!\ge\!1$, $M \in \EKP(K_r)\cap\rep(K_r;V_1,V_2)$. If there exist $\fv_0 \in \Gr_2(A_r)$ and $c,d,j \in \NN_0$ such that 
\[ [M|_{\fv_0}] = cP_j(2)\!\oplus\!dP_{j+1}(2),\]
then $M_{\gen}\!=\!cP_j(2)\!\oplus\!dP_{j+1}(2)$ and $M \in O_{\gen}$. \end{Cor}

\begin{proof} By assumption, we have $\Delta_{(V_1,V_2)}\!=\!\Delta_M\!=\!c\Delta_{P_j(2)}\!+\!d\Delta_{P_{j+1}(2)}\!=\!c\!+\!d$, while $\dim_\KK V_1\!=\!cj\!+\!d(j\!+\!1)\!=\!j\Delta_{(V_1,V_2)}\!+\!d$. If $c\!\ne\!0$,
then $0\!\le\!d\!\le\!\Delta_{(V_1,V_2)}\!-\!1$ and  Proposition \ref{GCD2}(3) implies that 
\[ U := \{(N,\alpha) \in \rep(K_r;V_1,V_2)\!\times\! \Inj_\KK(A_2,A_r) \ ; \  \alpha^\ast(N) \cong cP_j(2)\!\oplus\!dP_{j+1}(2)\}\]
is a dense open subset of $\rep(K_r;V_1,V_2)\!\times\! \Inj_\KK(A_2,A_r)$. Alternatively, $\dim_\KK V_1\!=\!(j\!+\!1)\Delta_{(V_1,V_2)}$ and we arrive at the same conclusion. In view of Proposition \ref{CatRep4}(1),
we conclude that
\[ \tilde{U} := U \cap ((\EKP(K_r)\cap\rep(K_r;V_1,V_2))\!\times\! \Inj_\KK(A_2,A_r))\]
is a dense open subset of $\rep(K_r;V_1,V_2)\!\times\! \Inj_\KK(A_2,A_r)$. 

Since
\[ \iota_M : \Inj_{\KK}(A_2,A_r) \lra \rep(K_r;V_1,V_2)\!\times\!\Inj_{\KK}(A_2,A_r) \ \ ; \ \ \alpha \mapsto (M,\alpha)\]
is continuous,  our current assumption ensures that $\tilde{O}\!:=\! \iota_M^{-1}(\tilde{U})$ is a dense open subset of $\Inj_{\KK}(A_2,A_r)$. As the morphism $\msim : \Inj_\KK(A_2,A_r) \lra \Gr_2(A_r)$ is open 
(see the proof of Proposition \ref{Rest2}), 
\[ \msim(\tilde{O})\!=\!\{ \fv \in \Gr_2(A_r) \ , \ M|_\fv \cong  cP_j(2)\!\oplus\!dP_{j+1}(2)\}\] 
is a dense open subset of $\Gr_2(A_r)$ and Corollary \ref{Rest3} gives
\[ M_{\gen} =  cP_j(2)\!\oplus\!dP_{j+1}(2).\]
As a result, $M \in O_{\gen}$. \end{proof}

\bigskip
\noindent
By way of example, we briefly discuss Schwarzenberger modules. Our result and its succeeding remark will provide the generic splitting type of the associated vector bundles and
show that these are $\GL(A_2)$-homogeneous, but not homogeneous (cf.\ Section \ref{S:HomM} below).  

Given $\ell \in \NN$, we consider the Schwarzenberger module $M_\cS[\ell\!+\!r]$ of dimension vector $\udim M_\cS[\ell\!+\!r]\!=\!(\ell, \ell\!+\!r\!-\!1)$. Then we have $(M_\cS[\ell\!+\!r])_1\!=\!\KK[X_1,X_2]_{\ell-1}$ and 
$(M_\cS[\ell\!+\!r])_2\!=\!\KK[X_1,X_2]_{\ell+r-2}$. Note that these spaces are $\ZZ^2$-graded. Without loss of generality, we may assume that $M_\cS[\ell\!+\!r](\gamma_i)(f)\!=\!X_i^{r-1}f$ for $i \in \{1,2\}$ and $f \in 
\KK[X_1,X_2]_{\ell-1}$. In particular, we have $\deg(M_\cS[\ell\!+\!r](\gamma_1))\!=\!(r\!-\!1,0)$ and $\deg(M_\cS[\ell\!+\!r](\gamma_2))\!=\!(0,r\!-\!1)$. 

\bigskip

\begin{Cor} \label{GCD6} If $\ell\!-\!1\!=\!a_\ell(r\!-\!1)\!+\!q_\ell$ for $q_\ell \in \{0,\ldots, r\!-\!2\}$, then 
\[ M_\cS[\ell\!+\!r]_{\gen} = (r\!-\!2\!-\!q_\ell)P_{a_\ell}(2)\!\oplus\!(q_\ell\!+\!1)P_{a_\ell+1}(2).\] 
In particular, $M_{\cS}[\ell\!+\!r] \in O_{\gen}$. \end{Cor} 

\begin{proof} Let $\fv_0\!:=\!\KK\gamma_1\!\oplus\!\KK\gamma_2$. Writing $M\!:=\!M_\cS[\ell\!+\!r]$, we shall first find the decomposition of $M|_{\fv_0}$.

Given $(a,b) \in \NN_0^2$, we write $X^{(a,b)}\!:=\!X_1^aX_2^b$. For $q \in \{0,\ldots, r\!-\!2\}$ we consider the vector spaces
\[ M_{1,q} := \langle\{ X^{(a,b)} \in \KK[X_1,X_2]_{\ell-1} \ ; \ a\equiv q \ \modd(r\!-\!1)\} \rangle\] 
as well as 
\[M_{2,q} := \langle\{ X^{(a,b)} \in \KK[X_1,X_2]_{\ell+r-2} \ ; \ a\equiv q \ \modd(r\!-\!1)\}\rangle.\]
Then we have $M(\gamma_i)(M_{1,q}) \subseteq M_{2,q}$ for $i \in \{1,2\}$. Defining $M_q\!:=\!(M_{1,q},M_{2,q})$, we obtain a decomposition
\[ (\ast) \ \ \ \ \ \ \ \ M|_{\fv_0} = \bigoplus_{q=0}^{r-2} M_q\]
of $\KK.\fv_0$-modules. Moreover, setting $t_q\!:=\!|\{ n \in \NN_0 \ ; \ q\!+\!n(r\!-\!1)\!\le\!\ell\!-\!1\}|$, we have $\udim M_q\!=\!(t_q,t_q\!+\!1)$, so that $M_q \in \EKP(K_2)$ is of minimal type. Since $\EKP(K_2)$ is 
closed under taking direct summands and $\Delta_X\!>\!0$ for all $X \in \EKP(K_2)\!\smallsetminus\!\{(0)\}$, this renders $M_q$ indecomposable. We conclude $M_q\!\cong\!P_{t_q}(2)$, so that ($\ast$) actually is the 
decomposition of the $\KK K_2$-module $M|_{\fv_0}$ into indecomposables. For $0\!\le\!q\!\le\!q_{\ell}$, we have $t_q\!=\!a_\ell\!+\!1$, while $t_q\!:=\!a_\ell$ for $q_\ell\!+\!1\!\le\!q\!\le\!r\!-\!2$. As a result, ($\ast$) yields
\[  M|_{\fv_0} \cong (r\!-\!2\!-\!q_\ell)P_{a_\ell}(2)\!\oplus\!(q_\ell\!+\!1)P_{a_\ell+1}(2).\]
Corollary \ref{GCD5} implies 
\[ (M_{\cS}[\ell\!+\!r])_{\gen} = (r\!-\!2\!-\!q_\ell)P_{a_\ell}(2)\!\oplus\!(q_\ell\!+\!1)P_{a_\ell+1}(2),\]
and $M_{\cS}[\ell\!+\!r] \in O_{\gen}$. \end{proof}

\bigskip

\begin{Remark} Let $M\!:=\!M_\cS[\ell\!+\!r]$ and consider
\[ \fv_1 := \KK\gamma_1\!\oplus\!\KK\gamma_3,\]
where
\[ M(\gamma_1)(X_1^aX_2^b) = X_1^{a+r-1}X_2^b  \ \ ; \ \  M(\gamma_3)(X_1^aX_2^b) = X_1^{a+r-2}X_2^{b+1}  \ \ \ \ (a\!+\!b\!=\!\ell\!-\!1).\]
Then $(\id_{\KK [X_1,X_2]_{\ell-1}}, X_1^{r-2}\cdot)$ defines an embedding $P_\ell(2) \hookrightarrow M|_{\fv_1}$ such that
\[ M|_{\fv_1} \cong (r\!-\!2)P_0(2)\!\oplus\!P_{\ell}(2).\]
Thus, for $\ell\!\ge\!2$, the dense open subset $O_M$ defined in Corollary \ref{Rest3} is properly contained in $\Gr_2(A_r)$. \end{Remark}

\bigskip

\section{Coherent sheaves on $\Gr_d(A_r)$ and representations of $K_r$} \label{S:Coh}
This section is concerned with the equivalence of the categories $\repp(K_r,d)$ of relative projective representations and Steiner bundles on the Grassmannians $\Gr_d(A_r)$ along with some consequences thereof. 
Our account builds on work by Jardim-Prata \cite{JP15}, where vector bundles over fields of characteristic $0$ were considered. For fields of characteristic $p\!>\!0$, Friedlander-Pevtsova \cite{FPe11} earlier employed 
universal nilpotent operators to define functors between module categories of infinitesimal group schemes and vector bundles on their support spaces. In the context of elementary abelian $p$-groups (or equivalently 
elementary abelian restricted Lie algebras), the two approaches are related via the functorial correspondence between Kronecker representations and modules of Loewy length $\le\!2$.  

\bigskip

\subsection{Conventions} \label{S:Co} Let $r\!\ge\!2$, $d \in \{1,\ldots, r\!-\!1\}$. We denote by $\Coh(\Gr_d(A_r))$ and $\Vect(\Gr_d(A_r))$ the categories of coherent sheaves and vector bundles (locally free coherent 
sheaves) on $\Gr_d(A_r)$, respectively. Note that $\Coh(\Gr_d(A_r))$ is an abelian category, cf.\ \cite[(7.4)]{GW}. Let $\cO_{\Gr_d(A_r)}$ be the structure sheaf of $\Gr_d(A_r)$, so that every $\cF \in \Coh(\Gr_d(A_r))$ is 
an $\cO_{\Gr_d(A_r)}$-module. In view of \cite{At56}, the category $\Vect(\Gr_d(A_r))$ is a Krull-Schmidt category. Thus, the indecomposable objects of any full subcategory $\cC \subseteq \Vect(\Gr_d(A_r))$ that is closed 
under taking direct summands are just the indecomposable vector bundles belonging to $\cC$. 

Given $\fv \in \Gr_d(A_r)$, we let $\cO_{\Gr_d(A_r),\fv}$ be the local ring of $\Gr_d(A_r)$ at $\fv$, whose maximal ideal we denote by $\fm_\fv$. Let $\cF \in \Coh(\Gr_d(A_r))$. For $\fv \in \Gr_d(A_r)$, the {\it stalk} $
\cF_\fv$ of $\cF$ at $\fv$ is an $\cO_{\Gr_d(A_r),\fv}$-module, and the finite-dimensional $\KK$-vector space $\cF(\fv)\!=\!\cF_\fv/\fm_\fv\cF_\fv$ is called the {\it fiber} of $\cF$ at $\fv$. If $\cF \in \Vect(\Gr_d(A_r))$, then $
\cF_\fv$ is a free $\cO_{\Gr_d(A_r),\fv}$-module of rank $\rk(\cF_\fv)\!=\!\dim_\KK\cF(\fv)$.

If $\cF,\cG \in \Coh(\Gr_d(A_r))$ are sheaves, we let $\msHom_{\Gr_d(A_r)}(\cF,\cG)$ be the sheaf of $\cO_{\Gr_d(A_r)}$-homomor- phisms from $\cF$ to $\cG$. By definition, we have $\msHom_{\Gr_d(A_r)}(\cF,\cG)
(\Gr_d(A_r))\!=\!\Hom_{\Gr_d(A_r)}(\cF,\cG)$, the space of homomorphisms from $\cF$ to $\cG$, cf.\ \cite[\S 7.4]{GW}. For $i \in \NN_0$, we let 
\[ \Ext^i_{\Gr_d(A_r)}(\cF,-) : \Coh(\Gr_d(A_r)) \lra \modd \KK\] 
be the $i$-th right derived functor of $\Hom_{\Gr_d(A_r)}(\cF,-) \cong \Ext^0_{\Gr_d(A_r)}(\cF,-)$. Setting 
\[ \HH^i(\Gr_d(A_r),\cF) := \Ext^i_{\Gr_d(A_r)}(\cO_{\Gr_d(A_r)},\cF),\] 
we recall that $\Ext^i_{\Gr_d(A_r)}(\cF,\cG) \cong \HH^i(\Gr_d(A_r),\msHom_{\Gr_d(A_r)}(\cF,\cG))$ for every $\cG \in \Coh(\Gr_d(A_r))$. 

We denote by
\[ \chi(\cF) = \sum_{i=0}^{d(r-d)}(-1)^i \dim_\KK\HH^i(\Gr_d(A_r),\cF)\]
the {\it Euler characteristic} of $\cF$ (see \cite[(III.2.7)]{Ha77}). We refer the reader to \cite[(III.6)]{Ha77} for more details. 

The following subsidiary result for ``special exceptional pairs`` of coherent sheaves on an arbitrary variety $X$ is inspired by \cite[(3.1)]{Bra05}. We begin by recalling the relevant terminology:

\bigskip
\noindent
A vector bundle $\cF \in \Vect(X)$ is referred to as {\it exceptional} if $\dim_\KK \Ext^i_X(\cF,\cF)\!=\!\delta_{i,0}$ for all $i\!\ge\!0$. We say that a pair $(E_0,E_1)$ of coherent sheaves on $X$ is {\it special
exceptional}, provided 
\begin{enumerate}
\item[(a)] $\dim_\KK\Ext^n_X(E_i,E_i)\!=\!\delta_{n,0}$ for all $n\!\ge\!0$, $i \in \{0,1\}$.
\item[(b)] $\dim_\KK\Ext^n_X(E_0,E_1)\!=\!0$ for all $n\!\ge\!1$.
\item[(c)] $\dim_\KK\Ext^n_X(E_1,E_0)\!=\!0$ for all $n\!\ge\!0$.\footnote{Our choice of terminology derives from Rudakov's more general notion of an exceptional pair, cf.\ \cite{Ru90}. Pairs satisfying the above conditions are 
referred to as ``strongly exceptional" in \cite{AM15}, which differs from the definition employed in \cite{BLV15}.} \end{enumerate}
We put $r\!:=\!\dim_\KK\Hom_X(E_0,E_1)$. Suppose there are exact sequences
\[ (\ast) \ \ \ \ (0) \lra E_0^m \lra E_1^n \lra \cF \lra (0)\]
and 
\[ (\ast\ast) \ \ \ (0) \lra E_0^s \lra E_1^t \lra \cG \lra (0).\]

\bigskip
\noindent
Recall the from Section \ref{S:reps} the definition of the Euler-Ringel bilinear form:
\[ \langle(x_1,y_1),(y_1,y_2)\rangle_r = x_1y_1\!+\!x_2y_2\!-\!rx_1y_2.\]

\begin{Prop} \label{Co1} We have
\[\chi(\msHom_X(\cF,\cG))=\!\langle (m,n),(s,t)\rangle_r.\]
\end{Prop}

\begin{proof} We proceed in several steps:
\medskip

(i) {\it We have an exact sequence}
\[ (0) \lra \Hom_X(\cF,E_1^t) \lra \Hom_X(\cF,\cG) \lra \Ext^1_X(\cF,E_0^s) \lra \Ext^1_X(\cF,E_1^t) \lra \Ext^1_X(\cF,\cG) \lra (0).\]

\smallskip
\noindent 
The assertion follows by applying $\Hom_X(\cF, -)$ to ($\ast\ast$), once we know that $\Hom_X(\cF,E_0^s)\!=\!(0)\!=\!\Ext^2_X(\cF,E_0^s)$.

By applying $\Hom_X(-,E_0^s)$ to ($\ast$) while observing (c), we obtain a sequence $(0)\!\lra\!\Hom_X(\cF,E_0^s)$ $\lra\!\Hom_X(E_1^n,E_0^s)\!=\!(0)$, so that $\Hom_X(\cF,E_0^s)\!=\!(0)$. To show exactness at 
$\Ext^1_X(\cF,\cG)$, we apply $\Ext^1_X(-,E_0^s)$ to $(\ast)$. Then (a) and (c) yield
\[ (0) = \Ext^1_X(E_0^m,E_0^s) \lra \Ext^2_X(\cF,E_0^s) \lra \Ext^2_X(E_1^n,E_0^s) = (0),\]
whence $\Ext^2_X(\cF,E_0^s)\!=\!(0)$.  \hfill $\diamond$ 

\medskip

(ii) {\it We have} $\dim_\KK\Hom_X(\cF,E_1^t)\!-\!\dim_\KK\Ext^1_X(\cF,E_1^t)\!=\!nt\!-\!rmt$. 

\smallskip
\noindent
We apply $\Hom_X(-,E_1^t)$ to ($\ast$) and obtain
\[ (0) \lra \Hom_X(\cF,E_1^t) \lra \Hom_X(E_1^n,E_1^t) \lra \Hom_X(E_0^m,E_1^t) \lra \Ext^1_X(\cF,E_1^t) \lra \Ext^1_X(E_1^n,E_1^t)\!=\!(0).\] 
As $\dim_\KK\Hom_X(E_1^n,E_1^t)\!=\!nt$, while $\dim_\KK\Hom_X(E_0^m,E_1^t)\!=\!rmt$, our assertion follows. \hfill $\diamond$

\medskip

(iii) {\it We have} $\Ext^\ell_X(\cF,\cG)\!=\!(0)$ for $\ell\!\ge\!2$.

\smallskip
\noindent
Let $\ell\!\ge\!2$. Application of $\Ext^{\ell-1}_X(-,E_0^s)$ to $(\ast)$ in conjunction with (a) and (c) yields
\[(0) = \Ext^{\ell-1}_X(E_0^m,E_0^s) \lra \Ext^\ell_X(\cF,E_0^s) \lra \Ext^\ell_X(E_1^n,E_0^s) = (0),\]
so that $\Ext^\ell_X(\cF,E_0^s)\!=\!(0)$. In view of (b) and (a), application of $\Ext^{\ell-1}_X(-,E_1^t)$ to ($\ast$) gives
\[(0) = \Ext^{\ell-1}_X(E_0^m,E_1^t) \lra \Ext^\ell_X(\cF,E_1^t) \lra \Ext^\ell_X(E_1^n,E_1^t) = (0),\]
so that $\Ext^\ell_X(\cF,E_1^t)\!=\!(0)$.

Applying $\Ext^\ell_X(\cF,-)$ to ($\ast\ast$), we finally arrive at
\[ (0) = \Ext^\ell_X(\cF,E_1^t) \lra \Ext^\ell_X(\cF,\cG) \lra \Ext^{\ell+1}_X(\cF,E_0^s)=(0),\]
so that $\Ext^\ell_X(\cF,\cG)\!=\!(0)$. \hfill $\diamond$

\medskip
\noindent
By applying (iii), (i) and (ii) consecutively, we obtain
\begin{eqnarray*}
\chi(\msHom_X(\cF,\cG)) & = & \dim_\KK\Hom_X(\cF,\cG)\!-\!\dim_\KK\Ext^1_X(\cF,\cG)\\  
                                   & = & \dim_\KK\Hom_X(\cF,E_1^t)\!-\!\dim_\KK\Ext^1_X(\cF,E_1^t)\!+\!\dim_\KK\Ext^1_X(\cF,E_0^s)\\ 
                                   & = & nt\!-\!rmt\!+\!\!\dim_\KK\Ext^1_X(\cF,E_0^s).
\end{eqnarray*}
Application of $\Hom_X(-,E_0^s)$ to ($\ast$) gives
\[ (0)= \Hom_X(E_1^n,E_0^s) \lra \Hom_X(E_0^m,E_0^s) \lra \Ext^1_X(\cF,E_0^s) \lra \Ext^1_X(E_1^n,E_0^s)=(0),\]
whence
\[ \chi(\msHom_X(\cF,\cG)) = nt\!-\!rmt\!+\!ms = \langle (m,n),(s,t)\rangle_r,\]
as desired. \end{proof}

\bigskip

\subsection{The functor $\TilTheta_d$}\label{S:Fun}
For $d \in \{1,\ldots,r\!-\!1\}$ we consider the {\it universal vector bundle} $\cU_{(r,d)}$ of $\Gr_d(A_r)$. By definition, $\cU_{(r,d)}$ is the locally free sheaf corresponding to the locally trivial vector space fibration $(E,p)$ 
over $\Gr_d(A_r)$, where
\[ E\!:=\!\{(\fv,a) \in \Gr_d(A_r)\!\times\!A_r \ ; \ a \in \fv\} \ \ \text{and}  \ \  p : E \lra \Gr_d(A_r) \ ; \ (\fv,a) \mapsto \fv\]
denote the incidence variety and the canonical projection, respectively. Consequently,
\[ \cU_{(r,d)}(U) = \{s : U \lra E \ ; \ p\circ s\!=\!\id_U\} \cong \{ t : U \lra A_r \ ; \ t(\fv) \in \fv \ \ \ \forall \ \fv \in U\}\]
for every open subset $U \subseteq \Gr_d(A_r)$. The vector bundle $\cU_{(r,d)}$ is known to be simple and the canonical map $\cU_{(r,d)} \lra \cO^r_{\Gr_d(A_r)}$ is locally split injective. Hence its cokernel $\cQ_{(r,d)}$ 
is a vector bundle (cf.\ \cite[(6.2.1)]{Be17}), referred to as the {\it universal quotient bundle}, cf.\ \cite[(3.2.3)]{EH}.

Given a pair $(V_1,V_2)$ of $\KK$-vector spaces, we put
\[ \widetilde{V_1} := V_1\!\otimes_\KK\!\cU_{(r,d)} \ \ \text{and} \ \ \widetilde{V_2} := V_2\!\otimes_\KK\!\cO_{\Gr_d(A_r)}.\]
Let $f_i : V_i \lra W_i$ ($i \in \{1,2\}$) be linear maps between $\KK$-vector spaces. These give rise to morphisms $\tilde{f_i} : \widetilde{V_i} \lra \widetilde{W_i}$, where 
\[ \tilde{f_1} := f_1\!\otimes\id_{\cU_{(r,d)}} \ \ \ \text{and} \ \ \ \tilde{f_2} := f_2\otimes\id_{\cO_{\Gr_d(A_r)}}.\]
For $\fv \in \Gr_d(A_r)$, the standard isomorphisms $\ev_{1,\fv} : \cU_{(r,d)}(\fv) \lra \fv$ and $\ev_{2,\fv} : \cO_{\Gr_d(A_r)}(\fv) \lra \KK$ induce commutative diagrams
\begin{equation}\label{diagram1} \begin{tikzcd}   \widetilde{V_i}(\fv) \arrow[d, "\id_{V_i}\otimes \ev_{i,\fv}"] \arrow[r,"\tilde{f_i}(\fv)"] & \widetilde{W_i}(\fv) \arrow[d,"\id_{W_i}\otimes \ev_{i,\fv}"]\\
                            V_i\!\otimes_\KK\!\fu_i  \arrow[r,"f_i\otimes\id_{\fu_i}"] & W_i\!\otimes_\KK\!\fu_i,
\end{tikzcd} \end{equation} 
where 
\[ \fu_i := \left\{\begin{array}{cc} \fv & i\!=\!1 \\ \KK & i\!=\!2.\end{array} \right.\]
Similarly, the isomorphism $\cO_{\Gr_d(A_r)}(\Gr_d(A_r)) \cong \KK$ allows us to identify $(\tilde{f_2})_{\Gr_d(A_r)}$ and $f_2$.  

Let $M \in \rep(K_r)$. In analogy with \cite[\S 8.2]{Be17} (see also \cite{FPe11}), we consider for every open subset $U \subseteq \Gr_d(A_r)$, the morphism
\[ \TilTheta_{M,d}(U) : \widetilde{M_1}(U) \lra \widetilde{M_2}(U) \ \ ; \ \ m\otimes t \mapsto \sum_{i=1}^r \gamma_i\dact m\otimes \gamma_i^\ast\circ t\]
of $\cO_{\Gr_d(A_r)}(U)$-modules.\footnote{Here $\{\gamma_1^\ast,\ldots, \gamma_r^\ast\} \subseteq A_r^\ast$ denotes the dual basis of $\{\gamma_1,\ldots, \gamma_r\}.$} Then $\TilTheta_{M,d} : \widetilde{M_1} \lra 
\widetilde{M_2}$ is a morphism of vector bundles, and we define
\[ \TilTheta_d(M) := \msCoker\TilTheta_{M,d} \in \Coh(\Gr_d(A_r)).\]
The definition of $\TilTheta_d(M)$ does not depend on the choice of the dual bases $\{\gamma_1,\ldots, \gamma_r\} \subseteq A_r$, $\{\gamma^\ast_1,\ldots, \gamma^\ast_r\} \subseteq A^\ast_r$. In fact, one can define 
sheaves $\widetilde{M}'_1$ and $\widetilde{M}'_2$ via
\[ \widetilde{M}'_1(U) := \{ s : U \lra A_r\!\otimes_\KK\!M_1 \ ; \ s(\fv) \subseteq \fv\!\otimes_\KK\!M_1 \ \ \  \ \forall \ \fv \in U\} \subseteq \Mor(U,A_r\!\otimes_\KK\!M_1)\]
and 
\[ \widetilde{M}'_2(U) := \Mor(U,M_2).\]
Setting
\[ \TilTheta'_{M,d}(U)(s) = \psi_M\circ s  \ \ \ \ \forall \ s \in \widetilde{M}'_1(U),\]
we obtain
\[ \TilTheta_d(M) \cong \coker\TilTheta'_{M,d}.\]
If $f : M \lra N$ is a morphism in $\rep(K_r)$ with components $f_i: M_i \lra N_i \ \ \ (i \in \{1,2\})$, then 
\[ \tilde{f}_i : \widetilde{M_i} \! \lra \widetilde{N_i} \ \ ; \ \ m\otimes g \mapsto f_i(m)\otimes g \  \  \  \  \  \  \   (i \in \{1,2\})\]
are morphisms of vector bundles and there results a commutative diagram
\[ \begin{tikzcd}   \widetilde{M_1} \arrow[d, "\tilde{f_1}"] \arrow[r,"\TilTheta_{M,d}"] & \widetilde{M_2} \arrow[d,"\tilde{f_2}"] \arrow[r] & \TilTheta_d(M) \arrow[r] & (0)\\
                            \widetilde{N_1}  \arrow[r,"\TilTheta_{N,d}"] & \widetilde{N_2} \arrow[r] & \TilTheta_d(N) \arrow[r] & (0).
\end{tikzcd} \] 
Consequently, there is a unique morphism 
\[ \TilTheta_d(f) : \TilTheta_d(M) \lra \TilTheta_d(N),\]
which completes the diagram above. We thus obtain a functor
\[ \TilTheta _d: \rep(K_r) \lra \Coh(\Gr_d(A_r)).\]
We denote by $\reppi(K_r,d)$ the full subcategory of $\rep(K_r)$, whose objects $M$ have the property that $\psi_{M,\fv}\!=\!\psi_M|_{\fv\otimes_\KK M_1}$ is surjective for every $\fv \in \Gr_d(A_r)$. (This category 
coincides with the essential image of $\repp(K_r,d)$ under the standard duality.) 

\bigskip

\begin{Definition} Let $M \in \rep(K_r)$. We say that $M$ has {\it constant $d$-rank}, provided there is $\rk_d(M) \in \NN_0$ such that $\rk(\psi_{M,\fv})\!=\!\rk_d(M)$ for all $\fv \in \Gr_d(A_r)$.  
We let $\CR(K_r,d)$ be the full subcategory of $\rep(K_r)$ whose objects have constant $d$-rank. \end{Definition}

\bigskip

\begin{Remark} The category $\CR(K_r,d)$ is the analog of the category of modules of constant $(d,1)$-radical rank that was considered in \cite[(4.1)]{CFP15}. \end{Remark}  

\bigskip
\noindent
We record a few basic properties:

\bigskip

\begin{Lem} \label{Fun1} The following statements hold:
\begin{enumerate} 
\item The functor $\TilTheta_d$ is right exact. 
\item Let $M \in \rep(K_r)$. Then $M \in \CR(K_r,d)$ if and only if $\TilTheta_d(M) \in \Vect(\Gr_d(A_r))$. In that case, we have $\rk(\TilTheta_d(M))\!=\!\dim_\KK M_2\!-\!\rk_d(M)$.
\item We have $\ker\TilTheta_d\cap\CR(K_r,d)\!=\!\reppi(K_r,d)$. \end{enumerate} \end{Lem}

\begin{proof} (1) The right exactness of $\TilTheta_d$ is a direct consequence of (the proof of) the Snake Lemma. 

(2) Suppose that $M \in \CR(K_r,d)$. Given $\fv \in \Gr_d(A_r)$, the commutative diagram
\[ \begin{tikzcd}   \widetilde{M_1}(\fv) \arrow[d, "\id_{M_1}\otimes \ev_{1,\fv}"] \arrow[r,"\TilTheta_{M,d}(\fv)"] & \widetilde{M_2}(\fv) \arrow[d,"\id_{M_2}\otimes \ev_{2,\fv}"]\\
                            M_1\!\otimes_\KK\!\fv \arrow[r,"\psi_{M,\fv}\circ \omega"] & M_2,
\end{tikzcd} \] 
where $\omega$ flips the tensor factors, yields
\[\rk(\TilTheta_{M,d}(\fv)) = \rk(\psi_{M,\fv}) = \rk_d(M).\] 
It now follows from \cite[(6.2.1)]{Be17} that $\TilTheta_d(M)$ is a vector bundle of rank $\rk(\TilTheta_d(M))\!=\!\dim_\KK M_2\!-\!\rk_d(M)$. 

Conversely, assume that $\TilTheta_d(M)$ is a vector bundle. Then \cite[(6.2.1)]{Be17} implies that there is $n \in \NN_0$ such that
\[ \rk(\TilTheta_{M,d}(\fv)) = n  \ \ \ \ \ \forall \ \fv \in \Gr_d(A_r).\]
The observation above thus implies that $\rk(\psi_{M,\fv})\!=\!n$ for all $\fv \in \Gr_d(A_r)$, so that $M \in \CR(K_r,d)$. 

(3) Let $M \in \CR(K_r,d)$. By definition, we have $M \in \reppi(K_r,d)$ if and only if $\rk_d(M)\!=\!\dim_\KK M_2$. In view of (2), this is equivalent to $\TilTheta_d(M)\!=\!(0)$. \end{proof}

\bigskip
\noindent
Steiner bundles on projective space were first systematically studied by Dolgachev-Kapranov \cite{DK93}. The following definition for Grassmannians is taken from \cite{AM15}.

\bigskip

\begin{Definition} Let $(s,t) \in \NN_0^2$. A vector bundle $\cF \in \Vect(\Gr_d(A_r))$ is referred to as an {\it $(s,t)$-Steiner bundle} if there exists an exact sequence
\[ (0) \lra \cU^s_{(r,d)} \lra \cO^t_{\Gr_d(A_r)} \lra \cF \lra (0).\] 
We denote by $\StVect(\Gr_d(A_r))$ the full subcategory of $\Vect(\Gr_d(A_r))$, whose objects are Steiner bundles (for some $(s,t) \in \NN_0^2$). \end{Definition}

\bigskip
\noindent
In view of $\cU_{(r,1)}\cong \cO_{\PP(A_r)}(-1)$ (cf.\ \cite[(3.2.3)]{EH}), we retrieve the original definition of \cite{DK93}. 

The relationship between $\repp(K_r)$ and $\StVect(\Gr_d(A_r))$, which for $d\!=\!1$ is implicit in Brambilla's work \cite{Bra05,Bra08} on Steiner bundles (for $\KK\!=\!\CC$), was investigated by Jardim and 
Prata \cite{JP15} for $\Char(\KK)\!=\!0$ in the more general context of cokernel bundles. In our context the category of ``globally injective representations`` defined in \cite{JP15} coincides with $\repp(K_r,d)$.

Since we will employ the functors $\TilTheta_d$ extensively, we recall the basic arguments of the proof of \cite[(3.5)]{JP15}, thereby hopefully convincing the reader that the result does 
not necessitate any assumption on the characteristic of the base field. To that end we require the following subsidiary result.

\bigskip

\begin{Lem} \label{Fun2} The pair $(\cU_{(r,d)},\cO_{\Gr_d(A_r)})$ is special exceptional. \end{Lem} 

\begin{proof} Given a partition $\alpha$, we let $|\alpha|\!=\!\sum_i\alpha_i$ be its degree, and $\alpha'$ be its transpose partition. Let $\cB_{r-d,d}$ be the set of those partitions, whose corresponding Young tableaux 
have at most $r\!-\!d$ rows and at most $d$ columns. We pick a total ordering $\prec$ on $\cB_{r-d,d}$ such that $|\alpha| < |\beta|$ implies $\alpha \prec \beta$.  

According to \cite[(5.7)]{BLV15}, the vector bundle 
\[ \cT := \bigoplus_{\alpha \in \cB_{r-d,d}}\bigwedge^{\alpha'}(\cU_{(r,d)})\]
is a tilting object in the bounded derived category $\msD^b(\Gr_d(A_r))$ of $\Coh(\Gr_d(A_r))$. In particular,  $\Ext^i_{\Gr_d(A_r)}(\cT,\cT)\!=\!(0)$ for all $i\!>\!0$. We consider the partitions $\alpha\!:=\!0$ and 
$\beta\!:=\!1$. Then $\alpha\!\prec\!\beta \in \cB_{r-d,d}$, so that  $\cO_{\Gr_d(A_r)}\!=\!\bigwedge^\alpha(\cU_{(r,d)})$ and $\cU_{(r,d)}\!=\!\bigwedge^\beta(\cU_{(r,d)})$ are direct summands of $\cT$. Consequently, 
$\Ext^i_{\Gr_d(A_r)}(\cX,\cY)\!=\!(0)$ for $\cX,\cY \in \{\cU_{(r,d)},\cO_{\Gr_d(A_r)}\}$ and $i\!>\!0$. Since both bundles are simple, our assertion follows from the fact that $\Hom_{\Gr_d(A_r)}(\cO_{\Gr_d(A_r)}, \cU_{(r,d)})\!
=\!\cU_{(r,d)}(\Gr_d(A_r))\!=\!(0)$, cf.\ \cite[(I.5.2)]{Sh}. \end{proof} 

\bigskip

\begin{Thm} \label{Fun3}  The following statements hold:
\begin{enumerate}
\item $\cF \in \Vect(\Gr_d(A_r))$ is a Steiner bundle if and only if there is $M \in \repp(K_r,d)$ such that $\cF \cong \TilTheta_d(M)$. 
\item The functor $\TilTheta_d : \repp(K_r,d) \lra \StVect(\Gr_d(A_r))$ is an equivalence.
\item Let $M \in \repp(K_r,d)$. Then we have $\rk(\TilTheta_d(M))\!=\!\Delta_M(d)$. \end{enumerate} \end{Thm}

\begin{proof} Let $M \in \CR(K_r,d)$. The proof of Lemma \ref{Fun1} yields  $\ker \TilTheta_{M,d}(\fv) \cong \ker \psi_{M,\fv}$ for every $\fv \in \Gr_d(A_r)$.

(1) Suppose that $M \in \repp(K_r,d)$. By virtue of Theorem \ref{Fam5} and the above we have 
\[ \ker \TilTheta_{M,d}(\fv) \cong \ker \psi_{M,\fv}\!=\!(0)\] 
for all $\fv \in \Gr_d(A_r)$. According to \cite[(6.2.1)(c)]{Be17}, the morphism $\TilTheta_{M,d}$ is locally split injective, whence $\msKer\TilTheta_{M,d}\!=\!(0)$. Consequently, $\TilTheta_d(M)$ is a Steiner bundle. 

Conversely, suppose that $\cF$ is a Steiner bundle. Then there exists a pair $(M_1,M_2)$ of vector spaces and an exact sequence
\[ (0) \lra \widetilde{M_1} \stackrel{\Psi}{\lra} \widetilde{M_2} \lra \cF \lra (0)\]
of vector bundles. It follows from \cite[(1.1.12)]{Ma12}, which holds for fields of arbitrary characteristic, that 
\[ \dim_\KK\Hom_{\Gr_d(A_r)}(\cU_{(r,d)},\cO_{\Gr_d(A_r)}) = \dim_\KK\HH^0(\cU_{(r,d)}^\vee) = r.\]
Hence the injective canonical map
\[ \Hom_\KK(M_1,M_2)\!\otimes_\KK\!A_r^\ast \lra \Hom_{\Gr_d(A_r)}(\widetilde{M_1},\widetilde{M_2})\]
sending $\zeta\otimes f$ to the map $m\otimes t \mapsto \zeta(m)\otimes f\circ t$ is also surjective and thus an isomorphism. There thus exist $\KK$-linear maps $M(\gamma_i) : M_1 \lra M_2$ such that
\[ \Psi(m\otimes t) = \sum_{i=1}^rM(\gamma_i)(m)\otimes \gamma_i^\ast\circ t \ \ \ \ \ \ \forall \ m \in M_1, t \in \cU_{r,d}.\]
Letting $M\!:=\!(M_1,M_2,(M(\gamma_i))_{1\le i\le r}) \in \rep(K_r)$, we see that $\cF\!\cong\!\TilTheta_d(M)$, while $\msKer\TilTheta_{M,d}\!=\!(0)$. Consequently, $M \in \repp(K_r,d)$, as $\ker \psi_{M,\fv} \cong 
(\msKer\TilTheta_{M,d})(\fv)\! =\! (0)$ for every $\fv \in \Gr_d(A_r)$. 

(2) Lemma \ref{Fun2} ensures that $\Ext^1_{\Gr_d(A_r)}(\cO_{\Gr_d(A_r)},\cU_{(r,d)})\!=\!(0)$ and we may now proceed as in \cite[(3.5)]{JP15}, observing that the assumption $\Char(\KK)\!=\!0$ is not needed. 

(3) This follows directly from Lemma \ref{Fun1}. \end{proof}

\bigskip

\begin{Remark} For $d\!=\!1$, the bundles $\TilTheta_1(M_{\cS}[m]), \ (m\!\ge\!r\!+\!1)$ are those introduced by Schwarzenberger, cf.\ \cite[Prop.2]{Sc61b}. Using Corollary \ref{GCD6} one obtains an explicit formula for 
the generic splitting types of these bundles. \end{Remark}  

\bigskip 

\begin{Cor} \label{Fun4} Let $(0) \lra A \stackrel{f}{\lra} B \stackrel{g}{\lra} C \lra (0)$ be an exact sequence in $\rep(K_r)$ such that $C \in \repp(K_r,d)$. Then the sequence $(0) \lra \TilTheta_d(A) \lra \TilTheta_d(B) \lra 
\TilTheta_d(C) \lra (0)$ is exact. \end{Cor}

\begin{proof} We have exact sequences $(0) \lra A_i \stackrel{f_i}{\lra} B_i \stackrel{g_i}{\lra} C_i \lra (0)$ for $i \in \{1,2\}$. Tensoring with $\cU_{(r,d)}$ and $\cO_{\Gr_d(A_r)}$  yields a commutative diagram
\[  \begin{tikzcd} 
(0) \arrow[r] & \widetilde{A_1} \arrow[r] \arrow[d,"\TilTheta_{A,d}"] & \widetilde{B_1} \arrow[r] \arrow[d,"\TilTheta_{B,d}"] & \widetilde{C_1} \arrow[d,"\TilTheta_{C,d}"] \arrow[r] & (0)\\
(0) \arrow[r] & \widetilde{A_2} \arrow[r] & \widetilde{B_2} \arrow[r] & \widetilde{C_2} \arrow[r] & (0)
\end{tikzcd} \]
with exact rows. The arguments of the proof of Theorem \ref{Fun3} show that 
\[ 0 = \dim_\KK \ker \psi_{C,\fv}  = \dim_\KK\ker(\TilTheta_{C,d}(\fv))\] 
for all $\fv \in \Gr_d(A_r)$. By virtue of \cite[(7.3.1)]{GW}, we obtain $\msKer \TilTheta_{C,d}\!=\!(0)$. The assertion now follows by applying the Snake Lemma. \end{proof} 

\bigskip

\begin{Cor} \label{Fun5} For $M,N \in \rep(K_r,d)$ the following statements hold:
\begin{enumerate}
\item We have
\[ \chi(\msHom_{\Gr_d(A_r)}(\TilTheta_d(M),\TilTheta_d(N))) = \langle\udim M, \udim N\rangle_r.\]
\item We have
\[ \dim_\KK\Ext^1_{K_r}(M,N) = \dim_\KK\Ext^1_{\Gr_d(A_r)}(\TilTheta_d(M),\TilTheta_d(N)).\] \end{enumerate} \end{Cor} 

\begin{proof} (1) This is a direct consequence of Theorem \ref{Fun3}, Proposition \ref{Co1} and Lemma \ref{Fun2}.

(2) Proposition \ref{Co1} and its proof yield
\begin{eqnarray*} 
\dim_\KK\Ext^1_{\Gr_d(A_r)}(\TilTheta_d(M),\TilTheta_d(N)) &= & \dim_\KK \Hom_{\Gr_d(A_r)}(\TilTheta_d(M),\TilTheta_d(N))\\
                                                                                                           &- & \chi(\msHom_{\Gr_d(A_r)}(\TilTheta_d(M),\TilTheta_d(N)))\\ 
                                                                                                           & = & \dim_\KK \Hom_{K_r}(M,N)\!-\!\langle\udim M,\udim N\rangle_r\\
                                                                                                           & = &\dim_\KK\Ext^1_{K_r}(M,N). \ \ \  \ \ \  \qedhere                                                                                                 
\end{eqnarray*}
\end{proof}

\bigskip

\subsection{Direct consequences} \label{S:Dc}
Following \cite{EH}, we begin by collecting a few facts concerning Chow rings of Grassmannians. We emphasize that our statements are valid for the field $\KK$, which has arbitrary characteristic.

For $\ell \in \{0,\ldots, d(r\!-\!d)\}$, we let $Z_\ell(\Gr_d(A_r))$ be the free abelian group with basis given by the $\ell$-dimensional irreducible closed subsets of $\Gr_d(A_r)$. The elements
of $Z_\ell(\Gr_d(A_r))$ are called {\it $\ell$-cycles}, those of
\[ Z(\Gr_d(A_r)) := \bigoplus_{\ell=0}^{d(r\!-\!d)} Z_\ell(\Gr_d(A_r))\]
are referred to as {\it cycles}. 

\bigskip
\noindent 
The {\it Chow group} of $\Gr_d(A_r)$ is the factor group
\[ A(\Gr_d(A_r)) := Z(\Gr_d(A_r))/\Rat(\Gr_d(A_r))\]
given by the subgroup $\Rat(\Gr_d(A_r))$ defined in \cite[(1.2.2)]{EH}. If $C$ is a cycle, we write $[C]$ for its residue class. The class $[\Gr_d(A_r)] \in A(\Gr_d(A_r))$ is called the {\it fundamental class} of $\Gr_d(A_r)$. 

\bigskip
\noindent
Given $\ell \in \{0,\ldots, d(r\!-\!d)\}$, we put
\[ A^\ell(\Gr_d(A_r)) := \{ [C] \ ; \ C \in Z_{d(r-d)-\ell}(\Gr_d(A_r))\}.\]
We have
\[ A(\Gr_d(A_r)) = \bigoplus_{\ell=0}^{d(r-d)} A^\ell(\Gr_d(A_r)).\] 
The group $A(\Gr_d(A_r))$ has the structure of a $\ZZ$-graded, commutative ring, see \cite[(1.2.3)]{EH}. One can show that $A^0(\Gr_d(A_r))\!=\!\ZZ\cdot [\Gr_d(A_r)] \cong \ZZ$.  

\bigskip 

\begin{Definition} The ring $A(\Gr_d(A_r))$ is called the {\it Chow ring} of $\Gr_d(A_r)$. \end{Definition}

\bigskip
\noindent
Let
\[ \cV : \ \ \ \ \ \ (0) \subseteq V_1 \subseteq V_2 \subseteq \cdots \subseteq V_{r-1} \subseteq V_r\!=\!A_r\]
be a complete flag in $A_r$, so that $\dim_\KK V_i\!=\!i$. 

We let $\NN_0^d(r)\!:=\!\{ a \in \NN_0^d \ ; \  r\!-\!d\!\ge\! a_1\! \ge\! \cdots\! \ge\! a_d\! \ge\! 0\}$. For $a \in \NN_0^d(r)$ we define the {\it Schubert variety} $\Sigma_a(\cV)$ via
\[ \Sigma_a(\cV) := \{\fv \in \Gr_d(A_r) \ ; \ \dim_\KK(\fv\cap V_{r-d+i-a_i})\!\ge\!i \ \ \ \forall \ i \in \{1,\ldots, d\}\}.\]
Then $\Sigma_a(\cV)$ is smooth, closed, irreducible and of codimension $|a|\!:=\!\sum_{i=1}^da_i$ in $\Gr_d(A_r)$, cf.\ \cite[(16.4.1)]{EH}. Its class 
\[ \sigma_a := [\Sigma_a(\cV)]\]
is the {\it Schubert class} of $a$. It does not depend on the choice of $\cV$.

Given $\ell \in \{0,\ldots, r\!-\!d\}$, we write
\[ \Sigma_\ell := \Sigma_{(\ell,0,\ldots,0)},\]
so that 
\[\Sigma_\ell(\cV) =  \{ \fv \in \Gr_d(A_r) \ ; \ \fv\cap V_{r-d+1-\ell} \!\ne\!(0)\}.\]
The classes $\sigma_\ell\!:=[\Sigma_\ell(\cV)]$ are referred to as {\it special Schubert classes}.

We have
\[ \sigma_\ell \in A^\ell(\Gr_d(A_r))\]
for all $\ell \in \{0,\ldots,r\!-\!d\}$.

\bigskip

\begin{Thm} \label{Dc1} (cf.\ \cite[(4.7)]{EH}) The Schubert classes $\{\sigma_a \ ; \ a \in \NN_0^d(r)\}$ form a basis of the $\ZZ$-module $A(\Gr_d(A_r))$. \end{Thm}

\bigskip
\noindent
It follows that 
\[\{\sigma_a \ ; \ a \in \NN_0^d(r), \ |a|\!=\!i\}\] 
is a basis of $A^i(\Gr_d(A_r))$ for every $i \in \{0,\ldots,d(r\!-\!d)\}$. 

\bigskip
\noindent
Let $\cF$ be a vector bundle on $\Gr_d(A_r)$. We denote by
\[ c(\cF) = 1\!+\!\sum_{i=1}^{\rk(\cF)} c_i(\cF)t^i \in A(\Gr_d(A_r))[t]\]
the {\it Chern polynomial} of $\cF$. The coefficient $c_i(\cF) \in A^i(\Gr_d(A_r))$ is called the $i$-th {\it Chern class} of $\cF$, cf.\ \cite[(5.2),(5.5)]{EH}. 

The Chern polynomial of the universal quotient bundle $\cQ_{(r,d)}$ on $\Gr_d(A_r)$ is given by
\[ c(\cQ_{(r,d)}) = 1\!+\!\sum_{\ell=1}^{r-d} \sigma_\ell t^{\ell},\]
see \cite[(5.6.2)]{EH}. 

\bigskip

\begin{Cor} \label{Dc2} Let $M \in \repp(K_r,d)$. Then we have
\[ c_1(\TilTheta_d(M)) = (\dim_\KK M_1)\sigma_1.\]
Hence the dimension vector of $M$ can be recovered from $c_1(\TilTheta_d(M))$ and $\rk(\TilTheta_d(M))$ (cf.\ Theorem \ref{Fun3}(3)). 
\end{Cor}

\begin{proof} By definition, there is an exact sequence 
\[ (0) \lra \cU_{(r,d)} \lra A_r\!\otimes_\KK\! \cO_{\Gr_d(A_r)} \lra \cQ_{(r,d)} \lra (0), \]
so that the Whitney sum formula yields $c_1(\cU_{(r,d)})\!=\!-c_1(\cQ_{(r,d)})\!=\!-\sigma_1$. By the same token, the defining sequence
\[ (0) \lra \widetilde{M}_1 \lra \widetilde{M}_2 \lra \TilTheta_d(M) \lra (0)\]
implies $c_1(\TilTheta_d(M))\!=\!-c_1(\widetilde{M}_1)\!=\!-(\dim_\KK M_1)c_1(\cU_{(r,d)})\!=\!(\dim_\KK M_1)\sigma_1$. \end{proof} 

\bigskip
\noindent
The foregoing result allows us to view the first Chern class $c_1(\cF)$ of a Steiner bundle $\cF$ as a number $c_1(\cF) \in \NN_0$. In the sequel, we will freely use this interpretation without further notice.

We next employ the functor $\TilTheta_d$ to translate some results of Section \ref{S:MinType} to the context of Steiner bundles. 

\bigskip

\begin{Prop} \label{Dc3} Let $\cF \in \StVect(\Gr_d(A_r))$ be a Steiner bundle. 
\begin{enumerate} 
\item We have $\rk(\cF)\!\ge\!\min\{c_1(\cF),d\}(r\!-\!d)$. 
\item If $c_1(\cF)\!\le\!d$ or $\rk(\cF)\!<\!d(r\!-\!d)$, then $\cF \cong \cO_{\Gr_d(A_r)}^{\rk(\cF)-(r-d)c_1(\cF)}\!\oplus\!\cQ_{(r,d)}^{c_1(\cF)}$.
\item If $\rk(\cF)\!=\!d(r\!-\!d)$, then $\cF$ is either as in (2), or $\cF$ is simple.
\item If $\rk(\cF)\!=\!d(r\!-\!d)$ and $c_1(\cF)\!\ge\!d\!+\!1$, then $\cF$ is simple.
\item Suppose that $\rk(\cF)\!>\!d(r\!-\!d)$. Then there exists a Steiner bundle $\cG$ such that $\rk(\cG)\!=\!d(r\!-\!d)$ and a short exact sequence
\[ (0) \lra \cO_{\Gr_d(A_r)}^{\rk(\cF)-d(r-d)} \lra \cF \lra \cG \lra (0).\]
\end{enumerate} \end{Prop}

\begin{proof} Theorem \ref{Fun3} provides $M \in \repp(K_r,d)$ such that $\TilTheta_d(M) \cong \cF$ and $\rk(\cF) \cong \Delta_M(d)$. In addition, Corollary \ref{Dc2} gives $\dim_\KK M_1\!=\!c_1(\cF)$.

(1) This is a direct consequence of Theorem \ref{MinType2}.

(2) In view of Corollary \ref{MinType3} and the remark at the beginning of Section \ref{S:CatRep}, the representation $M$ is projective, whence $M \cong \Delta_M(r)P_0(r)\!\oplus\!(\dim_\KK M_1)P_1(r)$.
We have $\Delta_M(r)\!=\!\Delta_M(d)\!-\!(r\!-\!d)\dim_\KK M_1 = \rk(\cF)\!-\!(r\!-\!d)c_1(\cF)$. Moreover, $\TilTheta_d(P_0(r))\cong \cO_{\Gr_d(A_r)}$ and since $\TilTheta_{P_1(r),d}(t)\!=\!\sum_{i=1}^r \gamma_i\otimes
\gamma_i^\ast\circ t$ for all $t \in \cU_{(r,d)}(U)$ and $U \subseteq \Gr_d(A_r)$ open, we obtain $\TilTheta_d(P_1(r))\cong \cQ_{(r,d)}$. 

(3) In view of Corollary \ref{MinType3}, $M$ is projective or a brick. In the former case, $\cF$ is as in (2), in the latter case, Theorem \ref{Fun3} gives $\End_{\Gr_d(A_r)}(\cF) \cong \End_{K_r}(M) \cong \KK$, so that $\cF$
is simple. 

(4) This follows as in (3), using Corollary \ref{MinType3}(2d,2b).

(5) Let $n\!:=\!\Delta_M(d)\!-\!(r\!-\!d)$. Proposition \ref{MinType5} provides a short exact sequence 
\[ (0) \lra nP_0(r) \lra M \lra N \lra (0),\]
where $N \in \repp(K_r,d)$ has minimal type. Observing Corollary \ref{Fun4}, we may apply $\TilTheta_d$ to obtain the asserted sequence. \end{proof}  

\bigskip

\begin{Remarks} Suppose that $\Char(\KK)\!=\!0$. 

(1) In this case, part (1) was proved in \cite[(2.4)]{AM15}, using Porteous' formula. 

(2) For $d\!=\!1$, part (5) was obtained in \cite[(3.6)]{MMR21}. \end{Remarks} 

\bigskip

\section{Indecomposable Steiner bundles: Kac's Theorem and the Auslander-Reiten quiver}\label{S:KAR}
In this section, we introduce our two main representation-theoretic tools: Kac's Theorem concerning dimension vectors of indecomposable representations of $K_r$ and the Auslander-Reiten quiver $\Gamma(K_r)$ of 
$\rep(K_r)$. In combination with the functors $\TilTheta_d$ the latter will be used to endow the set of isoclasses of indecomposable Steiner bundles with the structure of a quiver, whose connected components can be 
studied in terms of those of $\Gamma(K_r)$.   

\bigskip

\subsection{Kac's Theorem} \label{S:Kac} Recall that an indecomposable module $M \in \rep(K_r)$ is called regular, if it is neither preinjective nor preprojective. For future reference it will be convenient to have the following version of  \cite[Thm.B]{Ka82}, which also builds on \cite[Thm.3]{Ri76}, at our disposal:

\bigskip

\begin{Thm}[Kac's Theorem for $K_r$] \label{Kac1} Let $r\!\ge\!2$ and $\delta \in \NN^2_0\!\smallsetminus\!\{0\}$.
\begin{enumerate}
\item If $\delta\!=\!\udim M$ for some indecomposable $M \in \rep(K_r)$, then $q_r(\delta)\!\le\!1$. 
\item If $q_r(\delta)\!=\!1$, then there is a, up to isomorphism, unique indecomposable module $M \in \rep(K_r)$ such that $\udim M\!=\!\delta$. The module $M$ is preprojective or preinjective.
\item If $q_r(\delta)\!\le\!0$, then every indecomposable module $M \in \rep(K_r)$ such that $\udim M\!=\!\delta$ is regular. Moreover, there are infinitely many isoclasses of these modules. 
\end{enumerate} \end{Thm} 

\bigskip

\subsection{Exceptional Steiner bundles} \label{S:ExSt} In this section we shall classify exceptional Steiner bundles. Throughout, we assume that $r\!\ge\!2$ and $d \in \{1,\ldots, r\!-\!1\}$. Our first result characterizes the 
modules corresponding to exceptional Steiner bundles. 

\bigskip

\begin{Prop} \label{ExSt1} Suppose that $r\!\ge\!3$, and let $\cF \in \StVect(\Gr_d(A_r))$ be a Steiner bundle. Then the following 
statements are equivalent:
\begin{enumerate} 
\item $\cF$ is exceptional.
\item There is $M \in \rep(K_r)$ indecomposable preprojective such that $\cF \cong \TilTheta_d(M)$. \end{enumerate} \end{Prop}

\begin{proof} (1) $\Rightarrow$ (2). By definition, we have $\End_{\Gr_d(A_r)}(\cF)\cong \KK$, while $\Ext^i_{\Gr_d(A_r)}(\cF,\cF)\!=\!(0)$ for all $i\!\ge\!1$. Theorem \ref{Fun3} provides $M \in
\repp(K_r,d) \subseteq \EKP(K_r)$ such that $\TilTheta_d(M)\cong \cF$. In particular, $M$ is a brick and hence indecomposable. In view of \cite[Thm.A]{Wo13}, the module $M$ is not preinjective. Corollary 
\ref{Fun5} yields 
\[ \Ext^1_{K_r}(M,M) \cong \Ext^1_{\Gr_d(A_r)}(\TilTheta_d(M),\TilTheta_d(M)) = (0),\]
so that Theorem \ref{Kac1} in conjunction with the Euler-Ringel form (cf.\ Section \ref{S:reps}) ensures that $M$ is preprojective. 

(2) $\Rightarrow$ (1). Since $M$ is preprojective, an application of \cite[(VIII.2.7)]{ARS95} in conjunction with Corollary \ref{Fun5} implies 
\[\Ext^1_{\Gr_d(A_r)}(\cF,\cF) \cong  \Ext^1_{K_r}(M,M)\!=\!(0).\]
Now Theorem \ref{Kac1} ensures that $M$ is a brick, so that the $\KK$-space $\End_{\Gr_d(A_r)}(\cF,\cF) \cong \End_{K_r}(M)$ is one-dimensional. In view of Lemma \ref{Fun2}, part (iii) of the proof of Proposition \ref{Co1} 
yields $\Ext^i_{\Gr_d(A_r)}(\cF,\cF)\!=\!(0)$ for all $i\!\ge\!2$. As a result, the Steiner bundle $\cF$ is exceptional. \end{proof}

\bigskip
\noindent
Part (3) of our next result extends \cite[(2.1)]{Bra05}, where the bundles involved were assumed to be generic. Part (5) shows that \cite[(6.3)]{Bra08} holds for Grassmannians over fields of arbitrary characteristic.

Recall from Section \ref{S:reps} that $(P_n(r))_{n\ge 0}$ denotes the family of preprojective $K_r$-representations (of dimension vectors $\udim P_n(r)\!=\!(a_n(r),a_{n+1}(r))$).

\bigskip

\begin{Thm} \label{ExSt2} Suppose that $r\!\ge\!3$. Then the following statements hold:
\begin{enumerate}
\item For each $n \in \NN_0$ there exists an exceptional Steiner bundle $\cE_{n,d} \in \StVect(\Gr_d(A_r))$ such that $\rk(\cE_{n,d})\!=\! a_{n+1}(r)\! -\!da_n(r)$ and $c_1(\cE_{n,d})\! =\! a_n(r)\sigma_1$.
\item If $\cE \in \StVect(\Gr_d(A_r))$ is an exceptional Steiner bundle, then $\cE \cong \cE_{n,d}$ for some $n \in \NN_0$.
\item An indecomposable Steiner bundle $\cF \in \StVect(\Gr_d(A_r))$ is exceptional if and only if $\rk(\cF)\!=\!a_{n+1}(r)\!-\!da_n(r)$ and $c_1(\cF)\! =\! a_n(r)\sigma_1$ for some $n \in \NN_0$.
\item A non-zero Steiner bundle $\cF \in \StVect(\Gr_d(A_r))$ satisfies $\Ext^1_{\Gr_d(A_r)}(\cF,\cF)\!=\!(0)$ if and only if there exist $(a,b) \in \NN\!\times\!\NN_0$ and $n \in \NN_0$ such that $\cF \cong \cE_{n,d}^a\!
\oplus\!\cE_{n+1,d}^b$.
\item Let $(c,s) \in \NN_0\!\times\!\NN$ such that $q_r(c,s\!+\!dc)\!>\!1$, then there exist uniquely determined elements $n,a,b \in \NN_0$ such that a generic Steiner bundle $\cF \in \StVect(\Gr_d(A_r))$ of rank $s$ and with first Chern class $c \sigma_1$ is isomorphic to $\cE_{n,d}^a\!\oplus\! \cE_{n+1,d}^b$. \end{enumerate}\end{Thm}

\begin{proof} (1) Let $n \in \NN_0$ and consider the preprojective indecomposable representation $P_n(r)$. In view of $\udim P_n(r)\! =\! (a_n(r),a_{n+1}(r))$, an application of Theorem \ref{Fun3}, 
Proposition \ref{ExSt1} and Corollary \ref{Dc2} shows that $\cE_{n,d}\!:=\! \TilTheta_d(P_n(r))$ is an exceptional Steiner bundle of rank $\Delta_{P_n(r)}(d)\! =\! a_{n+1}(r)\! -\!da_n(r)$ and first Chern class 
$c_1(\cE_{n,d})\! =\! \dim_\KK P_n(r)_1\sigma_1\! =\! a_n(r)\sigma_1$.

(2) This follows immediately from Proposition \ref{ExSt1}, the definition of $\cE_{n,d}$ and the above-mentioned classification of the indecomposable preprojective $K_r$-representations. 

(3) In view of (1), exceptional Steiner bundles enjoy the stated properties. Conversely, let $\cF$ be an indecomposable Steiner bundle such that $\rk(\cF)\! =\!a_{n+1}(r)\!-\!da_n(r)$ and $c_1(\cF)\!=\!a_n(r)\sigma_1$. 
According to Theorem \ref{Fun3}, there is an indecomposable representation $M \in \repp(K_r,d)$ such that $\cF \cong \TilTheta_d(M)$ and $\udim M\!=\!(a_n(r),a_{n+1}(r))$. Since $q_r(\udim M)\! =\! 
q_r(a_n(r),a_{n+1}(r))\!=\!1$, Theorem \ref{Kac1} implies that there is a unique indecomposable representation with dimension vector $(a_n(r),a_{n+1}(r))$. Hence $M \cong P_n(r)$ and $\cF \cong \cE_{n,d}$ is an 
exceptional Steiner bundle. 

(4) Let $\cF \in \StVect(\Gr_d(A_r))$ be a non-zero Steiner bundle that satisfies $\Ext^1_{\Gr_d(A_r)}(\cF,\cF)\!=\!(0)$. Application of Theorem \ref{Fun3} implies the existence of $(0)\!\ne\!M \in \repp(K_r,d)$ such that 
$\cF \cong \TilTheta_d(M)$. Let $N$ be an indecomposable direct summand of $M$. By Proposition \ref{CatRep1} we have $N \in \repp(K_r,d)$. As $\Ext^1_{K_r}(N,N)\!=\!(0)$, Theorem \ref{Kac1} provides 
$i \in \NN_0$ such that $N\cong P_i(r)$. We fix $n \in \NN_0$ minimal subject to $P_n(r)$ being a direct summand of $M$. Let $m\!>\!n$  and assume that $P_m(r)$ is a direct summand of $M$. We set 
$l\!:=\!m\!-\!n\!-\!1\!\geq\!0$ and conclude with the Auslander-Reiten formula \cite[(IV.2.13)]{ASS06}
\begin{align*}
0 &= \dim_\KK \Ext^1_{K_r}(P_m(r),P_n(r)) = \dim_\KK \Hom_{K_r}(P_{n+2}(r),P_m(r))\\
&= \dim_\KK \Hom_{K_r}(\sigma_{K_r}^{n+1}(P_{n+2}(r)),\sigma_{K_r}^{n+1}(P_m(r))) = \dim_\KK \Hom_{K_r}(P_1(r),P_{l}(r)) = \dim_\KK P_{l}(r)_1.
\end{align*} 
Hence $l\!=\!0$ and $m\!=\!n\!+\!1$. The result now follows by applying $\TilTheta_d$.

The other implication follows in a similar fashion using the fact that $\Ext^1_{K_r}(P_i(r),P_j(r))\!=\!(0)$ for all $i,j \in \NN$ with $|i\!-\!j|\!\leq\!1$. 

(5) We show the statement on the level of representations and fix vector spaces $V_1,V_2$ of dimension $c$ and $s\!+\!dc$, respectively. Note that we have $\dim_\KK V_1\!<\! \dim_\KK V_2$. 
Now \cite[Thm 4 (b)]{Ka80} \footnote{Using the notation of \cite[Thm 4 (b)]{Ka80} we have $\beta_0\!=\!\udim I_0(r)$ and $\beta_{-i}\!=\!\udim I_i(r)$, $\beta_i\!=\!\udim P_{i-1}(r)$ for all $i \in \NN$. Moreover, since $
\Ext^1_{K_r}(P_i(r),P_j(r))\!=\!(0)$ for all $i,j \in \NN_0$ with $|i\!-\!j|\!\leq\!1$, \cite[Proposition 3]{Ka82} and Theorem \ref{Kac1} guarantee that \cite[Thm 4 (b)]{Ka80} holds in arbitrary characteristic.} ensures the 
existence of uniquely determined elements $a,b,n \in \NN_0$ such that a generic representation in $N \in \rep(K_r;V_1,V_2)$ is of the form $N \cong a P_n(r)\!\oplus\!b P_{n+1}(r) \in \repp(K_r,d)$. By Theorem 
\ref{Fun3} and Corollary \ref{Dc2} we have $\rk(\TilTheta_d(N))\!=\!\Delta_{(V_1,V_2)}(d)\!=\!s$ and $c_1(\TilTheta_d(N))\!=\!c \sigma_{1}$. \end{proof}

\bigskip

\begin{Remarks} (1) The numerical data in (3) do not imply the indecomposability of $\cF$: Let $E \in \rep(K_3)$ be indecomposable of dimension vector $\udim E\!=\!(1,1)$. Since $\udim \tau_{K_3}^{-1}(E)\!=\!(2,5)$, it 
follows from \cite[Thm.]{Bi20b} that $\tau_{K_r}^{-1}(E) \in \EKP(K_3)\!=\!\repp(K_r,1)$ and hence so is $M\!:=\!\tau^{-1}_{K_3}(E)\!\oplus\!P_1(3)$. We have $\udim M\!=\!(3,8)\!=\!\udim P_2(3)$, but $\cF\!:=\!
\TilTheta_1(M)$ is not  indecomposable.

(2) In view of Corollary \ref{Fun4}, the functor $\TilTheta_d|_{\repp(K_r,d)}$ is exact. For $i\!\ge\!2$, the almost split sequence 
\[(0) \lra P_{i-2}(r) \lra rP_{i-1}(r)\lra P_i(r) \lra (0)\]
thus induces an exact sequence
\[ (0) \lra \cE_{i,d}^\vee \lra r \cE_{i-1,d}^\vee \lra \cE_{i-2,d}^\vee \lra (0)\]
of vector bundles. For $d\!=\!1$, this corresponds to the construction given in \cite[(2.1)]{Bra05} (with an extra term given by $\cE_{-1,d}\!:=\!\cO(-1)$). As pointed out in \cite[(2.3)]{Bra05}, the sequences are thus 
related to the mutations studied in \cite{GR87}. \end{Remarks}

\bigskip

\subsection{Decomposing $\StVect(\Gr_d(A_r))$ via the Auslander-Reiten quiver} \label{S:Strat}
Let $\Gamma(K_r)$ be the Auslander-Reiten quiver of $\rep(K_r)$. Its vertices are the isoclasses of the indecomposable representations, while its arrows correspond to  the so-called 
{\it irreducible morphisms}.\footnote{By abuse of notation, we will henceforth usually not distinguish between indecomposable representations and the vertices of $\Gamma(K_r)$ associated to them.}  Let $\cC \subseteq 
\Gamma(K_r)$ be a component. Given $M \in \cC$, we let $(\rightarrow\! M)$ and $(M\!\rightarrow)$ be the sets of predecessors and successors of $M$, respectively. We refer the reader to \cite{ARS95} for basic  facts 
on Auslander-Reiten theory and continue by recording some special features of $\Gamma(K_r)$.

Let $r\!\ge\!3$, $d \in \{1,\ldots, r\!-\!1\}$. The Auslander-Reiten quiver $\Gamma(K_r)$ of $\rep(K_r)$ is a disjoint union of components. There is a preprojective component $\cP$ (containing all preprojective modules), a 
preinjective component $\cI$ (containing all preinjective modules), and all other components are called {\it regular}. Owing to Proposition \ref{CatRep1} the preprojective component $\cP$ is contained in 
$\repp(K_r,d)$ and \cite[Thm.A]{Wo13} implies that $\EKP(K_r)$ contains no preinjective modules. Recall that $\{P_n(r) \ ; \ n\!\ge\!0\}$ is the set of vertices of $\cP$, so that Proposition \ref{ExSt1} ensures that 
$\TilTheta_d(\cP)$ is just the set $\{\cE_{n,d} \ ; \ n\!\ge\!0\}$ of exceptional bundles. 

Let $\cC \subseteq \Gamma(K_r)$ be a regular AR-component. By Ringel's Theorem \cite[(2.3)]{Ri78}, which is fundamental for our methods, $\cC$ is a stable translation quiver of isomorphism type 
$\ZZ[A_\infty]$. A representation $M \in \cC$ is called {\it quasi-simple}, provided $M$ has precisely one immediate predecessor. The quasi-simple representations are therefore located at the ``bottom'' 
of the component. 

For each $M \in \cC$, there is a unique directed path
\[ M_{[1]} \rightarrow M_{[2]} \rightarrow \cdots \rightarrow M_{[n-1]} \rightarrow M_{[n]}\!=\!M\]
in $\cC$ such that
\begin{enumerate}
\item[(a)] for every $i \in \{1,\ldots,n\!-\!1\}$ the irreducible morphism $M_{[i]} \lra M_{[i+1]}$ is injective, and
\item[(b)] for every $i \in \{1,\ldots,n\}$ the factor module $M_{[i]}/M_{[i-1]} \cong \tau^{-i+1}_{K_r}(M_{[1]})$ ($M_{[0]}\!:=\!(0)$) is quasi-simple. \end{enumerate}
In that case $\ql(M)\!:=\!n$ is called the {\it quasi-length} of $M$, and we refer to $M_{[1]}$ and $M/M_{[n-1]}$ as the {\it quasi-socle} and {\it quasi-top} of $M$, respectively.  
The filtration
\[ M_{[1]} \subsetneq \cdots \subsetneq M_{[n-1]} \subsetneq M\]
is called the {\it quasi-composition series} of $M$ with {\it quasi-composition factors} $\{\tau^{-i}_{K_r}(M_{[1]}) \ ; \ 0 \!\le\! i\!\le \ql(M)\!-\!1\}$. 

\bigskip

\begin{Prop}\label{Strat1}  For each regular component $\cC$ the set $\repp(K_r,d) \cap \cC$ forms a non-empty cone in $\cC$, consisting of the successors $(M_{\cC,d}\!\rightarrow)$ of a uniquely determined 
quasi-simple representation $M_{\cC,d}$ in $\cC$. \end{Prop}

\begin{proof} This follows from Theorem \ref{Fam5} and \cite[(2.1.3)]{Bi20}. \end{proof}

\begin{figure*}[!h]
\centering 
\begin{tikzpicture}[very thick, scale=1]
                    [every node/.style={fill, circle, inner sep = 1pt}]

\def \n {9} 
\def \m {3} 
\def \translation {1} 

\def \ab {0.15} 
\def \Pab {0.6} 

\def \rcone {1} 
\def \rdist {2} 
\def \rcolor {gray} 

\def \rrcone {1} 
\def \rrdist {3} 

\node[color=black]  at (\n*\Pab*2-\Pab*2*3,-\ab-0.3) {$M_{\mathcal{C},1}$};
\node[color=black]  at (\n*\Pab*2-\Pab*2*2,-\ab-0.3) {$M_{\mathcal{C},r-1}$};

\foreach \a in {0,...,\n}{
\foreach \b in {0,...,\m}{
  
   \ifthenelse{\a = \n \and \b < \m}{
   \node[color=black] ({\a,\b,5})at ({\a*2*\Pab},{\b*2*\Pab}) {$\circ$};
     }
     {
      \ifthenelse{\b = \m \and \a < \n}{
      \node[color=black] ({\a,\b}) at ({\a*2*\Pab+\Pab},{\b*2*\Pab+\Pab}) {$\circ$};
      \node[color=black] ({\a,\b,5})at ({\a*2*\Pab},{\b*2*\Pab}) {$\circ$};
      }
      {
    
     \ifthenelse{\b = \m \and \a = \n}
     {\node[color=black] ({\a,\b,5})at ({\a*2*\Pab},{\b*2*\Pab}) {$\circ$};}
    {\node[color=black] ({\a,\b}) at ({\a*2*\Pab+\Pab},{\b*2*\Pab+\Pab}) {$\circ$};
    \node[color=black] ({\a,\b,5})at ({\a*2*\Pab},{\b*2*\Pab}) {$\circ$};

      }
      }
      }
    }
    }

\foreach \s in {0,...,\n}{
\foreach \t in {0,...,\m}
{  
 \ifthenelse{\t = \m \and \s < \n}{
    \draw[->] (\s*2*\Pab+\ab,\t*2*\Pab+\ab) to (\s*2*\Pab+\Pab-\ab,\t*2*\Pab+\Pab-\ab); 
    \draw[->] (\s*2*\Pab+\Pab+\ab,\t*2*\Pab+\Pab-\ab) to (\s*2*\Pab+2*\Pab-\ab,\t*2*\Pab+\ab); 

  }{
  
  \ifthenelse{\s = \n \and \t < \m}{
  
  }
  {
  \ifthenelse{\s = \n \and \t = \m}{
   
  }{
   \draw[->] (\s*2*\Pab+\ab,\t*2*\Pab+\ab) to (\s*2*\Pab+\Pab-\ab,\t*2*\Pab+\Pab-\ab); 
   \draw[->] (\s*2*\Pab+\Pab+\ab,\t*2*\Pab+\Pab+\ab) to (\s*2*\Pab+2*\Pab-\ab,\t*2*\Pab+2*\Pab-\ab);
   \draw[->] (\s*2*\Pab+\ab,\t*2*\Pab+2*\Pab-\ab) to (\s*2*\Pab+\Pab-\ab,\t*2*\Pab+\Pab+\ab); 
   \draw[->] (\s*2*\Pab+\Pab+\ab,\t*2*\Pab+\Pab-\ab) to (\s*2*\Pab+2*\Pab-\ab,\t*2*\Pab+\ab);    
   }
   }
  
    }
    }
    }

\ifthenelse{\isodd{\m}}
 { 
  \node[color=black] (Dots1) at (0,\m*\Pab+2*\Pab) {$\cdots$};
  \node[color=black] (Dots2) at (\n*2*\Pab,\m*\Pab+2*\Pab) {$\cdots$};
   \ifthenelse{\isodd{\n}}{
  \node[color=black] (Dots3) at (0.5*\n*2*\Pab,2*\m*\Pab+2*\Pab) {$\vdots$};}
  {\node[color=black] (Dots3) at (0.5*\n*2*\Pab,2*\m*\Pab+\Pab) {$\vdots$};} 
  }
  {
  \node[color=black] (Dots1) at (0,\m*\Pab+\Pab) {$\cdots$};
  \node[color=black] (Dots2) at (\n*2*\Pab,\m*\Pab+\Pab) {$\cdots$};
  \ifthenelse{\isodd{\n}}{
  \node[color=black] (Dots3) at (0.5*\n*2*\Pab,2*\m*\Pab+2*\Pab) {$\vdots$};}
  {\node[color=black] (Dots3) at (0.5*\n*2*\Pab,2*\m*\Pab+\Pab) {$\vdots$};}
  }
 
\ifthenelse{\translation = 1}{
   \foreach \s in {0,...,\n}{
   \foreach \t in {0,...,\m}{ 
   \ifthenelse{\s = 0}{}{
      \ifthenelse{\s = \n}{\draw[->,dotted,thin] (\s*2*\Pab-\ab,\t*2*\Pab) to (\s*2*\Pab-2*\Pab+\ab,\t*2*\Pab); }{
   \draw[->,dotted,thin] (\s*2*\Pab-\ab,\t*2*\Pab) to (\s*2*\Pab-2*\Pab+\ab,\t*2*\Pab); 
   \draw[->,dotted,thin] (\s*2*\Pab-\ab+\Pab,\t*2*\Pab+\Pab) to (\s*2*\Pab-2*\Pab+\Pab+\ab,\t*2*\Pab+\Pab); 
   }
   }}
}}
{}  

\begin{scope}[on background layer]

\ifthenelse{\rrcone = 1}{
        \draw[fill = \rcolor!20] (\n*\Pab*2+\ab,-\ab) node[anchor=north]{}
  -- (\n*\Pab*2-\rrdist*\Pab*2-0.7*\Pab,-\ab) node[anchor=north]{}
  -- (\n*\Pab*2+\ab,\rrdist*\Pab*2+0.7*\Pab) node[anchor=south]{};
    }
  {}

\ifthenelse{\rcone = 1}{
        \draw[fill= \rcolor!40](\n*\Pab*2+\ab,-\ab) node[anchor=north]{}
  -- (\n*\Pab*2-\rdist*\Pab*2-0.7*\Pab,-\ab) node[anchor=north]{}
  -- (\n*\Pab*2+\ab,\rdist*\Pab*2+0.7*\Pab) node[anchor=south]{};
    }
  {}   
   
\end{scope}
\end{tikzpicture}
\caption{Cones of successors of $M_{\mathcal{C},1}$ and $M_{\mathcal{C},r-1}$ in a regular component $\mathcal{C} \subseteq \Gamma(K_r)$ with $M_{\mathcal{C},i} \neq M_{\mathcal{C},i+1}$ for some 
$1\!\leq\!i\!<\! r\! -\! 1$ (cf.\ Theorem \ref{CatRep3}). The Auslander-Reiten translation is indicated by dotted arrows.}
\end{figure*}
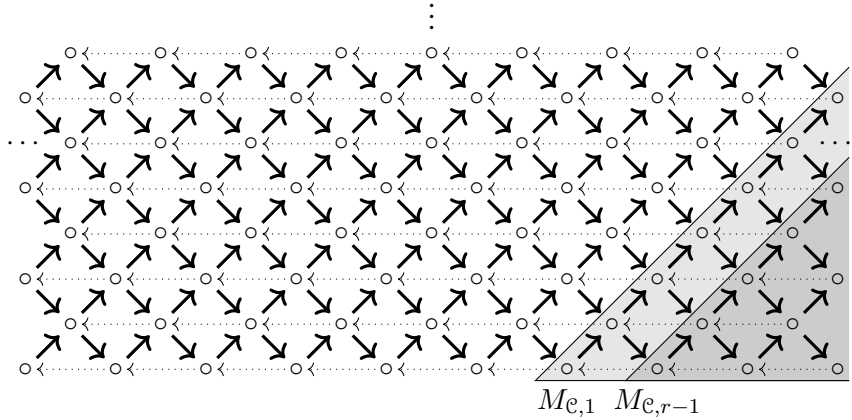

\bigskip

\begin{Lem} \label{Strat2} Let $M \in \repp(K_r,d)$ be an indecomposable representation. Then $\TilTheta_d(M)$ is an indecomposable object in $\Coh(\Gr_d(A_r))$. \end{Lem}

\begin{proof} By Theorem \ref{Fun3}, the functor $\TilTheta_d$ is full and faithful, so that the ring $\End_{\Gr_d(A_r)}(\TilTheta_d(M))\cong \End_{K_r}(M)$ is local. Consequently, $\TilTheta_d(M)$ is 
indecomposable. \end{proof}

\bigskip
\noindent
We denote by $\msInd\StVect(\Gr_d(A_r))$ the full subcategory of indecomposable Steiner bundles. Theorem \ref{Fun3} and Lemma \ref{Strat2} provide a decomposition
\[ \msInd\StVect(\Gr_d(A_r)) = \TilTheta_d(\cP)\sqcup \bigsqcup_{\cC} \TilTheta_d((M_{\cC,d}\!\rightarrow)),\]
with $\cC$ running through the set of regular AR-components of $K_r$.\footnote{Abusing notation, we write $\msInd\StVect(\Gr_d(A_r))$ for the set of isoclasses of the objects of $\msInd\StVect(\Gr_d(A_r))$.}  

Recall that $\Phi_r^{-1}\!=\! \left(\begin{smallmatrix} -1 & r \\ -r & r^2\!-\!1 \end{smallmatrix}\right)$, while $\udim \tau^{-1}_{K_r}(M)\!=\!\Phi^{-1}_{K_r}.\udim M$ for every $M \in \rep(K_r)$ without non-zero injective direct 
summands, cf.\ \cite[(VIII.2.2)]{ARS95}. 

If $M \not \cong P_0(r)$ is indecomposable, then 
\[ \udim \sigma_{K_r}(M) = (r \dim_\KK M_1\!-\!\dim_\KK M_2, \dim_\KK M_1),\] 
while $\sigma_{K_r}(P_0(r))\!=\!(0)$ and $\sigma_{K_r}(P_{i+1}(r)) \cong  P_{i}(r)$ for all $i\! \geq\! 0$. Similarly, 
\[\udim \sigma_{K_r}^{-1}(M)  = (\dim_\KK M_2, r\dim_\KK M_2\!-\!\dim_\KK M_1)\] 
for $M \not\cong S_1$. Thus, viewing $M \in \EKP(K_r)\!\smallsetminus\!\{0\}$ as a $\KK K_r$-module, we obtain 
\[ (\dagger) \ \ \ \ \ \ \dim_\KK\sigma^{-1}_{K_r}(M) = \dim_\KK M\!+\!2\Delta_M(1)\!\!+\!(r\!-\!2)\dim_\KK M_2 > \dim_\KK M.\]
In particular, all irreducible morphisms between preprojective representations in $\rep(K_r)$ are injective, and $P_i(r)$ is isomorphic to a subrepresentation of $P_j(r)$ if and only if $i\! \leq\! j$.

\bigskip

\begin{Lem} \label{Strat3} Let $M \in \repp(K_r,d)$ be indecomposable. Then the following statements hold:
\begin{enumerate}
\item If $M\not\cong P_0(r)$, then $\Delta_{\sigma_{K_r}^{-1}(M)}(d)\!>\!\Delta_M(d)$.
\item $\Delta_{\tau^{-1}_{K_r}(M)}(d)\!>\!\Delta_M(d)$. \end{enumerate} \end{Lem}

\begin{proof} (1) As $M$ is not isomorphic to $I_0(r)\!=\!S_1$, we have 
\[\udim \sigma^{-1}_{K_r}(M) = (\dim_\KK M_2, r\dim_\KK M_2\!-\!\dim_\KK M_1).\] 
We thus obtain, observing $\Delta_M(d)\!\ge\!0$ and $M \not \cong P_0(r)$,
\begin{eqnarray*}
\Delta_{\sigma^{-1}_{K_r}(M)}(d)\!-\!\Delta_M(d) & = & r\dim_\KK M_2\!-\!\dim_\KK M_1\!-\!d\dim_\KK M_2\!-\!\dim_\KK M_2\!+\!d\dim_\KK M_1 \\
                                                                            & = & (r\!-\!d\!-\!1)\dim_\KK M_2\!+\!(d\!-\!1)\dim_\KK M_1 \\
                                                                            &\ge & ((r\!-\!d)d\!-\!1)\dim_\KK M_1 > 0. 
\end{eqnarray*}

(2) This follows from (1), as $\tau^{-1}_{K_r}(M)\cong \sigma_{K_r}^{-2}(M)$ and $\sigma^{-1}_{K_r}(M)\not \cong P_0(r)$. \end{proof}  

\bigskip

\bigskip
\begin{Prop} \label{Strat4} Let $r\!\ge\!3$, $d \in \{1,\ldots,r\!-\!1\}$. The following statements hold:
\begin{enumerate}
\item Let $\cC \subseteq \Gamma(K_r)$ be a regular AR-component. 
\begin{enumerate}
\item[(a)] We have $\rk(\TilTheta_d(M_{\cC,d}))\!<\!\rk(\TilTheta_d(M))$ for all $M \in (M_{\cC,d}\!\rightarrow)\!\smallsetminus\!\{M_{\cC,d}\}$. 
\item[(b)] If $M_{\cC,d}$ is a brick, then every $\cF \in \TilTheta_d((M_{\cC,d}\!\rightarrow))$ has a filtration by simple Steiner bundles, whose filtration factors are pairwise non-isomorphic. \end{enumerate}
\item If $M \in \repp(K_r,d)$ is indecomposable of minimal type, then either $M\cong P_1(r)$ and $d\!=\!1$, or $M \cong P_2(r)$ and $d\!=\!r\!-\!1$, or there is a regular AR-component $\cC$ such that $M \cong 
M_{\cC,d}$. 
\item If $\cF \in \StVect(\Gr_d(A_r))$ is simple, then either $\cF \in \TilTheta_d(\cP)$, or there is a regular AR-component $\cC$ and a natural number $n \in \NN_0$ such that $\cF \cong \TilTheta_d(\tau^{-n}_{K_r}
(M_{\cC,d}))$. \end{enumerate} \end{Prop}

\begin{proof}  Let $\cC$ be a regular component, $f : (M_\cC\!\rightarrow) \lra \NN_0$ be a function that is additive on almost split sequences and split exact sequences. For $M \in (M_{\cC,d}\!\rightarrow)$, we let 
$\cW(M) \subseteq \cC$ be the {\it wing} of $M$. (This is the full subgraph of $\cC$ bounded by the triangle with vertices $M$, the quasi-socle of $M$ and the quasi-top of $M$. 

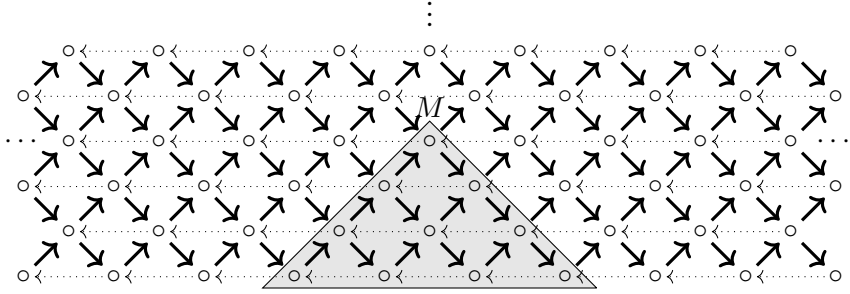
\begin{figure*}[!h]
\centering 
\begin{tikzpicture}[very thick, scale=1]
                    [every node/.style={fill, circle, inner sep = 1pt}]

\def \n {9} 
\def \m {2} 
\def \translation {1} 

\def \ab {0.15} 
\def \Pab {0.6} 

\def \wing {1} 
\def \wingheight {3} 
\def \wingcolor {gray} 

\node[color=black]  at (\n*\Pab*2-1.5*\wingheight*\Pab*2,-\ab+\wingheight*\Pab+\Pab) {$M$};

\foreach \a in {0,...,\n}{
\foreach \b in {0,...,\m}{
  
   \ifthenelse{\a = \n \and \b < \m}{
   \node[color=black] ({\a,\b,5})at ({\a*2*\Pab},{\b*2*\Pab}) {$\circ$};
     }
     {
      \ifthenelse{\b = \m \and \a < \n}{
      \node[color=black] ({\a,\b}) at ({\a*2*\Pab+\Pab},{\b*2*\Pab+\Pab}) {$\circ$};
      \node[color=black] ({\a,\b,5})at ({\a*2*\Pab},{\b*2*\Pab}) {$\circ$};
      }
      {
    
     \ifthenelse{\b = \m \and \a = \n}
     {\node[color=black] ({\a,\b,5})at ({\a*2*\Pab},{\b*2*\Pab}) {$\circ$};}
    {\node[color=black] ({\a,\b}) at ({\a*2*\Pab+\Pab},{\b*2*\Pab+\Pab}) {$\circ$};
    \node[color=black] ({\a,\b,5})at ({\a*2*\Pab},{\b*2*\Pab}) {$\circ$};

      }
      }
      }
    }
    }

\foreach \s in {0,...,\n}{
\foreach \t in {0,...,\m}
{  
 \ifthenelse{\t = \m \and \s < \n}{
    \draw[->] (\s*2*\Pab+\ab,\t*2*\Pab+\ab) to (\s*2*\Pab+\Pab-\ab,\t*2*\Pab+\Pab-\ab); 
    \draw[->] (\s*2*\Pab+\Pab+\ab,\t*2*\Pab+\Pab-\ab) to (\s*2*\Pab+2*\Pab-\ab,\t*2*\Pab+\ab); 

  }{
  
  \ifthenelse{\s = \n \and \t < \m}{
  
  }
  {
  \ifthenelse{\s = \n \and \t = \m}{
   
  }{
   \draw[->] (\s*2*\Pab+\ab,\t*2*\Pab+\ab) to (\s*2*\Pab+\Pab-\ab,\t*2*\Pab+\Pab-\ab); 
   \draw[->] (\s*2*\Pab+\Pab+\ab,\t*2*\Pab+\Pab+\ab) to (\s*2*\Pab+2*\Pab-\ab,\t*2*\Pab+2*\Pab-\ab);
   \draw[->] (\s*2*\Pab+\ab,\t*2*\Pab+2*\Pab-\ab) to (\s*2*\Pab+\Pab-\ab,\t*2*\Pab+\Pab+\ab); 
   \draw[->] (\s*2*\Pab+\Pab+\ab,\t*2*\Pab+\Pab-\ab) to (\s*2*\Pab+2*\Pab-\ab,\t*2*\Pab+\ab);    
   }
   }
  
    }
    }
    }

\ifthenelse{\isodd{\m}}
 { 
  \node[color=black] (Dots1) at (0,\m*\Pab+2*\Pab) {$\cdots$};
  \node[color=black] (Dots2) at (\n*2*\Pab,\m*\Pab+2*\Pab) {$\cdots$};
   \ifthenelse{\isodd{\n}}{
  \node[color=black] (Dots3) at (0.5*\n*2*\Pab,2*\m*\Pab+2*\Pab) {$\vdots$};}
  {\node[color=black] (Dots3) at (0.5*\n*2*\Pab,2*\m*\Pab+\Pab) {$\vdots$};} 
  }
  {
  \node[color=black] (Dots1) at (0,\m*\Pab+\Pab) {$\cdots$};
  \node[color=black] (Dots2) at (\n*2*\Pab,\m*\Pab+\Pab) {$\cdots$};
  \ifthenelse{\isodd{\n}}{
  \node[color=black] (Dots3) at (0.5*\n*2*\Pab,2*\m*\Pab+2*\Pab) {$\vdots$};}
  {\node[color=black] (Dots3) at (0.5*\n*2*\Pab,2*\m*\Pab+\Pab) {$\vdots$};}
  }
 
\ifthenelse{\translation = 1}{
   \foreach \s in {0,...,\n}{
   \foreach \t in {0,...,\m}{ 
   \ifthenelse{\s = 0}{}{
      \ifthenelse{\s = \n}{\draw[->,dotted,thin] (\s*2*\Pab-\ab,\t*2*\Pab) to (\s*2*\Pab-2*\Pab+\ab,\t*2*\Pab); }{
   \draw[->,dotted,thin] (\s*2*\Pab-\ab,\t*2*\Pab) to (\s*2*\Pab-2*\Pab+\ab,\t*2*\Pab); 
   \draw[->,dotted,thin] (\s*2*\Pab-\ab+\Pab,\t*2*\Pab+\Pab) to (\s*2*\Pab-2*\Pab+\Pab+\ab,\t*2*\Pab+\Pab); 
   }
   }}
}}
{}  

\begin{scope}[on background layer]

\ifthenelse{\wing = 1}{
         \draw[fill = \wingcolor!20] (\n*\Pab*2-\wingheight*\Pab*2+0.7*\Pab,-\ab) node[anchor=north]{}
  -- (\n*\Pab*2-2*\wingheight*\Pab*2-0.7*\Pab,-\ab) node[anchor=north]{}
  -- (\n*\Pab*2-1.5*\wingheight*\Pab*2,-\ab+\wingheight*\Pab+0.7*\Pab) node[anchor=south]{}
  -- (\n*\Pab*2-\wingheight*\Pab*2+0.7*\Pab,-\ab) node[anchor=north]{};
      }

  {}

\end{scope}
\end{tikzpicture}
\caption{Wing $\mathcal{W}(M)$ of a regular representation $M$ of quasi-length $4$ in a regular component.}
\end{figure*}

Then $\cW(M) \subseteq 
(M_{\cC.d}\!\rightarrow)$, and we put
\[ \cW(M)_1 := \{ X \in \cW(M) \ ; \ \ql(X)\!=\!1\}.\]
(This is the set of quasi-composition factors of $M$.) Using induction on $\ql(M)$ one verifies that 
\[ (\ast) \ \ \ \ \ \ \ \ f(M) = \sum_{X \in \cW(M)_1} f(X).\]

(1a) Let $M \in (M_{\cC,d}\!\rightarrow)\!\smallsetminus\!\{M_{\cC,d}\}$. According to Theorem \ref{Fun3}(3), the function
\[ f : (M_{\cC,d}\!\rightarrow) \lra \NN_0 \ \ ; \ \ M \mapsto \rk(\TilTheta_d(M))\]
is additive on almost split sequences and split exact sequences. Consequently, identity ($\ast$) in conjunction with Lemma \ref{Strat3}(2) and Theorem \ref{Fun3}(3) yields
\[ \rk(\TilTheta_d(M)) = \sum_{X \in \cW(M)_1} \rk(\TilTheta_d(X)) > \rk(\TilTheta_d(M_{\cC,d})).\]

(1b) Since $M_{\cC,d}$ is a brick, it follows that $\tau_{K_r}^{-n}(M_{\cC,d})$ is a brick for all $n \in \NN_0$.  Let $M \in (M_{\cC,d}\!\rightarrow)$. The quasi-composition series of $M$ has $\cW(M)_1 
\subseteq \{\tau_{K_r}^{-n}(M_{\cC,d}) \ ; \ n \in \NN_0\}$ as set of filtration factors, with each factor occurring with multiplicity $1$. By Corollary \ref{Fun4} and Theorem \ref{Fun3}, the functor $\TilTheta_d : \repp(K_r,d) 
\lra \StVect(\Gr_d(A_r))$ is an exact equivalence, so that the assertion follows.  

(2) Suppose that $M$ is regular and let $\cC$ be the regular AR-component containing $M$. A consecutive application of Corollary \ref{MinType3}(2b) and \cite[(XVIII.2.15)]{SS07} ensures that $M$ is quasi-simple. 
Hence there is $n\!\ge\!0$ such that $M \cong \tau_{K_r}^{-n}(M_{\cC,d})$. Since $M$ has minimal type, Lemma \ref{Strat3}(2) and Corollary \ref{MinType3}(1a) yield $n\!=\!0$, whence $M\cong M_{\cC,d}$.

If $M$ is preprojective, then $M\cong P_i(r)$ for some $i\!\ge\!0$. As $r\!\ge\!3$, we have $\Delta_{P_0(r)}(d)\!=\!1\!<\!d(r\!-\!d)$, while $\Delta_{P_1(r)}(d)\!=\!r\!-\!d$ and 
\[ \Delta_{P_2(r)}(d)\!-\! d(r\!-\!d) = \Delta_{\sigma^{-1}_{K_r}(P_1(r))}(d)\!-\!d(r\!-\!d) = (r\!-\!d)^2\!-\!1.\]  
Lemma \ref{Strat3}(1) thus implies $M\!\cong P_1(r)$ and $d\!=\!1$, or $M\cong P_2(r)$ and $d\!=\!r\!-\!1$.

(3) Theorem \ref{Fun3} provides $M \in \repp(K_r,d)$ such that $\cF \cong \TilTheta_d(M)$. Suppose that $M \not\in \cP$. As $\cF$ is simple, $M$ is a brick, and \cite[(XVIII.2.15)]{SS07} shows that $M$ is quasi-simple. 
Hence there is a regular AR-component $\cC$ and a natural number $n \in \NN_0$ such that $M\cong \tau^{-n}_{K_r}(M_{\cC,d})$.  \end{proof}

\bigskip
\noindent
For future reference we record structural features of regular AR-components that are related to the functor $\sigma_{K_r}$. Recall that $\sigma_{K_r}$ and $\sigma^{-1}_{K_r}$ induce quasi-inverse auto-equivalences 
on the full subcategory category $\cR(K_r) \subseteq \rep(K_r)$ of regular representations. Moreover, we have a natural equivalence $\sigma_{K_r} \circ \sigma_{K_r} \cong \tau_{K_r}$. 

\bigskip

\begin{Lem}\label{Strat5} Let $1\! \leq\! d\! <\! r$ and $M \in \rep(K_r)$ be regular indecomposable. Then there exists a unique representation $M_{\sigma,d}$ in the $\sigma_{K_r}$-orbit $\langle\sigma_{K_r}\rangle.M$ 
of $M$ such that 
\[ \repp(K_r,d) \cap \langle\sigma_{K_r}\rangle.M = \{ \sigma^{-n}_{K_r}(M_{\sigma,d}) \ ; \ n \in \NN_0\}.\]
Moreover, we have $M_{\sigma,d} \in \{M_{\sigma,1},\sigma_{K_r}^{-1}(M_{\sigma,1})\}$. \end{Lem}

\begin{proof} In view of Proposition \ref{Strat1}, we have $\repp(K_r,d) \cap \langle\sigma_{K_r}\rangle.M\!\ne\!\emptyset$. Let $M_{\sigma,d} \in \repp(K_r,d) \cap \langle\sigma_{K_r}\rangle.M$ be a vertex of minimal 
dimension. Identity ($\dagger$) right before Lemma \ref{Strat3} in conjunction with Theorem \ref{CatRep3}(3) yields $\repp(K_r,d) \cap \langle\sigma_{K_r}\rangle.M\!=\!\{ \sigma^{-n}_{K_r}(M_{\sigma,d}) \ ; 
\ n \in \NN_0\}$. 

Thanks to Proposition \ref{CatRep1}(3), we have $M_{\sigma,d} \in \repp(K_r,1)\cap\langle \sigma_{K_r} \rangle.M$ and there exists $n\!\ge\!0$ such that $M_{\sigma,d}\!=\!\sigma_{K_r}^{-n}(M_{\sigma,1})$. Moreover,
Theorem \ref{CatRep3}(3) yields $\sigma_{K_r}^{-1}(M_{\sigma,1}) \in \repp(K_r,d)$, whence $n\!\le\!1$. \end{proof}

\bigskip

\begin{Thm}\label{Strat6} Let $\cC \subseteq \Gamma(K_r)$ be a regular component and $M\!:=\!M_{\cC,1}$. 
\begin{enumerate}
\item We have $M_{\sigma,1} \in \{M,\sigma_{K_r}(M)\}$.
\item If $M_{\sigma,1}\!=\!\sigma_{K_r}{(M)}$, then $(M\!\rightarrow)\!=\!\cC\cap \repp(K_r,r\!-\!1)$.
\item If ${M}_{\sigma,1}\!=\!M$, then $(\tau^{-1}_{K_r}(M)\!\rightarrow)\!=\!\cC\cap\repp(K_r,r\!-\!1)$.
\item We have
\[ ({M_{[i]}})|_{\fv} \cong M|_{\fv}\!\oplus\! \Delta_{\tau^{-1}_{K_r} (M_{[i-1]})}(d) P_0(d)\!\oplus\!\dim_\KK \tau^{-1}_{K_r}(M_{[i-1]})_1 P_1(d)\]
for all $i \in \NN_{\geq 2}$, $1\!\leq\!d\!\leq\! r\!-\!1$ and $\fv \in \Gr_d(A_r)$.  \end{enumerate} \end{Thm}

\begin{proof} (1) Since $\sigma^2_{K_r}(M) \not\in \EKP(K_r)\!=\!\repp(K_r,1)$, we have ${M}_{\sigma,1} \in \{M,\sigma_{K_r}(M)\}$.

(2) If $M_{\sigma,1}\!=\!\sigma_{K_r}(M)$, then Theorem \ref{CatRep3} implies $M \in \repp(K_r,r\!-\!1)$. By Proposition \ref{Strat1}, the same is true for every element of $(M\! \rightarrow)\!=\!
\cC \cap \EKP(K_r)$. Consequently, $(M\!\rightarrow) \subseteq \cC\cap \repp(K_r,r\!-\!1) \subseteq \cC\cap \EKP(K_r)\!=\!(M\!\rightarrow)$.

(3) If $M\! =\! M_{\sigma,1}$, then $\sigma_{K_r}(M)\not \in \EKP(K_r)$ and Theorem \ref{CatRep3}(2) shows that $M \not \in \repp(K_r,r\!-\!1)$. On the other hand, Theorem 
\ref{CatRep3}(3) implies  $\tau^{-1}_{K_r}(M) \cong \sigma_{K_r}^{-2}(M) \in \repp(K_r,r\!-\!1)$, so that $\tau^{-1}_{K_r}(M)\!=\!M_{\cC,r-1}$. Consequently, $(\tau^{-1}_{K_r}(M)\!\rightarrow)\!=\! 
\cC\cap\repp(K_r,r\!-\!1)$.

(4) Let $i \in \NN_{\geq 2}$, $1\!\leq\!d\!<\! r$, and $\fv \in \Gr_d(A_r)$. General theory provides a short exact sequence
\[(0) \lra  M \lra M_{[i]} \lra \tau^{-1}_{K_r} (M_{[i-1]}) \lra (0).\]
By virtue of (2), (3) and Proposition \ref{CatRep1}(3), we have  $\tau^{-1}_{K_r}(M_{[i-1]}) \in \repp(K_r,d)$, so that
\[\tau^{-1}_{K_r}(M_{[i-1]})|_{\fv} \cong \Delta_{\tau^{-1}_{K_r} (M_{[i-1]})}(d)P_0(d)\!\oplus\! \dim_\KK \tau^{-1}_{K_r}(M_{[i-1]})_1 P_1(d)\]
is projective. In particular, the above short exact sequence splits upon restriction to $\KK.\fv$. \end{proof}

\bigskip

\section{Homogeneity and uniformity for representations and Steiner bundles}\label{S:UniHom}
Throughout this section, we assume that $r\!\ge\!3$. We will employ AR-theory to obtain results on indecomposable modules and Steiner bundles, whose restrictions to $K_2$ and $\PP^1$ all have the 
same isomorphism type, respectively. We shall see that all preprojective representations as well as ''almost all`` regular indecomposable $\EKP$-representations are uniform in an even stronger sense. As a result, 
uniform representations belonging to $\EKP(K_r)$ and Steiner bundles are usually not homogeneous.

\bigskip

\subsection{Homogeneous modules and bundles} \label{S:HomM}
Given $d \in \{1,\ldots, r\}$, a linear map $\alpha \in \Inj_\KK(A_d,A_r)$ gives rise to an injective morphism $\hat{\alpha} : \PP(A_d) \lra \PP(A_r)$. There results the inverse image functor 
\[ \hat{\alpha}^\ast : \Coh(\PP(A_r)) \lra \Coh(\PP(A_d)),\]
which is known to be right exact, cf.\ \cite[(7.11)]{GW}. The additive functor $\hat{\alpha}^\ast$ commutes with direct sums and the map that sends a fraction $\frac{f}{g}$ of homogeneous polynomials to 
$\frac{f\circ\alpha}{g\circ\alpha}$ induces isomorphisms $\hat{\alpha}^\ast(\cO_{\PP^{r-1}}(j)) \cong \cO_{\PP^{d-1}}(j)$ for $j \in \ZZ$.\footnote{Here we interpret $f$ and $g$ as elements of the algebra $\KK[A_r]\cong 
S(A_r^\ast)$ of polynomial functions on $A_r$.} Moreover, for $\gamma^\ast \in A_r^\ast$, the morphism
\[ \gamma^\ast\cdot : \cO_{\PP^{r-1}}(j) \lra \cO_{\PP^{r-1}}(j\!+\!1) \ \ ; \ \ f \mapsto \gamma^\ast f\]
is being sent to
\[ \hat{\alpha}^\ast(\gamma^\ast\cdot) : \cO_{\PP^{d-1}}(j) \lra \cO_{\PP^{d-1}}(j\!+\!1) \ \ ; \ \ g \mapsto (\gamma^\ast\!\circ\!\alpha)g.\] 
For ease of notation, we will write $\TilTheta\!:=\!\TilTheta_1$ and $\TilTheta_M\!:=\!\TilTheta_{M,d}$.

\bigskip

\begin{Lem} \label{HomM1} Let $M \in \rep(K_r)$. For $\alpha \in \Inj_\KK(A_d,A_r)$ we have
\[ \TilTheta(\alpha^\ast(M)) \cong \hat{\alpha}^\ast(\TilTheta(M)).\] \end{Lem}

\begin{proof} Let $\{\gamma_1,\ldots, \gamma_d\}$ be a basis for $A_d$ and put $\delta_i\!:=\!\alpha(\gamma_i)$. We choose $\delta_{d+1}, \ldots, \delta_r \in A_r$ such that $\{\delta_1,\ldots, \delta_r\}$ 
is a basis for $A_r$. For the corresponding dual bases we obtain 
\[ \delta_i^\ast\circ \alpha = \left\{ \begin{array}{cc} \gamma_i^\ast & i \in \{1,\ldots, d\} \\ 0 & \text{else.} \end{array}\right.\]
We compute the morphism $\hat{\alpha}^\ast(\TilTheta_M)$. Let $U \subseteq \PP^{d-1}$ be open, $g \in \cO_{\PP^{d-1}}(-1)(U)$ and $m \in M$. In view of our foregoing observations we obtain
\begin{eqnarray*}
\hat{\alpha}^\ast(\TilTheta_M)(m\otimes g) & = & \sum_{i=1}^r M(\delta_i)(m)\otimes \hat{\alpha}^\ast(\delta_i^\ast\cdot)(g) = \sum_{i=1}^d M(\delta_i)(m)\otimes (\delta_i^\ast\!\circ\alpha)g \\
& = & \sum_{i=1}^d M(\alpha(\gamma_i))(m)\otimes\gamma_i^\ast g  = \TilTheta_{\alpha^\ast(M)}(m\otimes g),
\end{eqnarray*}
so that 
\[ \hat{\alpha}^\ast(\TilTheta_M) = \TilTheta_{\alpha^\ast(M)}.\]
In view of the right exactness of $\hat{\alpha}^\ast$, this implies
\[ \hat{\alpha}^\ast(\TilTheta(M)) = \hat{\alpha}^\ast(\coker\TilTheta_M) \cong \coker \hat{\alpha}^\ast(\TilTheta_M) = \coker \TilTheta_{\alpha^\ast(M)} = \TilTheta(\alpha^\ast(M)),\]
as desired. \end{proof}

\bigskip
\noindent
Recall that the general linear group $\GL(A_r)$ naturally acts on $A_r$ and $\KK K_r$, on the latter via automorphisms of associative $\KK$-algebras. For every $M \in \modd \KK K_r$ and $g \in \GL(A_r)$, the twisted 
module $M^{(g)}$ has underlying $\KK$-space $M$ and action given by
\[ a\dact m := (g^{-1}(a)).m \ \ \ \ \forall \ g \in \GL_r(A_r), a \in \KK K_r, m \in M.\]
In this fashion, $\GL(A_r)$ acts on $\modd \KK K_r$ and $\rep(K_r)$ via auto-equivalences. For $g \in \GL_r(\KK)\cong\GL(A_r)$ with $g^{-1}\!=\!(a_{ij})_{1\le i,j\le r}$ we obtain
\[ M^{(g)}(\gamma_j) = \sum_{i=1}^ra_{ij}M(\gamma_i) \ \ \ \ \ \ \ \ 1\!\le\!j\!\le\!r.\]
Given $M \in \rep(K_r)$, the stabilizer
\[ \GL(A_r)_M := \{ g \in \GL_r(A_r) \ ;  \ M^{(g)} \cong M\}\]
is a closed subgroup \cite[(2.1)]{Fa11}, and we say that $M$ is {\it homogeneous}, provided $\GL(A_r)_M\!=\!\GL(A_r)$.

The group $\GL_r(A_r)\!=\!\Inj_\KK(A_r,A_r)$ acts on $\PP^{r-1}\!=\!\PP(A_r)$ via automorphisms and hence on $\Coh(\PP^{r-1})$: Given $g \in \GL(A_r)$ and $\cF \in \Coh(\PP^{r-1})$, we define
\[ (g^\ast.\cF)(U) := \cF(g^{-1}.U)\]
for every open subset $U \subseteq \PP^{r-1}$. A coherent sheaf $\cF \in \Coh(\PP^{r-1})$ is called {\it homogeneous}, provided $g^\ast.\cF \cong \cF$ for every $g \in \GL(A_r)$.

\bigskip

\begin{Cor}\label{HomM2} Let $M \in \EKP(K_r)$. Then $M$ is homogeneous if and only if $\TilTheta(M)$ is homogeneous.\footnote{This result can readily be seen to hold for $M \in \repp(K_r,d)$ and $\TilTheta_d(M) \in \Vect(\Gr_d(A_r))$ as well.} \end{Cor}

\begin{proof} In view of $g^\ast.\cF\!=\!\widehat{g^{-1}}^\ast(\cF)$ and $M^{(g)}\!=\!(g^{-1})^\ast(M)$, our assertion follows directly from Lemma \ref{HomM1}. \end{proof}

\bigskip

\begin{Remark} Homogeneous vector bundles of small rank are scarce: Suppose that $\Char(\KK)\!=\!0$ and let $\cF \in \StVect(\PP^{r-1})$ be an indecomposable homogeneous Steiner bundle of rank at most $2r\!-\!3$. 
In view of \cite[Thm.]{BaEl}, we have:
\begin{enumerate}
\item[(a)] $r\!=\!3$ and $\cF \in \{\cO_{\PP^2}, \cT_{\PP^{2}}(-1)\}, S^2 \cT_{\PP^2}(-2)$\}, or
\item[(b)] $r\!\geq\!4$ and $\cF \in \{ \cO_{\PP^{r-1}},\cT_{\PP^{r-1}}(-1)\}$.
\end{enumerate} \end{Remark}

\bigskip
\noindent
Let $\{\gamma^\ast_1,\ldots, \gamma^\ast_r\} \subseteq A_r^\ast$ be the dual basis of $\{\gamma_1,\ldots, \gamma_r\}$. For $M \in \rep(K_r)$ and $i \in \{1,\ldots, r\}$, we consider the linear map
\[ \widehat{\gamma}_i : A_r\!\otimes_\KK\!M_1 \lra M_1 \ \ ; \ \ a\otimes m \mapsto \gamma^\ast_i(a) m.\]
Then we have
\[ \sigma_{K_r}(M)(\gamma_i) = \widehat{\gamma}_i|_{\ker\psi_M}.\]

\bigskip

\begin{Lem} \label{HomM3} Let $M \in \rep(K_r)$.
\begin{enumerate}
\item We have $\sigma_{K_r}(M^{(g)})\cong \sigma_{K_r}(M)^{((g^{-1})^{\tr})}$ for all $g \in \GL(A_r)\cong \GL_r(\KK)$. 
\item $M$ is homogeneous if and only if $\sigma_{K_r}(M)$ is homogeneous. \end{enumerate} \end{Lem}

\begin{proof} (1) For $g \in \GL(A_r)$ we have 
\[ g(\gamma_j) = \sum_{i=1}^r g_{ij}\gamma_i \ \ \ \ \ \ \ 1\!\le\!j\!\le\!r.\]
Direct computation shows that $\ker\psi_{M^{(g)}}\!=\!(g\otimes\id_{M_1})(\ker\psi_M)$.

Setting $\omega_{K_r}(M)\!:=\!(A_r\!\otimes_\KK\!M_1,M_1)$ and 
\[ \omega_{K_r}(M)(\gamma_i)\!=\!\widehat{\gamma_i},\]
we obtain
\begin{eqnarray*}
\omega_{K_r}(M^{(g)})(\gamma_\ell)((g\otimes\id_{M_1})(\gamma_j\otimes m)) & = & \widehat{\gamma_\ell}(\sum_{i=1}^rg_{ij}\gamma_i\otimes m) =  g_{\ell j}m \\
& = & (\sum_{i=1}^r g_{\ell i} \omega_{K_r}(M)(\gamma_i))(\gamma_j\otimes m) \\
& = & \omega_{K_r}(M)^{((g^{-1})^{\tr})}(\gamma_\ell)(\gamma_j\otimes m)
\end{eqnarray*}
for $m \in M_1$.  By the above, the subrepresentation $\sigma_{K_r}(M)\subseteq \omega_{K_r}(M)$ is $(g\otimes\id_{M_1})$-stable, and hence $(g\!\otimes\!\id_{M_1},\id_{M_1})$ induces an isomorphism 
$\sigma_{K_r}(M)^{((g^{-1})^{\tr})} \cong \sigma_{K_r}(M^{(g)})$.
 
(2) This is a direct consequence of (1).  \end{proof}

\bigskip

\subsection{Generalized uniform modules}\label{S:GUM}
Recall from Section \ref{S:Rest} that every $M \in \EKP(K_r)$ affords a generic decomposition $M_{\gen}\!=\!\bigoplus_{i\ge 0} n_i(M)P_i(2)$, with 
\[ O_M := \{ \fv \in \Gr_2(A_r) \ ; \ n_i(M,\fv)\!=\!n_i(M)\}\]
being a dense open subset of $\Gr_2(A_r)$. We say that $M \in \EKP(K_r)$ is {\it uniform}, provided $O_M\!=\!\Gr_2(A_r)$. In \cite[Chap.1,3.3.1]{OSS}, the authors provide a notion
that encompasses those of homogeneity and uniformity. In our context the corresponding property reads as follows:

\bigskip

\begin{Definition} Let $d \in \{1,\ldots, r\}$. We say that a representation $M \in \rep(K_r)$ is \textit{$d$-homogeneous}, provided $\alpha^\ast(M) \cong \beta^\ast(M)$ for all $\alpha,\beta \in \Inj_\KK(A_d,A_r)$. 
\end{Definition}

\bigskip
\noindent
Note that every $d$-homogeneous representation belongs to $\CR(K_r,d)$. (Here we set $\CR(K_r,r)\!:=\!\rep(K_r)$.) For $d\!=\!1$, we have in fact equality. Moreover, $M$ is $r$-homogeneous if and only if it is 
homogeneous, while the uniform $\EKP$-representations correspond to those that are $2$-homogeneous.\footnote{If $M \in \rep(K_r)$ is $2$-homogeneous, then $M \in \CR(K_r)$ has constant rank. The observations at
the end of Section \ref{S:Rest} show that $M$ has a generic decomposition $M_{\gen}\!=\!\bigoplus_{i\ge 0} n_i(M)P_i(2)\!\oplus\!\bigoplus_{j\ge\!0}m_j(M)I_j(2)$. Hence $M$ is $2$-homogeneous if and only if $
\alpha^\ast(M)\cong M_{\gen}$ for all $\alpha \in \Inj_\KK(A_2,A_r)$.}

\bigskip

\begin{Remarks} (1) If $M$ is $d$-homogeneous, then we either have $\cV(K_r,d)_M\!=\!\emptyset$ or $\cV(K_r,d)_M\!=\!\Gr_d(A_r)$. (Here $\cV(K_r,r)_M\!:=\!\emptyset$ in case $M$ is projective and $\cV(K_r,r)_M\!:=\!
\{A_r\}$ otherwise.) 

(2) If $\fv \not \in \cV(K_r,d)_M$, then
\[ M|_\fv \cong \Delta_M(d) P_0(d)\!\oplus\!(\dim_\KK M_1) P_1(d).\] 
Consequently, every $M \in \repp(K_r,d)$ is $d$-homogeneous for $1\!\le\!d\!\le\!r\!-\!1$. 

(3) If $M$ is homogeneous, then $M$ is $d$-homogeneous for all $d \in \{1,\ldots, r\}$. \end{Remarks}

\bigskip

\begin{Prop} \label{GUM1} Let $d \in \{1,\ldots, r\!-\!1\}$. Then there is a brick $M \in \rep(K_r,d)$ such that $M$ not $(d\!+\!1)$-homogeneous. \end{Prop}

\begin{proof} We first assume that $d\!\le\!r\!-\!2$ and pick a pair $(V_1,V_2)$ of $\KK$-vector spaces such that $\udim (V_1,V_2)\!=\!(2d,d(r\!+\!d))$. Then we have
\begin{enumerate}
\item[(a)] $\Delta_{(V_1,V_2)}(d)\!=\!d(r\!-\!d)$, and
\item[(b)] $\Delta_{(V_1,V_2)}(d\!+\!1)\!\ge\!0$. \end{enumerate}
We consider $\rep(K_r;V_1,V_2)\!=\!\Hom_\KK(A_r\!\otimes_\KK\!V_1,V_2)$. Given $\fv \in \Gr_{d+1}(A_r)$, the restriction map
\[ \res_\fv : \Hom_\KK(A_r\!\otimes_\KK\!V_1,V_2) \lra \Hom_\KK(\fv\!\otimes_\KK\!V_1,V_2) \  \  ;  \ \ \psi \mapsto \psi|_{\fv\otimes_\KK V_1}\]
is linear and hence a morphism. Consequently,
\[ \cO_{\fv} := \res_\fv^{-1}(\Inj_\KK(\fv\!\otimes_\KK\!V_1,V_2))\]
is an open subset of $\rep(K_r;V_1,V_2)$, which, by virtue of (b), is not empty. In view of (a), Proposition \ref{CatRep4} implies $\repp(K_r,d)\cap\cO_\fv\ne \emptyset$. Theorem \ref{Fam5} thus provides $M \in \repp(K_r,d)
\cap\rep(K_r;V_1,V_2)$ such that $\fv \not \in \cV(K_r,d\!+\!1)_M$. As $M \in \repp(K_r,d)$ has minimal type, Corollary \ref{MinType3}(2b,2d) and Corollary \ref{MinType4} ensure that $M \not \in \rep(K_r,d\!+\!1)$ is 
a brick. In particular, $\cV(K_r,d\!+\!1)_M\ne\emptyset$. 

Since $M$ belongs to $\repp(K_r,d)$, it is clearly $d$-homogeneous. In view of $\emptyset \ne \cV(K_r,d\!+\!1)_M\!\ne\!\Gr_{d+1}(A_r)$, there are  $\alpha, \beta\in \Inj_\KK(A_{d+1},A_r)$ such that 
$\alpha^\ast(M)$ is projective, while $\beta^\ast(M)$ is not. Consequently, $M$ is not $(d\!+\!1)$-homogeneous. 

If $d\!=\!r\!-\!1$, we pick the Schwarzenberger module $M_\cS[r\!+\!2] \in \EKP(K_r)$. This module has minimal type and Corollary \ref{MinType3}(2b,2d) implies that $M$ is a regular brick. In view of Corollary \ref{GCD6} 
and its succeeding remark $M_\cS[r\!+\!2]$ is not homogeneous. Thanks to Theorem \ref{CatRep3}, Lemma \ref{HomM3} and $\sigma_{K_r}$ being an auto-equivalence on regular modules, the representation 
$N\!:=\!\sigma_{K_r}^{-1}(M) \in \repp(K_r,r\!-\!1)$ is an $(r\!-\!1)$-homogeneous brick, which is not $r$-homogeneous. \end{proof}

\bigskip
\noindent
Every $\alpha \in \Inj_\KK(A_d,A_r)$ gives rise to a linear embedding $\hat{\alpha} : \PP^{d-1} \hookrightarrow \PP^{r-1}$. 

\bigskip

\begin{Definition} Let $d \in \{0,\ldots, r\!-\!1\}$. A vector bundle $\cF \in \Vect(\PP^{r-1})$ is said to be {\it $d$-homogeneous}, provided
\[ \hat{\alpha}^\ast(\cF) \cong \hat{\beta}^\ast(\cF)\]
for all $\alpha, \beta \in \Inj_\KK(A_{d+1},A_r)$. \end{Definition}

\bigskip
\noindent
It directly follows from Lemma \ref{HomM1} that $M \in \EKP(K_r)$ is $d$-homogeneous if and only if $\TilTheta(M)$ is $(d\!-\!1)$-homogeneous.

A closed subset $L \subseteq \PP^{r-1}$ is called a {\it line}, provided $L\!=\!\im\hat{\alpha}$ for some $\alpha \in \Inj_\KK(A_2,A_r)$. Let $\iota : L \hookrightarrow \PP^{r-1}$ be the inclusion. Given 
$\cF \in \Vect(\PP^{r-1})$, we put 
\[\cF|_L:= \iota^\ast(\cF).\] 
Since vector bundles on $\PP^1$ are homogeneous, we have
\[ \cF|_L \cong \hat{\alpha}^\ast(\cF) \ \ \ \ \ \ \text{whenever} \ \ L= \hat{\alpha}(\PP^1).\] 

\bigskip

\begin{Definition} A vector bundle $\cF \in \Vect(\PP^{r-1})$ is said to be {\it uniform}, if there are $a_i(\cF) \in \NN_0$ such that
\[ \cF|_L \cong \bigoplus_{i\ge 0} a_i(\cF)\cO_L(i)\]
for every line $L \subseteq \PP^{r-1}$. \end{Definition}

\bigskip
\noindent
It follows that $\cF$ is $(r\!-\!1)$-homogeneous if and only if it is homogeneous. Moreover, $\cF$ is uniform if and only if $\cF$ is $1$-homogeneous.  

\bigskip

\begin{Example} Since the projective representation $P_1(3)$ is homogeneous, so is  $M\!:=\!P_2(3) \cong \sigma_{K_r}^{-1}(P_1(3))$, cf.\ Lemma \ref{HomM3}. Then $M$ has dimension vector $(3,8)$ and Proposition 
\ref{CatRep1}(2) ensures that $M \in \repp(K_3,2)$. Consequently, $M$ is uniform with $M_{\gen}\!=\!\Delta_M(2)P_0(2)\!\oplus\!(\dim_\KK M_1)P_1(2)$. Note that $\cF\!:=\!\TilTheta_1(M)$ is an indecomposable Steiner 
bundle such that $c_1(\cF)\!=\!3$ and $\rk(\cF)\!=\!5$. 

Let $N \in \repp(K_3,2)$ be such that $\udim N\!=\!(3,8)$. Since $\Delta_N(2)\!=\!2$, (2b) and (2d) of Corollary \ref{MinType3} ensure that $N$ is a brick and not projective. Consequently, $\udim \sigma_{K_3}(N)\!=\!(1,3)$, 
so that $\sigma_{K_3}(N)\cong P_1(3)$ and $N\cong P_2(3)\!=\!M$. As a result, $\cF$ is isomorphic to the dual of the bundle constructed in \cite[Ex.3.1]{MMR21}. 

Corollary \ref{HomM2} and Lemma \ref{Fun1} imply that $\cF$ is homogeneous (and hence uniform) such that 
\[\cF|_L \cong 2\cO_L\!\oplus 3\cO_L(1)\] 
for every line $L \subseteq \PP^2$. \end{Example}

\bigskip
\noindent
According to \cite[Chap.1, 3.3.2]{OSS}, there exists, for every $d\!\in \{1,\ldots,\!r\!-\!2\}$, a holomorphic vector bundle $\cF_d$ on $\PP^{r-1}(\CC)$ such that $\cF_d$ is $(d\!-\!1)$-homogeneous but not $d$-homogeneous.
The following result, which slightly strengthens that of loc.\ cit., implies that such examples may also be found within the category of Steiner bundles on $\PP^{r-1}$. 

\bigskip

\begin{Cor} \label{GUM2} Given $d \in \{1,\ldots, r\!-\!1\}$, there exists a simple Steiner bundle $\cF_d$ that is $(d\!-\!1)$-homogeneous, but not $d$-homogeneous. \end{Cor}

\begin{proof} Proposition \ref{GUM1} provides a brick $M \in \rep(K_r,d)$ such that $M$ is not $(d\!+\!1)$-homogeneous. Consequently, the bundle $\cF_d\!:=\!\TilTheta(M)$ has the requisite properties. \end{proof}

\bigskip
\noindent
In \cite{MMR21} the authors introduced the notion of uniform vector bundles of $k$-type and studied the case $k\!=\!1$ in detail.\footnote{The vector bundles studied in \cite{MMR21} are dual to those considered here.} 
Given $d \in \{1,\ldots, r\!-\!1\}$ and $k \in \NN_0$, we say that a $d$-homogeneous Steiner bundle $\cF$ is of { \it $k$-type}, provided there are $n_0, \ldots, n_k \in \NN_0$, $n_k\!\ne\!0$ such that
\[ \alpha^\ast(\cF) \cong \bigoplus_{i=0}^k \cE_i[d\!+\!1]^{n_i}\] 
for all $\alpha \in \Inj_\KK(A_{d+1},A_r)$. Here the $\cE_i[d\!+\!1] \in \StVect(\PP(A_{d+1}))$ are the exceptional Steiner bundles discussed in Section \ref{S:ExSt}. 

Our next result implies that, for $k\!=\!1$, we obtain the essential image of $\repp(K_r,d\!+\!1)$.

\bigskip

\begin{Lem} \label{GUM3} Let $d \in \{2,\ldots, r\!-\!1\}$, $\cF \in \StVect(\PP^{r-1})$ be a Steiner bundle. Then the following statements are equivalent:
\begin{enumerate}
\item $\cF \in \TilTheta(\repp(K_r,d))$.
\item $\hat{\alpha}^\ast(\cF)\cong \cO^{c_1(\cF)}_{\PP(A_d)}\!\oplus\!\cT_{\PP(A_d)}^{\rk(\cF)-(d-1)c_1(\cF)}(-1)$ for every $\alpha \in \Inj_\KK(A_d,A_r)$. \end{enumerate} \end{Lem}

\begin{proof} (1) $\Rightarrow$ (2). By assumption, there exists $M \in \repp(K_r,d)$ such that $\cF \cong \TilTheta(M)$. Given $\alpha \in \Inj_\KK(A_d,A_r)$, it follows that $\alpha^\ast(M)$ is projective. Observing Lemma 
\ref{HomM1}, we thus obtain 
\begin{eqnarray*} 
\hat{\alpha}^\ast(\TilTheta(M)) & \cong & \TilTheta(\alpha^\ast(M)) \cong \TilTheta(P_0(d))^{\Delta_M(d)}\!\oplus\!\TilTheta(P_1(d))^{\dim_\KK M_1}  \\ 
                                      & \cong & \cO_{\PP(A_d)}^{\rk(\cF)-(d-1)c_1(\cF)}\!\oplus\!\cT_{\PP(A_d)}^{c_1(\cF)}(-1).
\end{eqnarray*}

(2) $\Rightarrow$ (1). There is $M \in \EKP(K_r)$ such that $\cF \cong \TilTheta(M)$. The foregoing calculation in conjunction with Lemma \ref{HomM1} and Theorem \ref{Fun3} now shows that
\[ \alpha^\ast(M) \cong \Delta_M(d) P_0(d)\!\oplus\!(\dim_\KK M_1)P_1(d)\]
for all $\alpha\in \Inj_\KK(A_d,A_r)$, so that each $\alpha^\ast(M)$ is projective. It follows that $M \in \repp(K_r,d)$. \end{proof}

\bigskip

\begin{Cor} \label{GUM4} Let $\cC \subseteq \Gamma(K_r)$ be a regular AR-component. 
\begin{enumerate}
\item Every $M \in (\tau_{K_r}^{-1}(M_{\cC,1})\!\rightarrow)$ is $(r\!-\!1)$-homogeneous.
\item If $M \in (\tau_{K_r}^{-1}(M_{\cC,1})\!\rightarrow)$, then $\TilTheta(M)$ is $(r\!-\!2)$-homogeneous of $1$-type.
\item Let $M_{\cC,1}$ be $d$-homogeneous for some $d\!\le\!r$.
\begin{enumerate}
\item Every $M \in (M_{\cC,1}\!\rightarrow)$ is $d$-homogeneous.
\item Let $d\!\ge\!2$. If $\TilTheta(M_{\cC,1})$ is $(d\!-\!1)$-homogeneous of $k$-type, then $\TilTheta(M)$ is $(d\!-\!1)$-homogeneous of $k$-type for every $M \in (M_{\cC,1}\rightarrow) \!\smallsetminus\!(\tau_{K_r}^{-1}
(M_{\cC,1})\rightarrow)$.\end{enumerate}\end{enumerate} \end{Cor}

\begin{proof} (1) In view of Theorem \ref{CatRep3}(3), we have $(\tau_{K_r}^{-1}(M_{\cC,1})\!\rightarrow) \subseteq \repp(K_r,r\!-\!1)$, so that every $M \in (\tau_{K_r}^{-1}(M_{\cC,1})\!\rightarrow)$ is in particular 
$(r\!-\!1)$-homogeneous.

(2) This is a direct consequence of (1), Lemma \ref{GUM3} and Lemma \ref{HomM1}.

(3a) If $d\!=\!r$, then $M$ is homogeneous, and our assertion follows from \cite[(2.3)]{Fa11} in conjunction with $\GL(A_r)$ being connected. Alternatively, Theorem \ref{Strat6}(3) implies that $(M_{\cC,1})_{[i]}$ is 
$d$-homogeneous for every $i\!\ge\!2$. The result thus follows from (1). 

(3b) Since $\TilTheta(M_{\cC,1})$ $(d\!-\!1)$-homogeneous of $k$-type, Theorem \ref{ExSt2}(2) in conjunction with Lemma 
\ref{HomM1} and Theorem \ref{Fun3}(2) implies that there are $n_0,\ldots, n_k \in \NN_0$, $n_k\!\ne\!0$ such that  
\[ \alpha^\ast(M_{\cC,1}) \cong \bigoplus_{i=0}^k n_i P_i(d)  \ \ \ \ \  \ \forall \ \alpha \in \Inj_\KK(A_d,A_r).\]
Parts (3) of Theorem \ref{Strat6} now shows that for every $M \in (M_{\cC,1}\rightarrow) \!\smallsetminus\!(\tau_{K_r}^{-1}(M_{\cC,1})\rightarrow) $ its restriction $\alpha^\ast(M)$ can be written in such a way. 
This implies (b). \end{proof}

\bigskip

\begin{Remark} If $d\!\le\!r\!-\!2$, then Proposition \ref{GUM1} provides an indecomposable representation $M \in \rep(K_r,d)$ such that $M$ is not $(d\!+\!1)$-homogeneous. Let $\cC_M$ be the AR-component containing $M$. As ``most'' vertices of the cone $\cC_M \cap\repp(K_r,d)$ are $(d\!+\!1)$-homogenous, we see that the natural analogue of \cite[(2.3)]{Fa11} does not hold in this context. \end{Remark}

\bigskip
\noindent
In 1960 Schwarzenberger \cite[(1.2)]{Sc61b} raised the question, whether all uniform vector bundles on $\PP^{r-1}$ are homogeneous. The first counterexample was given by Elencwajg \cite{El79} in 1979. The following 
result implies that every indecomposable inhomogeneous Steiner bundle gives rise to families of inhomogeneous, but $(r\!-\!2)$-homogeneous indecomposable vector bundles of $1$-type.

\bigskip

\begin{Cor} \label{UM5} Let $\cC \subseteq \Gamma(K_r)$ be a component that contains a representation $M \in \cC$ that is not homogeneous. Then every $X \in (\tau_{K_r}^{-1}(M_{\cC,1})\!\rightarrow)$ is 
$(r\!-\!1)$-homogeneous, but not homogeneous. \end{Cor} 

\begin{proof} Since preprojective and preinjective representations are homogeneous, it follows that $\cC$ is regular.  By Corollary \ref{GUM4}, the cone $(\tau_{K_r}^{-1}(M_{\cC,1})\!\rightarrow) \subseteq \repp(K_r,r\!-\!1)$ 
consists of $(r\!-\!1)$-homogeneous representations. As $M \in \cC$ is not homogeneous and $\GL(A_r)$ is connected, \cite[(2.3)]{Fa11} ensures that all representations in $\cC$ are inhomogeneous. \end{proof}

\bigskip

\begin{Remark} In the situation above, every $X \in (\tau_{K_r}^{-1}(M_{\cC,1})\!\rightarrow)$ gives rise to a family of Steiner bundles that are not homogeneous but $(r\!-\!2)$-homogeneous: In view of \cite{Fa11}, 
\[ \GL(A_r)_X := \{ g \in \GL(A_r) \ ; \ X^{(g)} \cong X\}\]
is a closed, proper subgroup of $\GL(A_r)$, so that the variety $\GL(A_r)/\GL(A_r)_X$ has dimension $\ge\!1$. Hence there are at least $|\KK|$ pairwise non-isomorphic inhomogeneous but $(r\!-\!1)$-homogeneous 
representations of dimension vector $\udim X$. Upon application of $\TilTheta$, we obtain a family of Steiner bundles, whose members are inhomogeneous but $(r\!-\!2)$-homogeneous. \end{Remark} 

\bigskip

\subsection{Constructions} \label{S:Const} In this section we turn to the existence of modules and bundles with certain prescribed properties. This contributes to understanding the conceptual 
sources of several computations and ad hoc arguments in \cite{MMR21}.    

\bigskip

\begin{Cor} \label{Const1} Let $d\!\ge\!2$. For every $r\!\ge\!d\!+\!1$ and $n\!\ge\!d\!+\!2$, there exists $N \in \repp(K_r,d)$ such that
\begin{enumerate}
\item[(a)] $\dim_\KK N_1\!=\!n$, and
\item[(b)] $N$ is not homogeneous and of minimal type. \end{enumerate} \end{Cor}

\begin{proof} We proceed by induction on $r$, starting with $r\!=\!d\!+\!1$. Let $n\!\ge\!d\!+\!2$, so that $\ell\!:=\!n\!-\!r\!+\!1\!\ge\!2$. We consider the Schwarzenberger module $M\!:=\!M_\cS[\ell+r] \in \EKP(K_r)\!=\!
\repp(K_r,1)$ of dimension vector $\udim M\!=\!(\ell,\ell\!+\!r\!-\!1)$. Thanks to Theorem \ref{CatRep3} the representation $N\!:=\!\sigma_{K_r}^{-1}(M)$ belongs to $\repp(K_r,d)$. It has dimension vector $\udim N\!=\!(n,(r\!-\!1)n\!+\!d)$ and 
hence is of minimal type. Since $\ell\!\ge\!2$, Corollary \ref{GCD6} and its succeeding remark imply that $M$ is not homogeneous. By Lemma \ref{HomM3}, the representation $N$ is also not homogeneous. 

Suppose the statements hold for some $r\!\ge\!d\!+\!1$. Given $n\!\ge\!d\!+\!2$, there is $M \in \repp(K_r,d)$ of minimal type such that (a) and (b) hold. Proposition \ref{MinType7} provides $N \in \repp(K_{r+1},d)$ of
minimal type such that $N|_{K_r} \cong M\!\!\oplus\!d P_0(r)$. In particular, (a) holds. If $N$ is homogeneous, then its restriction $N|_{K_r}$ inherits this property. As $\GL(A_r)$ is connected, this then holds for every 
direct summand of $N|_{K_r}$, a contradiction.\footnote{One can show that the category of uniform representations is closed under taking direct summands. Hence one can also find $N$ of minimal type that is not 
uniform.} \end{proof} 

\bigskip

\begin{Remark} Upon application of $\TilTheta_d$, Corollary \ref{Const1} yields a generalization of \cite[(4.1)]{MMR21}: One obtains inhomogeneous vector bundles $\cF_d \in \StVect(\Gr_d(A_r))$ with $c_1(\cF_d)\!=\!
n\sigma_1$ and $\rk(\cF_d)\!=\!\dim \Gr_d(A_r)$. \end{Remark}

\bigskip
\noindent
We provide a conceptual proof for the following generalization of  \cite[(3.5)]{MMR21}:

\bigskip

\begin{Prop} \label{Const2} Let $d \in \{2,\ldots,r\!-\!1\}$. The following statements hold: 
\begin{enumerate} 
\item Let $\cF \in \StVect(\PP^{r-1})$ be a Steiner bundle.
\begin{enumerate}
\item If $c_1(\cF)\!\ge\!d$ and $\cF$ is $(d\!-\!1)$-homogeneous of $1$-type, then $\rk(\cF)\!\ge\!d(r\!-\!d)\!+\!(d\!-\!1)c_1(\cF)$.
\item If $\cO_{\PP^{r-1}}$ is not a direct summand of $\cF$, then $\rk(\cF)\!\le\!c_1(\cF)(r\!-\!1)$. \end{enumerate}
\item For each pair $(c,\ell) \in \NN^2$ such that $\ell\!\ge\!d(r\!-\!d)\!+\!c(d\!-\!1)$ there is $\cF \in \StVect(\PP^{r-1})$ which is $(d\!-\!1)$-homogeneous of $1$-type and such that $(c_1(\cF),\rk(\cF))\!=\!(c,\ell)$.
\end{enumerate} \end{Prop}

\begin{proof} (1) Let $M \in \EKP(K_r)$ be such that $\TilTheta(M)\cong\cF$. 

(a) Lemma \ref{GUM3} implies $M\in\repp(K_r,d)$ with $\dim_\KK M_1\!=\!c_1(\cF)\!\ge\!d$ and $\Delta_M\!=\!\rk(\cF)$. Theorem \ref{MinType2} now yields $\rk(\cF)\!=\!\Delta_M(d)\!+\!(d\!-\!1)\dim_\KK M_1\!\ge 
\!d(r\!-\!d)\!+\!(d\!-\!1)\dim_\KK M_1\!=\!d(r\!-\!d)\!+\!(d\!-\!1)c_1(\cF)$. 

(b) Since $\cO_{\PP^{r-1}}$ is not a direct summand of $\cF$, the map $\psi_M$ is surjective, so that there is a projective resolution
\[ (0) \lra -\Delta_M(r)P_0(r) \lra (\dim_\KK M_1)P_1(r) \lra M \lra (0).\]
In particular, $\Delta_M(r)\!\le\!0$, whence $\rk(\cF)\!=\!\Delta_M\!=\!\Delta_M(r)\!+\!(r\!-\!1)\dim_\KK M_1\!\le\! c_1(\cF)(r\!-\!1)$. 

(2) Let $(V_1,V_2)$ be a pair of $\KK$-vector spaces such that $\dim_\KK V_1\!=\!c$ and $\dim_\KK V_2\!=\!c\!+\!\ell$. By assumption, we have $\Delta_{(V_1,V_2)}(d)\!=\!\ell\!-\!(d\!-\!1)c\!\ge\!d(r\!-\!d)$ and 
Proposition \ref{CatRep4} provides $M \in \rep(K_r;V_1,V_2)\cap\repp(K_r,d)$. Hence $\cF\!:=\!\TilTheta(M)$ has the requisite properties. \end{proof}

\bigskip
\noindent
We continue by providing a general method of constructing uniform $1$-type Steiner bundles which are not homogeneous relative to certain subgroups of $\GL(A_r)$. 

Let $H \subseteq \GL(A_r)$ be a subgroup. Then $M \in \rep(K_r)$ is called {\it $H$-homogeneous}, provided
\[ M^{(h)} \cong M \ \ \ \ \ \forall \ h \in H.\]
There is an analogous notion of $H$-homogeneous vector bundles on $\PP^{r-1}$.

For $s\!<\!r$, we put $A_s\!:=\!\bigoplus_{j=1}^s\KK\gamma_j$ and $A_{r-s}\!:=\!\bigoplus_{j=s+1}^r\KK\gamma_j$, so that $A_r\!=\!A_s\!\oplus\!A_{r-s}$. Given $M \in \rep(K_s)$, we denote by $\Inf(M) \in \rep(K_r)$ the 
representation with underlying pair of spaces $(M_1,M_2)$ and structure map
\[ \psi_{\Inf(M)}((a\!+\!b)\otimes m) := \psi_M(a\otimes m) \ \ \ \ \ \ a \in A_s, b \in A_{r-s}, m \in M_1.\]
Equivalently,
\[\Inf(M)(\gamma_i) = \left\{\begin{array}{cc} M(\gamma_i) & 1\!\le\!i\! \le\! s \\ 0 & s\!+\!1\!\le\!i\!\le\! r.\end{array} \right.\] 
Note that
\[ \iota : \GL(A_s) \lra \GL(A_r) \ \ ; \ \ h \mapsto h\!\oplus\!\id_{A_{r-s}}\]
is an injective homomorphism of algebraic groups such that
\[ (\ast) \ \ \ \ \ \Inf(M)^{\iota(h)} = \Inf(M^{(h)}) \ \ \ \ \ \forall \ h \in \GL(A_s).\]
In the sequel, we shall identify $\GL(A_s)$ with its image under $\iota$ in $\GL(A_r)$.

\bigskip

\begin{Remark} If $N \in \rep(K_s)$ is homogeneous, then ($\ast$) shows that $\Inf(N)$ is $\GL(A_s)$-homogeneous. However, $\Inf(N)$ is usually not $s$-homogeneous. \end{Remark}  

\bigskip

\begin{Lem} \label{Const3} Let $d \in \{1,\ldots, r\!-\!1\}$. The following statements hold:
\begin{enumerate}
\item For $\fv \in \Gr_d(A_r)$, we have $\udim E(\fv)\!=\!(d,rd\!-\!1)$, $\udim \sigma_{K_r}(E(\fv))\!=\!(1,d)$ and $\udim \tau_{K_r}(E(\fv))\!=\!(r\!-\!d,1)$. 
\item Let $M \in \rep(K_r)$ be regular such that $\cV(K_r,d)_M\!\ne\!\Gr_d(A_r)$.
\begin{enumerate}
\item We have $\tau^{-1}_{K_r}(M) \in \repp(K_r,d)$. 
\item If $d\!>\!1$, then $\sigma_{K_r}^{-1}(M) \in \repp(K_r,d\!-\!1)$. 
\item If $d\!>\!1$, then $\tau_{K_r}^{-1}(M)\in \repp(K_r,r\!-\!1)$.\end{enumerate} \end{enumerate} \end{Lem}

\begin{proof} (1) Let $\alpha \in \Inj_\KK(A_d,A_r)$ be such that $\im \alpha\!=\!\fv$. Directly from the definition we obtain 
\[ \udim D_{K_r}(E(\fv)) = \udim\tau_{K_r}(\coker\bar{\alpha}) = \Phi_r\dact(1,r\!-\!d) = (rd\!-\!1,d),\]
whence $\udim E(\fv)\!=\!(d,rd\!-\!1)$. Writing $\sigma_r(a,b)\!:=\!(ra\!-\!b,a)$, we have $\udim \sigma_{K_r}(N)\!=\!\sigma_r(\udim N)$ for every $N \in \rep(K_r)$. The remaining assertions now follow by applying $\sigma_r$ 
to $(d,rd\!-\!1)$ and $(1,d)$, respectively.

(2) In view of Theorem \ref{Fam5}, our current assumption provides $\fv \in \Gr_d(A_r)$ with $\Hom_{K_r}(E(\fv),M)$ $=\!(0)$.

(a) Let $\fu \in \Gr_d(A_r)$ be such that $\Hom_{K_r}(E(\fu),\tau_{K_r}^{-1}(M))\!\ne\!(0)$. Then $\Hom_{K_r}(\tau_{K_r}(E(\fu)),M)\!\ne\!(0)$ and since
$d\dim_\KK \tau_{K_r}(E(\fu))_1\!=\!d(r\!-\!d)\!>\!1\!=\!\dim_\KK\tau_{K_r}(E(\fu))_2$, an application of Proposition \ref{Fam4} ensures that
\[ 0 \ne \dim_\KK\ker\psi_{\tau_{K_r}(E(\fu)),\fv} = \dim_\KK\Hom_{K_r}(E(\fv),\tau_{K_r}(E(\fu)).\]
Since $\tau_{K_r}(E(\fu))$ is elementary, while $E(\fv)$ and $M$ are regular, it follows from \cite[(2.1.1)]{Bi20} that $\Hom_{K_r}(E(\fv),M)\!\ne\!(0)$, a contradiction. 

(b) This is analogous to (a). 

(c) Thanks to (b), we have $\sigma_{K_r}(\tau^{-1}_{K_r}(M))\cong \sigma^{-1}_{K_r}(M) \in \repp(K_r,d\!-\!1) \subseteq \EKP(K_r)$. The assertion now follows from Theorem \ref{CatRep3}. \end{proof}

\bigskip
\noindent
Our next result produces $(s\!-\!1)$-homogeneous $1$-type vector bundles $\cF$ that fail to be $\GL(A_s)$-homogeneous. 

\bigskip

\begin{Prop} \label{Const4} Let $r\!>\!s\!>\!d\geq\!1$. Suppose that $M \in \rep(K_s)$ is an indecomposable, non-simple representation such that $\cV(K_s,d)_M\!\ne\!\Gr_d(A_s)$.
Then the following statements hold:
\begin{enumerate}
\item $\Inf(M)$ and $\sigma^{-3}_{K_r}(\Inf(M))$ are regular, quasi-simple representations.
\item $\sigma^{-2}_{K_r}(\Inf(M)) \in \repp(K_r,d)$.
\item We have $(\sigma^{-3}_{K_r}(\Inf(M))\!\rightarrow) \subseteq \repp(K_r,r\!-\!1)$. For each $N \in (\sigma^{-3}_{K_r}(\Inf(M))\!\rightarrow)$, the Steiner bundle $\TilTheta(N) \in \StVect(\PP^{r-1})$ is $(r\!-\!2)$-homogenous 
of $1$-type, but not homogeneous. 
\item Suppose that $M$ is not homogeneous. For each $N \in (\sigma^{-3}_{K_r}(\Inf(M))\!\rightarrow)$, the Steiner bundle $\TilTheta(N) \in \StVect(\PP^{r-1})$ is $(s\!-\!1)$-homogeneous of $1$-type, but not 
$\GL(A_s)$-homogeneous. \end{enumerate} \end{Prop}

\begin{proof} (1) By \cite[(3.2.1), (3.2.2)]{Bi20}, the representation $\Inf(M) \in \rep(K_r)$ is indecomposable, regular and quasi-simple. Therefore, the same holds for $\sigma^{-3}_{K_r}(\Inf(M))$.

(2) By assumption, there is $\alpha \in \Inj_\KK(A_d,A_s)$ such that $\alpha^\ast(M)$ is projective. Let $\iota : A_s \lra A_r$ be the inclusion. Then $\iota\circ\alpha \in \Inj_\KK(A_d,A_r)$ and $(\iota\circ \alpha)^\ast(\Inf(M)) \cong \alpha^\ast(\iota^\ast(\Inf(M))) \cong \alpha^\ast(M)$, so that $\im (\iota\circ \alpha) \not \in \cV(K_r,d)_{\Inf(M)}$. The assertion now follows from (1) and Lemma \ref{Const3}(2a). 

(3) By (1),(2) and Theorem \ref{CatRep3}(3) we know that $\sigma_{K_r}(\sigma^{-3}_{K_r}(\Inf(M))) \in \EKP(K_r)$ is quasi-simple. Now Theorem \ref{CatRep3}(2) yields $\sigma^{-3}_{K_r}(\Inf(M)) 
\in \repp(K_r,r\!-\!1)$. We conclude with Theorem \ref{Strat6} that $(\sigma^{-3}_{K_r}(\Inf(M))\!\rightarrow) \subseteq \repp(K_r,r\!-\!1)$. As $\Inf(M)(\gamma_r)\!=\!0$, the 
representation $\Inf(M)$ is not homogeneous. Hence Lemma \ref{HomM3} shows that $\sigma^{-3}_{K_r}(\Inf(M))$ enjoys the same property. Consequently, every representation belonging to the (regular) component of 
$\sigma^{-3}_{K_r}(\Inf(M))$ is not homogeneous. In conclusion, each bundle $\TilTheta (N) \in \StVect(\PP^{r-1})$ with $N \in (\sigma^{-3}_{K_r}(\Inf(M))\!\rightarrow)$ is $(r\!-\!2)$-homogeneous of $1$-type and not 
homogeneous. 

(4) The arguments of Corollary \ref{UM5} in conjunction with ($\ast$) show that the representations belonging to $(\sigma^{-3}_{K_r}(\Inf(M))\!\rightarrow) \subseteq \repp(K_r,r\!-\!1)$ are $(r\!-\!1)$-homogeneous and therefore belong to $\repp(K_r,s).$ Hence the corresponding Steiner bundles are $(s\!-\!1)$-homogeneous of $1$-type.  Since $\GL(A_s)$ is connected, none of these representations is  $\GL(A_s)$-homogeneous. \end{proof}

\bigskip

\begin{Example} Let $s\!=\!2$ and consider $M\!=\!(\KK,\KK,(\id_\KK,\id_\KK)) \in \rep(K_2)$. Then $\cV(K_2,1)_M\!\ne\!\PP^1$ and $M$ is indecomposable, non-simple and not homogeneous. Proposition \ref{Const4} 
thus provides Steiner bundles $\cF \in \StVect(\PP^{r-1})$ that are uniform but not $\GL(A_2)$-homogeneous.  \end{Example}

\bigskip

\section{Stability} \label{S:Stab}
Throughout this section, we assume that $r\!\ge\!3$. Let $\cF \in \Coh(\PP^{r-1})$ be a torsion-free sheaf. By definition, the $\cO_{\PP^{r-1},x}$-module $\cF_x$ is torsion-free for every $x \in \PP^{r-1}$. If $\cF\!\ne\!(0)$,
then $\rk(\cF)\!\ne\!0$ and we denote by
\[ \mu(\cF) := \frac{c_1(\cF)}{\rk(\cF)}\]
the {\it slope} of $\cF$. Recall that $\cF$ is called {\it semistable} ({\it stable}), provided 
\[ \mu(\cG) \le \mu(\cF) \ \ \ \ (\mu(\cG) < \mu(\cF))\]
for every proper subsheaf $(0) \subsetneq \cG \subsetneq \cF$ (such that $\rk(\cG)\!<\!\rk(\cF)$).

In this section, we relate the stability of Steiner bundles on $\PP^{r-1}$ to that of the corresponding representations of $\EKP(K_r)\!=\!\repp(K_r,1)$. Let $M \in \EKP(K_r)$. As before, write $\TilTheta(M)\!:=\!\TilTheta_1(M)
$ and recall the defining exact sequence
\[ (0) \lra \widetilde{M_1} \stackrel{\TilTheta_M}{\lra} \widetilde{M_2} \lra \TilTheta(M) \lra (0),\]
where $\widetilde{M_i}\!=\!M_i\!\otimes_\KK\!\cO_{\PP^{r-1}}(i\!-\!2)$ for $i \in \{1,2\}$. We continue to write $\Delta_M\!:=\!\Delta_M(1)$, so that Corollary \ref{MinType3} yields $\Delta_M\!\ge\!r\!-\!1$ in case $M_1\!\ne\!(0)$.

We begin by recalling the See-Saw Lemma:

\bigskip

\begin{Lemma}\label{Stab1} Let $(0) \lra \cE \lra \cF \lra \cG \lra (0)$ be a short exact sequence of non-zero torsion-free coherent sheaves. Then we have $\mu(\cE)\! \leq\! \mu(\cF)\!\leq\!\mu(\cG)$ or $\mu(\cG)\!
\leq\!\mu(\cF)\!\leq\!\mu(\cE)$. Moreover, $\mu(\cH)\!=\!\mu(\cF)$ for $\cH \in \{\cE,\cG\}$ implies $\mu(\cE)\!=\!\mu(\cF)\!=\!\mu(\cG)$. \end{Lemma}

\begin{proof} The Whitney sum formula implies
\[\mu(\cF)\!=\!\frac{\rk(\cE) \mu(\cE)+\rk(\cG) \mu(\cG)}{\rk(\cF)}.\]
If $\mu(\cF)\!\le \!\mu(\cE)$ and $\mu(\cF)\!< \!\mu(\cG)$, the foregoing formula yields
\[ \mu(\cF) = \frac{\rk(\cE)\!+\!\rk(\cG)}{\rk(\cF)} \mu(\cF) < \frac{\rk(\cE) \mu(\cE)\!+\!\rk(\cG) \mu(\cG)}{\rk(\cF)} = \mu(\cF),\]
a contradiction. Thus, $\mu(\cG)\!\le\!\mu(\cF)$, and the other assertions follow analogously.\end{proof}

\bigskip

\subsection{Stability for $\EKP$-representations and Steiner bundles} \label{S:StabE} The main result of this section recasts \cite[(5.1)]{Kar93}, which was formulated for $\KK\!=\!\CC$, into our set-up.

\bigskip

\begin{Definition} For $(0)\! \neq \!M \in \EKP(K_r)$, we define 
\[\mu(M) := \frac{\dim_\KK M_1}{\Delta_M}.\] 
We say that $M$ is {\it stable}, provided $\mu(N)\! <\!\mu(M)$ for every subrepresentation $(0) \subsetneq N \subsetneq M$. We say that $M$ is {\it semistable}, provided $\mu(N)\! \leq\! \mu(M)$ for every 
subrepresentation $(0) \subsetneq N \subsetneq M$. \end{Definition}

\bigskip
\noindent
We identify the Chow ring $A(\PP^{r-1})$ with the truncated polynomial ring $\ZZ[X]/(X^r)$. Theorem \ref{Fun3} and Corollary \ref{Dc2} then show that
\[ \mu(M) = \mu(\TilTheta(M))\]
for every non-zero $M \in \EKP(K_r)$. In particular, $\mu(\TilTheta(M))\!\ge\!0$ for every such $M$.

\bigskip

\begin{Prop} \label{StabE1} Let $M \in \EKP(K_r)\!\smallsetminus\!\{(0)\}$. The following statements hold:
\begin{enumerate}
\item If $\TilTheta(M)$ is semistable, then $M$ is semistable.
\item If $\TilTheta(M)$ is stable, then $M$ is stable.
\end{enumerate} \end{Prop}

\begin{proof} Let $(0) \subsetneq N \subsetneq M$ be a proper submodule with canonical inclusion $\iota : N \hookrightarrow M$. There results a commutative diagram
\[ \begin{tikzcd}   (0) \arrow[r] & \widetilde{N_1} \arrow[r, "\TilTheta_N"] \arrow[d,"\tilde{\iota}_1"] & \widetilde{N_2} \arrow[d,"\tilde{\iota}_2"] \arrow[r] & \TilTheta(N) \arrow[r] \arrow[d,"\TilTheta(\iota)"] & (0) \\
                           (0) \arrow[r] & \widetilde{M_1} \arrow[r, "\TilTheta_M"]   & \widetilde{M_2}  \arrow[r] & \TilTheta(M) \arrow[r] & (0) 
\end{tikzcd} \] 
with exact rows. In view of $\msCoker \tilde{\iota}_1 \cong \widetilde{(M/N)_1}$ the Snake Lemma provides a monomorphism $\msKer \TilTheta(\iota) \hookrightarrow \widetilde{(M/N)_1}$. By definition,
$\widetilde{(M/N)_1} \cong (M_1/N_1)\!\otimes_\KK\cO_{\PP^{r-1}}(-1)$ is semistable (see \cite[(1.2.4)]{OSS}) and with slope $\mu( \widetilde{(M/N)_1})\!=\!-1$ whenever $ \widetilde{(M/N)_1}\!\ne\!(0).$ 

Since $\Coh(\PP^{r-1})$ is an abelian category, there is an exact sequence
\[ (\ast) \ \ \ \ \ \ (0) \lra \msKer \TilTheta(\iota) \lra \TilTheta(N) \lra \msIm\TilTheta(\iota) \lra (0)\]
and $\TilTheta$ being faithful entails $\msIm\TilTheta(\iota)\!\ne\!(0)$.

(1) Suppose that $\TilTheta(M)$ is semistable. If $\msKer \TilTheta(\iota)\!=\!(0)$, we have $\mu(N)\!=\!\mu(\TilTheta(N))\!=\!\mu(\msIm\TilTheta(\iota))\!\le\!\mu(\TilTheta(M))\!=\mu(M)$. Alternatively, 
$\mu(\msKer\TilTheta(\iota))\!\le\!-1\!<\!\mu(\TilTheta(N))$ and Lemma \ref{Stab1} in conjunction with ($\ast$) implies $\mu(N)\!=\!\mu(\TilTheta(N))\!<\!\mu(\msIm \TilTheta(\iota))\!\le\!\mu(\TilTheta(M))\!=\mu(M)
$. Hence $M$ is semistable.

(2) Suppose that $\TilTheta(M)$ is stable. Then $M$ is indecomposable and the assumption $N_1\!=\!M_1$ either implies $M\!=\!P_0(r)\!=\!N$, or $N_2 \subseteq M_2\!=\im\psi_M\!=\!\im\psi_N\subseteq N_2$, a 
contradiction. If $\rk(\TilTheta(N))\!\geq\!\rk(\TilTheta(M))$, we obtain $\Delta_N\!\geq\!\Delta_M$, whence $\mu(N)\!<\!\mu(M)$. We may therefore assume that $\rk(\TilTheta(N))\!<\!\rk(\TilTheta(M))$. 

If $\msKer\TilTheta(\iota)\!=\!(0)$, we have $\rk(\msIm\TilTheta(\iota))\!=\!\rk(\TilTheta(N))\!<\!\rk(\TilTheta(M))$ as well as $\mu(N)\!=\!\mu(\TilTheta(N))\!=\!\mu(\msIm\TilTheta(\iota))\!<\!\mu(\TilTheta(M))\!=\!\mu(M)$.
Alternatively, we may proceed as in (1). \end{proof}

\bigskip

\begin{Remark} There is a slope $\mu$, defined for non-zero, torsion-free coherent sheaves of $\Gr_d(A_r)$, such that
\[ \mu(\TilTheta_d(M)) = \mu(M)\]
for all $M \in \repp(K_r,d)\!\smallsetminus\!\{(0)\}$. In view of \cite{Bis04}, the bundle $\cU_{(r,d)}$ is semistable relative to this slope, so that the foregoing result may be generalized to this context. \end{Remark} 

\bigskip

\begin{Lem} \label{StabE2} Let $M \in \repp(K_r,d)$ be non-projective and of minimal type. Then $M$ is stable. \end{Lem}

\begin{proof} Let $N \subsetneq M$ be a proper submodule. We shall show that $\mu(N)\!<\!\mu(M)$. Decomposing 
\[ N = \bigoplus_{i=1}^n N_i\]
into its indecomposable constituents, Lemma \ref{Stab1} implies $\mu(N) \le \max_{1\le i \le n} \mu(N_i)$. Hence we may assume $N$ to be indecomposable.

If $N$ is projective, then $\mu(N)\!\le\!\frac{1}{r-1}$. Assuming $\mu(N)\!\ge\!\mu(M)$, we have $d(r\!-\!d)\!+\!(d\!-\!1)\dim_\KK M_1=\!\Delta_M\!\ge\!(r\!-\!1)\dim_\KK M_1$, so that
$(r\!-\!d)\dim_\KK M_1\!\le\!d(r\!-\!d)$. Hence $\dim_\KK M_1\!\le\!d$, and the remark at the beginning of Section \ref{S:CatRep} shows that $M$ is projective, a contradiction. 

We may thus assume that $N$ is not projective, so that Corollary \ref{MinType3}(1) implies $\Delta_N(d)\!\ge\!\Delta_M(d)$. In view of Corollary \ref{MinType3}(2), the non-projective representation $M$ is indecomposable. 
The assumption $N_1\!=\!M_1$ thus entails $N_2\!=\!\psi_M(A_r\!\otimes_\KK\! N_1)\!=\!\psi_M(A_r\otimes_\KK\! M_1)\!=\!M_2$, a contradiction. Consequently,
\begin{eqnarray*} 
\dim_\KK N_1 \Delta_M & = & (\dim_\KK N_1)(\Delta_M(d)\!+\!(d\!-\!1)\dim_\KK M_1) = \dim_\KK N_1\Delta_M(d)\!+\!(d\!-\!1)\dim_\KK M_1\dim_\KK N_1\\ 
                                         & < & \dim_\KK M_1 \Delta_N(d)\!+\!(d\!-\!1)\dim_\KK M_1\dim_\KK N_1 =  (\dim_\KK M_1)(\Delta_N(d)\!+\!(d\!-\!1)\dim_\KK N_1)\\ 
                                         & = & \dim_\KK M_1\Delta_N,
\end{eqnarray*}
whence $\mu(N)\!<\!\mu(M)$. As a result, $M$ is stable. \end{proof} 

\bigskip

\begin{Remark} For $d\!=\!1$ and $\Char(\KK)\!=\!0$, the foregoing result is a consequence of Proposition \ref{StabE1}(2) and \cite[(2.7)]{BoSp92}. \end{Remark}

\bigskip

\subsection{The distribution of slopes} \label{S:Dist} For $N \in \rep(K_r)$ with $N_1\!\ne\!(0)$ we put
\[ \nu(N) := \frac{\dim_\KK N_2}{\dim_\KK N_1}\]
and note that
\[ \mu(M) = \frac{1}{\nu(M)\!-\!1}\]
for every $M \in \EKP(K_r)$ such that $\dim_\KK M_1\!\ne\!0$. Given $r\! \geq\! 3$, recall that $L_r\! =\! \frac{r+\sqrt{r^2\!-\!4}}{2}$. We have $\frac{1}{L_r}\!=\! \frac{r\!-\!\sqrt{r^2\!-\!4}}{2}$ and $r\!-\!L_r\!=\!\frac{1}{L_r}$.

\bigskip

\begin{Prop}\label{Dist1} Let $M,N \in \EKP(K_r)$ be indecomposable with $M \not \cong P_0(r)$.
\begin{enumerate}
\item We have $\lim_{n \to \infty} \nu(\sigma^{-n}_{K_r}(M))\!=\!L_r$.
\item The sequence $(\mu(\sigma^{-n}_{K_r}(M)))_{n \in \NN_0}$ of slopes converges to $\frac{1}{L_r - 1}$.
\item If $M$ is regular, then the sequence $(\mu(\sigma^{-n}_{K_r}(M)))_{n \in \NN_0}$ is strictly decreasing.
\item If $M$ is regular and $N$ is preprojective, then $\mu(N)\! <\! \frac{1}{L_r -1}\!<\!\mu(M)$.
\end{enumerate}\end{Prop}

\begin{proof} We write $\udim\sigma^{-n}_{K_r}(M)\!=:\!(x_n,y_n)$ for all $n \in \NN_0$ and recall that 
\[\udim \sigma^{-1}_{K_r}(X) = (\dim_\KK X_2,r\dim_\KK X_2\!-\!\dim_\KK X_1)\] 
for each representation $X$ without $S_1$ as a direct summand. Theorem \ref{CatRep3} yields $\sigma^{-n}_{K_r}(M) \in \EKP(K_r)$ and $\sigma^{-n}_{K_r}(M)$ is regular or preprojective 
with $x_n\!\neq\!0$. For $n \in \NN_0$ we set 
\[d_n(M) := \nu(\sigma^{-n}_{K_r}(M))\] 
and note that $\sigma^{-n}_{K_r}(M) \in \EKP(K_r)$ implies $d_n(M)\!>\!1$. 

(1) We first consider the case where $M$ is preprojective.  Kac's Theorem implies $0\!<\!q_r(\udim \sigma^{-n}_{K_r}(M))$ $=\!q_r(x_n,y_n)\!=\!x_n^2\!+\!y_n^2\!-\!rx_n y_n$, whence
 \[ (\ast) \ \ \ \ \ \ d_n(M) = \frac{y_n}{x_n} > r\!-\!\frac{x_n}{y_n} = \frac{r y_n\!-\!x_n}{y_n} = d_{n+1}(M) > 1.\]
Now let $M$ be regular and note that $f\!:=\!X^2\!-\!rX\!+\!1 \in \RR[X]$ has roots $L_r,\frac{1}{L_r}$. As $r\!\geq\!3$, Kac's Theorem yields $0\!>\!x_n^2\!+\!y_n^2\!-\!rx_n y_n$, so that $f(d_n(M))\!<\!0$ and $d_n(M)\!<\!L_r$.
Moreover, the computation of ($\ast$) implies $d_n(M)\!<\!d_{n+1}(M)$. 

We have shown that the sequence $(d_n(M))_{n \in \NN_0}$ converges in both cases and the identity
\[ d_{n+1}(M) = \frac{ry_n-x_n}{y_n} = r\!-\!\frac{1}{d_n(M)}\] 
implies that $d\!:=\!\lim_{n \to \infty} d_n(M)$ is a root of $f$. Since $\frac{1}{L_r}\!<\!1\!<\!d_n(M)$ for all $n \in \NN_0$, it follows that $d\!=\!L_r$.

(2) We have $\mu(\sigma^{-n}_{K_r}(M))\!=\!\frac{1}{d_n(M)-1}$. Since $d_n(M)\!>\!1$ for all $n \in \NN_0$, we conclude with (1) that $(\mu(\sigma^{-1}_{K_r}(M)))_{n \in \NN_0}$ converges to $\frac{1}{L_r-1}$.

(3) Let $M$ be regular. As noted in the proof of (1), the sequence $(d_n(M))_{n \in \NN_0}$ is strictly increasing, so that the identity $\mu(\sigma^{-n}_{K_r}(M))\!=\!\frac{1}{d_n(M)-1}$ implies the assertion.
    
(4) Let $M$ be regular. Then $1\!<\!d_0(M)\!<\!L_r$ and therefore $\mu(M)\!=\!\frac{1}{d_0(M)-1}\!>\!\frac{1}{L_r-1}$. If $N\!=\!P_0(r)$, then $\mu(N)\!=\!0$. Alternatively, $d_0(N)$ is well-defined and $q_r(\udim N)\!>\!0$ 
yields $d_0(N)\!>\!L_r$. This shows $\mu(N)\!=\!\frac{1}{d_0(N)-1}\!<\!\frac{1}{L_r-1}$. \end{proof}

\bigskip
\noindent
We continue by recording three important consequences.

\bigskip

\begin{Cor} \label{Dist2} Let $M \in \EKP(K_r)$ be indecomposable.
\begin{enumerate}
    \item The representation $M$ is (semi)stable if and only if $\sigma_{K_r}^{-1}(M) \in \EKP(K_r)$ is (semi)stable.
    \item The following statements are equivalent:
\begin{enumerate}
    \item All representations in $\{ \sigma_{K_r}^z(M) ; z \in \ZZ\} \cap \EKP(K_r)$ are (semi)stable\footnote{For $M = P_{\ell}(r)$ we only consider $z \in \ZZ_{\geq -\ell}$.}.
    \item The representation $M$ is (semi)stable.
\end{enumerate}
\item If $M$ is regular and semistable, then $M$ is quasi-simple.
\end{enumerate}
\end{Cor}
\begin{proof}
Let $Z \in \EKP(K_r)$. Recall that we have $\sigma_{K_r}^{-1}(Z) \in \EKP(K_r)$ and 
\[\udim \sigma_{K_r}^{-1}(Z) = ( \dim_{\KK} Z_2,r \dim_{\KK} Z_2\! -\! \dim_{\KK} Z_1).\] 
Hence we obtain for $U,V \in \EKP(K_r)$:
\[ (\ast) \quad \mu(\sigma_{K_r}^{-1}(U)) \leq(<)  \mu(\sigma_{K_r}^{-1}(V)) \Leftrightarrow \dim_{\KK} U_1 \dim_{\KK} V_2 \leq(<) \dim_{\KK} V_1 \dim_{\KK} U_2 \Leftrightarrow \mu(U) \leq(<) \mu(V).\]

(1) Since $\mu(P_0(r)) = 0 < \mu(P_1(r))$ and $\sigma^{-1}_{K_r}(P_{\ell}(r)) = P_{\ell+1}(r))$ for all $\ell \in \NN_0$, we obtain $\mu(P_{\ell}(r)) < \mu(P_{\ell+1}(r))$ for all $\ell \in \NN_0$. Clearly, $P_0(r)$ is stable. Now we fix $\ell \in \NN$ and let $(0) \subsetneq X \subsetneq P_{\ell}(r)$ be a subrepresentation. We fix a direct summand $Y \subseteq X$ of maximal slope. Then Lemma \ref{Stab1} yields $\mu(Y) \geq \mu(X)$. Since the only proper indecomposable subrepresentations of $P_{\ell}(r)$ are (up to isomorphism) $P_0(r),\ldots,P_{\ell-1}(r)$, we conclude $\mu(P_{\ell}(r))\! >\! \mu(P_{\ell-1}(r))\! \geq\! \mu(Y)\! \geq\! \mu(X)$. Hence every preprojective indecomposable representation is stable and we only have to consider the case that $M$ is regular.

At first we assume that $M$ is stable. Let $(0) \subsetneq X \subsetneq N\!:=\!\sigma_{K_r}^{-1}(M)$, and assume that $X$ is indecomposable. If $X$ is preprojective, then Proposition \ref{Dist1}(4) implies $\mu(X)\!<\!\mu(N)$. 
Alternatively, $X$ is regular. Since $\sigma_{K_r}$ is left exact and an equivalence on regular modules, there is a proper inclusion $(0)\!\ne\!\sigma_{K_r}(X) \hookrightarrow M$. As $M$ is stable, we obtain $\mu(\sigma_{K_r}(X))\!<\!
\mu(M)$ and $(\ast)$ implies $\mu(X)\!\le\!\mu(Y)$. If $X$ is not indecomposable, we consider an indecomposable summand $Y \subseteq X$ of maximal slope and argue as in the preprojective case. This shows that $\sigma_{K_r}^{-1}(M)$ is stable.

Now we assume that $\sigma_{K_r}^{-1}(M)$ is stable and let $(0) \subsetneq X \subsetneq M$ be a subrepresentation of maximal slope. We show that $\mu(X)\! <\! \mu(M)$. If $M/X \in \add(I_0(r))$, then we find $a \in \NN$ such that $\udim X = (\dim_{\KK} M_1-a,\dim_{\KK}M_2)$ and conclude with $\dim_{\KK} M_2 \neq 0$ that
\[ \mu(X) = \frac{\dim_{\KK} M_1-a}{\Delta_M + a} < \frac{\dim_{\KK} M_1}{\Delta_M} = \mu(M).\]
Hence we may assume that $M/X\!=\!Q\! \oplus \! aI_0(r)$ with $a \in \NN_0$ and $Q \in \rep_1(K_r)$ non-zero. We find a subrepresentation $0 \subsetneq Y \subsetneq M$ with $\udim Y\!=\!\udim X\!+\!(a,0)$ and maximality of $\mu(X)$ implies 
\[  \frac{\dim_{\KK} X_1+a}{\Delta_X - a} = \mu(Y) \leq \mu(X) =\frac{\dim_{\KK} X_1}{\Delta_X}.\]
This can only happen if $a\!=\!0$. Hence we have a short exact sequence
\[ (0) \lra X \lra M \lra Q \lra (0)\]
in $\rep_1(K_r)$. Since $\sigma^{-1}_{K_r}: \rep_1(K_r)\lra \rep_2(K_r)$ is exact (cf.\ Section \ref{S:CatRep}), we obtain  a proper  inclusion $(0) \neq \sigma^{-1}_{K_r}(X) \hookrightarrow \sigma_{K_r}^{-1}(M)$ and $\sigma_{K_r}^{-1}(M)$ being stable in conjunction with $(\ast)$ implies that $\mu(X) < \mu(M)$. Hence $M$ is stable.
The arguments concerning semistability are analogous. 

(2) This is a direct consequence of (1).

(3)  We write $M\!=\!N_{[\ell]}$ with quasi-socle $N$ and assume that $\ell\! \geq\! 2$. General AR-theory provides a short exact sequence
   \[ (0) \lra N \lra M \lra \tau_{K_r}^{-1}(N_{[\ell\!-\!1]}) \lra (0)\]
in $\EKP(K_r)$ and the See-Saw Lemma for representations give us $\mu(N)\!\leq\! \mu(M)\! \leq \! \mu(\tau_{K_r}^{-1}(N_{[\ell\!-\!1])})$.
Moreover, $N_{[\ell\!-\!1]}$ being a subrepresentation of $M\!=\!N_{[\ell]}$ and Proposition \ref{Dist1}(3) imply 
  \[ \mu(M) \leq \mu(\tau_{K_r}^{-1}(N_{[\ell\!-\!1]})) < \mu(N_{[\ell\!-\!1]}) \leq \mu(M),\]
a contradiction. Hence $\ell =  1$ and $M$ is quasi-simple.
\end{proof}

\bigskip

\begin{Remarks} (1) Owing to Corollary \ref{Dist2}, a cone of a regular component either contains no (semi)stable representations or all of its quasi-simple vertices are (semi)stable.

(2) Let $r \geq 3$ and $B_1,B_2$ be non-isomorphic indecomposable representations with dimension vector $(1,1)$. Clearly, $B_1,B_2$ are elementary representations and by \cite[(1.4)]{KL96} $B_1$ and $B_2$ are $\Hom$-orthogonal. We conclude $\dim_{\KK} \Ext^1_{K_r}(B_2,B_1)\!=\!  \dim_{\KK} \Hom_{K_r}(B_2,B_1)\! -\!q_r(1,1)\!  =\! r\!-\!2\! >\! 0$. In particular, we find a non-split short exact sequence 
\[(\ast) \quad (0) \lra B_1 \lra N \lra B_2 \lra (0)\]
and conclude with \cite[(VIII.2.8)]{ASS06} $\dim_{\KK} \End_{K_r}(N)\! <\! \dim_{\KK} \End(B_1\! \oplus\! B_2)\! =\! 2$. Hence $N$ is a brick with dimension vector $(2,2)$ and regular by Kac's Theorem. Since regular bricks for $K_r$ are quasi-simple (see \cite[(9.2),(9.4)]{Ke96}), we conclude that every quasi-simple vertex in the regular component $\cC$ of $N$ is a brick. 

There exists $n \in \NN$ such that $\tau^{-n}_{K_r}(B_1),\tau^{-n}_{K_r}(B_2), \tau^{-n}_{K_r}(N) \in \EKP(K_r)$. Recall that $\tau^{-n}_{K_r}$ is an exact endo-functor of the category $\reg(K_r)$ of regular representations. 
In view of ($\ast$), there results an exact sequence
\[ (0) \lra \tau_{K_r}^{-n}(B_1) \lra \tau^{-n}_{K_r}(N) \lra \tau^{-n}_{K_r}(B_2)\lra (0),\]
whose (nonzero) extreme terms have equal dimension vectors and hence equal slopes. The See-Saw Lemma now shows that $\tau^{-n}_{K_r}(N)$ is not stable. By Corollary \ref{Dist2}(2), $\cC\cap\EKP(K_r)$ does not contain any stable vertices. \end{Remarks} 

\bigskip

\begin{Prop} \label{Dist3} Let $E \in \EKP(K_r)$ be elementary. Then $E$ is stable. \end{Prop}

\begin{proof} Let $(0) \subsetneq X \subsetneq E$ be a proper indecomposable subrepresentation. In view of Proposition \ref{Dist1} we may assume that $X$ is regular. Thanks to \cite[(4.4)]{Ke96}, the factor module $E/X$ is 
preinjective. Consequently, $\Delta_{E/X}\!<\!0$, so that $\Delta_E\!=\!\Delta_X\!+\!\Delta_{E/X}\!<\!\Delta_X$. As $\dim_\KK X_1\ne\!0$, we obtain
\[ \mu(X) < \frac{\dim_\KK X_1}{\Delta_E} \le \mu(E).\]
As in the proof of Corollary \ref{Dist2}, this implies the assertion. \end{proof} 

\bigskip

\subsection{Filtrations} \label{S:Filtrat} As before, we assume that  $r\!\ge\!3$ and let 
\[\sigma_r\!:=\!\begin{pmatrix}
r & -1 \\
1 & 0
\end{pmatrix} \in \GL_2(\ZZ)\]
be the invertible matrix that describes the $\sigma_{K_r}$-orbits on the level of the Grothendieck group $K_0(\rep(K_r))$: For $M \in \rep_2(K_r)$ we have $\udim \sigma_{K_r}(M)\!=\!\sigma_r.\udim M$, while 
$\udim \sigma^{-1}_{K_r}(N)\!=\!\sigma_r^{-1}.\udim N$ for $N \in \rep_1(K_r)$, cf.\ Section \ref{S:CatRep}.  

We denote by $\lfloor \ \rfloor : \RR \lra \ZZ$ the floor function. 

\bigskip

\begin{Lem}\label{Filtrat1} Let $(a,b) \in \NN^2\!\smallsetminus\!\{0\}$ be such that $q_r(a,b)\!<\!0$ and $b\!=\!\lfloor a L_r \rfloor$. Let $M \in \rep(K_r)$ be an indecomposable representation with $\udim M\!=\! \sigma_{r}^{-
\ell}.(a,b)$ for some $\ell \in \NN$. Then $M$ is regular and there is a filtration
\[(0) = M_0 \subsetneq M_1 \subsetneq \cdots \subsetneq M_{n-1} \subsetneq M_n = M\]
such that for each $i \in \{1,\ldots,n\}$ the filtration factor $M_i/M_{i-1}$ is elementary and belongs to $\EKP(K_r)$.\end{Lem}

\begin{proof} As $q_r(\udim M)\!=\!q_r(a,b)\!<\!0$, Kac's Theorem ensures that the representation $M$ is regular. We define $N\!:=\!\sigma_{K_r}^{\ell}(M)$ and note that $\udim N\!=\!(a,b)$. Thanks to \cite[(5.2(1))]{Bi20b}
there is a filtration
\[(0) = N_0 \subsetneq N_1 \subsetneq \cdots \subsetneq N_{n-1} \subsetneq N_n = N\] 
such that for each $j \in \{1,\ldots,n\}$ the filtration factor $E^j := N_j/N_{j-1}$ is elementary. By the same token, there is at most one $i \in \{1,\ldots,n\}$ such that $E^i$ is not in $\EKP(K_r)$ and in this case there is $\fv \in 
\Gr_1(A_r)$ such that $E^i \cong E(\fv)$. We have $\udim E(\fv) = (1,r\!-\!1)$ and therefore $\udim \sigma^{-1}_{K_r}(E^i) = (r\!-\!1,r^2\!-\!r\!-\!1)$. The representation $X\!:=\!\sigma^{-1}_{K_r}(E^i)$ satisfies
\[ q_r(\udim X)\!+\!\dim_\KK X_2\!-\!\dim_\KK X_1 = (r\!-\!2)(r\!-\!1)\geq 1\]
and \cite[(3.1)]{Bi20b} ensures that $X \in \EKP(K_r)$. Since $\sigma^{-1}_{K_r}$ is exact on the category of regular representations and $\EKP(K_r)\!=\!\repp(K_r,1)$ is closed under $\sigma_{K_r}^{-1}$ (see Theorem 
\ref{CatRep3}), we conclude that 
\[(0) = \sigma_{K_r}^{-\ell}(N_0) \subsetneq \sigma_{K_r}^{-\ell}(N_1) \subsetneq \cdots \subsetneq \sigma_{K_r}^{-\ell}(N_{n-1}) \subsetneq \sigma_{K_r}^{-\ell}(N_n) = M\]
is the desired filtration. \end{proof}

\bigskip
\noindent
For $(a,b) \in \NN^2_0$ and $d \in \{1,\ldots, r\!-\!1\}$ we put
\[ \Delta_{(a,b)}(d):=b\!-\!da.\]

\bigskip

\begin{Lem}\label{Filtrat2} Let $(a,b) \in \NN^2\!\smallsetminus\!\{0\}$ be such that $q_r(a,b)\!<\!0$ and $q_r(a,b)\!+\!\Delta_{(a,b)}(r\!-\!1)\!\geq\!1$. For $\ell \in \NN_0$, we let $(a_\ell,b_\ell)\!:=\!\sigma_r^{-\ell}.(a,b)$. Then the 
following statements hold:
\begin{enumerate}
\item We have $q_r(a_\ell,b_\ell)\!+\Delta_{(a_{\ell},b_{\ell})}(r\!-\!1)\geq\! 1$.
\item Let $M$ be an indecomposable representation with dimension vector $(a_\ell,b_\ell)$. Then $M$ has a filtration such that each filtration factor is an elementary representation belonging to $\EKP(K_r)$. In particular, 
$M \in \EKP(K_r)$. \end{enumerate} \end{Lem} 

\begin{proof} 
Let $(x,y) \in \NN^2$. Direct computation shows that $\Delta_{\sigma_r^{-1}.(x,y)} (r\!-\!1)= \Delta_{(x,y)}$.

(1) We prove the statement by induction on $\ell\!\geq\!0$, the case $\ell\!=\!0$ being trivial. For $\ell\!>\!0$ we obtain
\begin{align*}
q_r(a_\ell,b_\ell)\!+\!\Delta_{(a_{\ell},b_{\ell})}(r\!-\!1) &= q_r(a_\ell,b_\ell)\!+\!\Delta_{\sigma_r^{-1}.(a_{\ell-1},b_{\ell-1})}(r\!-\!1) \\
&=  q_r(a_{\ell-1},b_{\ell-1})\!+\! \Delta_{(a_{\ell-1},b_{\ell-1})}\\
&> q_r(a_{\ell-1},b_{\ell-1})\!+\!\Delta_{(a_{\ell-1},b_{\ell-1})}(r\!-\!1) \geq 1.
\end{align*}

(2) In view of Lemma \ref{Filtrat1}, it suffices to show that $(x,y)\!:=\! \sigma_r.(a_\ell,b_\ell)$ satisfies $q_r(x,y)\!<\!0$ and $y\!=\!\lfloor x L_r \rfloor$. Clearly, we have $q_r(x,y)\!=\!q_r(a_\ell,b_\ell)\!=\!q_r(a,b)\!<\!0$. 
By definition, we obtain
\[ q_r(x,y)\! +\! y\! -\! x = q_r(a_\ell,b_\ell)\!+\! \Delta_{\sigma_r.(a_{\ell},b_{\ell})} = 
 q_r(a_\ell,b_\ell)\!+\! \Delta_{(a_{\ell},b_{\ell})}(r\!-\!1),\]
so that (1) implies $q_r(x,y)\!+\!y\!-\!x\!\geq\!1$. As $q_r(x,y)\!<\!0$, we have $x\!<\!y$ and \cite[(4.2)]{Bi20b} yields $y\!=\! \lfloor x L_r \rfloor$. \end{proof}

\bigskip

\begin{Cor}\label{Filtrat3} Let $\cC \subseteq \Gamma(K_r)$ be a regular component and assume that $X \in \cC$ is quasi-simple such that $q_r(\udim X)\!+\!\Delta_X(r\!-\!1)\!\geq\!1$. Then every representation in 
$(X\! \rightarrow)$ possesses a filtration whose filtration factors are elementary and belong to $\EKP(K_r)$.\end{Cor}

\begin{proof} Let $M \in (X\!\rightarrow)$, then each of the filtration factors $M_{[i]}/M_{[i-1]}$ of the quasi-composition series of $M$ is isomorphic to $\tau^{-\ell_i}_{K_r}(X) \cong \sigma_{K_r}^{-2\ell_i}(X)$ for some 
$\ell_i \in \NN_0$. By Lemma \ref{Filtrat2}, each of these representations possesses a filtration such that each filtration factor is elementary and in $\EKP(K_r)$.\end{proof}

\bigskip

\begin{Remark} Let $M \in \rep(K_r)$ be regular indecomposable. In view of Lemma \ref{Strat3} and $q_r(\udim\sigma^{-1}_{K_r}(M))\!=\!q_r(\udim M)$, it follows that every regular
AR-component $\cC \subseteq \Gamma(K_r)$ contains a quasi-simple representation $X$ such that $q_r(\udim X)\!+\!\Delta_X(r\!-\!1)\!\ge\!1$. \end{Remark}

\bigskip

\subsection{The Harder-Narasimhan filtration} \label{S:HN}

\bigskip

\begin{Definition} Let $\cF \in \Coh(\PP^{r-1})$ be a torsion-free coherent sheaf. There exists a uniquely determined filtration, the \it{Harder-Narasimhan filtration} (cf. \cite[Section 1.2]{La07}),
\[ (0) = \cF_0 \subsetneq \cF_1 \subsetneq \cF_2 \subsetneq \cdots \subsetneq \cF_{t-1} \subsetneq \cF_t = \cF\]
that satisfies the following two properties:
\begin{enumerate}
\item Each quotient $\cF^i\!:=\!\cF_i/\cF_{i-1}$ is semistable ($i \in \{1,\ldots,t\}$).
\item We have $\mu(\cF^1)\!>\!\mu(\cF^2)\!> \cdots >\!\mu(\cF^{t-1})\!>\!\mu(\cF^t)$.
\end{enumerate} \end{Definition}

\bigskip

\begin{Remark} The sheaf $\cF$ is semistable if and only if $t\!=\!1$. \end{Remark}

\bigskip

\begin{Lem}\label{HN1} Let $\cF \in \Coh(\PP^{r-1})$ be  a non-zero, torsion-free sheaf that is not semistable. There exists a non-zero subsheaf $\cU \subsetneq \cF$ with the following properties:
\begin{enumerate}
\item[(a)] $\cF/\cU$ is semistable,
\item[(b)] $\mu(\cU)\!>\!\mu(\cF)$,
\item[(c)] $\Hom_{\PP^{r-1}}(\cS,\cF/\cU)\!=\!(0)$ for each semistable sheaf $\cS \in \Coh(\PP^{r-1})$ with $\mu(\cF)\!\leq\!\mu(\cS)$. \end{enumerate} \end{Lem}

\begin{proof} We consider the Harder-Narasimhan filtration
\[ (0) = \cF_0 \subsetneq \cF_1 \subsetneq \cF_2 \subsetneq  \cdots \subsetneq \cF_{t-1} \subsetneq \cF_t = \cF\]
with $t\!\geq\!2$ and set $\cU\!:=\!\cF_{t-1}$. Then (a) holds by definition. Using Lemma \ref{Stab1} and the defining property (2), induction on $s$ implies that $\mu(\cF_{s-1})\!>\!\mu(\cF_s)\!>\!\mu(\cF^s)$ for $s \in 
\{2,\ldots,t\}$. Consequently, (b) holds as well. Lemma \ref{Stab1} also yields
\[ \mu(\cF/\cU) = \mu(\cF^t) < \mu(\cF) \leq \mu(\cS)\]
for each semistable sheaf $\cS \in \Coh(\PP^{r-1})$ with $\mu(\cF)\!\leq\!\mu(\cS)$.
Let $f : \cS \lra \cF/\cU$ be a morphism for such a sheaf $\cS$. As $\cF/\cU$ and $\cS$ are semistable, the assumption $f\!\neq\!0$ in conjunction with Lemma \ref{Stab1} implies
\[  \mu(\cF/\cU) \geq \mu(\im f) \geq \mu(\cS), \]
which contradicts $\mu(\cF/\cU)\!<\!\mu(\cS)$.\end{proof}

\bigskip

\subsection{The Beilinson Spectral Sequence} \label{S:BSS} For $\cF \in \Coh(\PP^{r-1})$ and $n \in \ZZ$, we write $\cF(n)$ for the $n$-th Serre twist of $\cF$. 

We denote by $\Omega$ the cotangent bundle of $\PP^{r-1}$ (see \cite[(6.10)]{Be17}), and, for $m \in \NN_0$, we let $\Omega^m\!:=\!\bigwedge^m_{\PP^{r-1}} \Omega$ be 
the $m$-th exterior power of $\Omega$, where $\Omega^0\!:=\!\cO_{\PP^{r-1}}$.

\bigskip

\begin{Lem}\label{BSS1} Let $-r\!<\!b\!\leq\!0$ and $i \in \NN_0$. For $m \in \NN_0$ we have 
\[ \HH^i(\PP^{r-1},\Omega^{m}(m\!+\!b)) \cong  \left\{ \begin{array}{cc}  \KK & \text{if} \ i\!=\!m\!=\!-b \\ (0) & \text{otherwise.} \end{array} \right. \]
\end{Lem}

\begin{proof} For $i\!\geq\!r$ there is nothing to show (see \cite[(3.10.5)]{Be17}). Hence we assume that $0\!\leq\!i\!\leq\!r\!-\!1$. By Bott's formula (see \cite[(6.10.1)]{Be17}) the only cases to consider are 
\begin{itemize}
\item $i\!=\!0$, $0\!\leq\!m\!<\!r$ and $m\!+\!b\!>\!m$: This can not happen since $b\!\leq\!0$.
\item $m\!+\!b\!=\!0$ and $m\!=\!i$: Then $\dim_\KK \HH^i(\PP^{r-1},\Omega^{m}(m\!+\!b))\!=\!1$.
\item $i\!=\!r\!-\!1$, $0\!\leq\!m\!<\!r$, $m\!+\!b\!\leq\!m\!-\!r$: This can not happen since $b\!> \!-r$.  \qedhere
\end{itemize} \end{proof}

\bigskip

\begin{Cor}\label{BSS2} (cf.\ \cite[(3.2)]{DK93})
Let $\cF \in \Coh(\PP^{r-1})$ be a coherent sheaf, $-(r\!-\!1)\!<\!d\!\leq\!0$ and 
\[ (0) \lra \widetilde{V_1}(d) \lra \widetilde{V_2}(d) \lra \cF \lra(0)\]
be a short exact sequence. If $\HH^i(\PP^{r-1},\cF\!\otimes_{\PP^{r-1}}\!\Omega^m(m))\!\neq\!(0)$, then $i\!=\!-d$ and $m \in \{-d,-(d\!-\!1)\}$, \end{Cor}

\begin{proof} Upon tensoring the exact sequence over $\cO_{\PP^{r-1}}$ with the vector bundle $\Omega^m(m)$, we obtain the exact sequence
\[ (0) \lra V_1\!\otimes_\KK\!\Omega^m(m\!+\!d\!-\!1)  \lra V_2\!\otimes_\KK\!\Omega^m(m\!+\!d) \lra  \cF\!\otimes_{\PP^{r-1}}\!\Omega^m(m) \lra (0).\]
The long exact sequence in cohomology provides an exact sequence
\[V_2\! \otimes_\KK\! \HH^{i}(\PP^{r-1},\Omega^m(m\!+\!d)) \lra \HH^{i}(\PP^{r-1},\cF\!\otimes_{\PP^{r-1}}\!\Omega^m(m)) \lra V_1\!\otimes_\KK\!\HH^{i+1}(\PP^{r-1},\Omega^m(m\!+\!d\!-\!1)).\]
Since $-r\!<\!d\!-\!1,d\! \leq\! 0$,  Lemma \ref{BSS1} shows, that the assumption $\HH^i(\PP^{r-1},\cF\!\otimes_{\PP^{r-1}}\!\Omega^m(m))\!\neq\!(0)$ entails either $i\!=\!m\!=\!-d$ or 
$i\!+\!1\!=\!m\!=-(d\!-\!1)$. Hence $(i,m)\!=\!(-d,-d)$ or $(i,m)\!=\!(-d,-(d\!-\!1))$. \end{proof}

\bigskip

\begin{Thm}(see \cite[(8.28)]{Huy06})\label{BSS3} Let $\cF \in \Coh(\PP^{r-1})$ be a coherent sheaf. There is a second quadrant spectral sequence $E\!=\!E_{\cF}$ with first page
\[ E^{pq}_1 = \HH^q(\PP^{r-1},\cF\!\otimes_{\PP^{r-1}}\!\Omega^{-p}(-p))\!\otimes_\KK\!\cO_{\PP^{r-1}}(p)\]
for $(-p,q) \in \{0,\ldots,r\!-\!1\}^2$ and zero otherwise, which converges to
\[E^{p+q} \Rightarrow \begin{cases}
\cF, \ \text{for} \ p\!+\!q \!=\!0 \\
(0), \ \text{for} \ p\!+\!q\!\neq\!0.
\end{cases}
\]
Moreover, this sequence is natural in $\cF$. \end{Thm} 

\bigskip

\begin{Definition} Given $-(r\!-\!1)\!<\!d\!\leq\!0$ and $\cF \in \Coh(\PP^{r-1})$, we denote the differential $d_1^{d-1,-d}$ on the first page of the spectral sequence $E_\cF$ by  
\[ \alpha_{\cF,d}  : E_1^{d-1,-d} \lra E_1^{d,-d}.\]
\end{Definition}

\bigskip

\begin{Cor}\label{BSS4} Let $\cF$ be a Steiner bundle, $-(r\!-\!1)\!<\!d\!\leq\!0$. There exists a short exact sequence
\[ (0) \lra E_1^{d-1,-d} \stackrel{\alpha_{\cF(d),d}}{\lra} E_1^{d,-d} \lra \cF(d) \lra (0).\]
\end{Cor}

\begin{proof} By Corollary \ref{BSS2} we know that the first page of the spectral sequence $E_{\cF(d)}$ has at most two non-zero entries $E_1^{d-1,-d}\!=
\!\HH^{-d}(\PP^{r-1},\cF(d)\!\otimes_{\PP^{r-1}}\! \Omega^{-(d-1)}(-(d\!-\!1)))\!\otimes_\KK\!\cO_{\PP^{r-1}}(d\!-\!1)$ and $E_1^{d,-d}\!=\!\HH^{-d}(\PP^{r-1},\cF(d)\!\otimes_{\PP^{r-1}}\!\Omega^{-d}(-d))\!
\otimes_\KK\!\cO_{\PP^{r-1}}(d)$. In particular, $E_{\cF(d)}$ collapses at $E_2$, and we have $E_\infty\!=\!E_2$. Since $E_1^{d+1,-d}\!=\!(0)\!=\!E_1^{d-2,-d}$, we obtain
\[ E_2^{d,-d} = \ker d_1^{d,-d}/\im \alpha_{\cF(d),d} = E_1^{d,-d} / \im \alpha_{\cF(d),d} = \coker \alpha_{\cF(d),d} \ \ \text{and} \ \ E_2^{d-1,-d} = \ker \alpha_{\cF(d),d} \] 
while all other terms vanish. Now Theorem \ref{BSS3} implies $\ker \alpha_{\cF(d),d}\!=\!(0)$ and $\cF(d) \cong \coker \alpha_{\cF(d),d}$.
There results the short exact sequence
\[ (0) \lra E_1^{d-1,-d} \stackrel{\alpha_{\cF(d),d}}{\lra} E_1^{d,-d}  \lra \cF(d) \lra (0). \qedhere\]
\end{proof}

\bigskip

\subsection{Stability in $\EKP(K_3)$ and $\StVect(\PP^2)$} \label{Stab-KP}
In terms of its methods, this section is based on those set forth in \cite[(IV.3)]{Dr87} (see also \cite[(5.1.2)]{Dr99}).

\bigskip

\begin{Lem}\label{Stab-KP1} Let $\cF \in \Coh(\PP^{r-1})$ be torsion-free, $b \in \ZZ$. Then $\cF$ is (semi)stable if and only if the same is true for $\cF(b)$. \end{Lem}

\begin{proof} Note that $\cO_{\PP^{r-1}}(b)\!\otimes_{\PP^{r-1}}\!-$ induces an exact auto-equivalence on the category of coherent sheaves. Given a torsion-free coherent sheaf $\cG\!\neq\!(0)$, the sheaf $\cG(b)$ is locally 
isomorphic to $\cG$ and hence also torsion-free and of the same rank. We conclude with \cite[(7.5.1)]{Be17}
\[ \mu(\cO_{\PP^{r-1}}(b)\!\otimes_{\PP^{r-1}}\!\cG) = \mu(\cG(b)) = \frac{c_1(\cG(b))}{\rk(\cG)} = \frac{c_1(\cG)\!+\!b\rk(\cG)}{\rk(\cG)} = \mu(\cG)\!+\!b.\]
Hence the equivalence respects the order of slopes. \end{proof}

\bigskip

\begin{Prop}\label{Stab-KP2} Let $M \in \EKP(K_3)$. Then the following statements hold: 
\begin{enumerate} 
\item If $\Delta_M(2)\!\ge\!0$ and $M$ is semistable, then $\TilTheta(M)$ is semistable.
\item If $\Delta_M(2)\!>\!0$ and $M$ is stable, then $\TilTheta(M)$ is stable. \end{enumerate} \end{Prop}

\begin{proof} We begin the proof by listing several properties that both cases have in common. 

Given $\cU \in \Coh(\PP^2)$, we consider the differentials
\[\alpha_{\cU,-1} \colon \HH^{1}(\PP^{2},\cU\!\otimes_{\PP^2}\!\Omega^{2}(2))\!\otimes_\KK\!\cO_{\PP^2}(-2)  \lra  \HH^{1}(\PP^{2},\cU\!\otimes_{\PP^2}\!\Omega^{1}(1))\!\otimes_\KK\!\cO_{\PP^2}(-1)\]
and
\[\delta_{\cU,-1} \colon \HH^{2}(\PP^{2},\cU\!\otimes_{\PP^2}\!\Omega^{2}(2))\!\otimes_\KK\! \cO_{\PP^2}(-2)  \lra  \HH^{2}(\PP^{2},\cU\!\otimes_{\PP^2}\!\Omega^{1}(1))\!\otimes_\KK\!\cO_{\PP^2}(-1)\]
from the first page of the Beilinson spectral sequence $E_\cU$, see Theorem \ref{BSS3}. 

Let $\cF\!:=\!\TilTheta(M)(-1)$ and $\cU \subsetneq \cF$ be a subsheaf. The inclusion $\cU \hookrightarrow \cF$
induces $\KK$-linear maps
\[ \tau_i : \HH^1(\PP^{2},\cU\!\otimes_{\PP^2}\!\Omega^i(i)) \lra  \HH^1(\PP^{2},\cF\!\otimes_{\PP^2}\!\Omega^i(i))\]
for $i \in \{1,2\}$. 

Throughout, we assume that $\cU \subsetneq \cF$ is a non-zero subsheaf such that
\[ (\star) \ \ \ \ \ \ \ \Hom_{\PP^2}(\cO_{\PP^2}(\ell),\cF/\cU)\!=\!(0) \ \ \text{for all} \ \ \ell\!\geq\!0.\]

\medskip
(a) {\it $\tau_1$ is injective}.

\smallskip
\noindent
 Recall that $\HH^0(\PP^2,\cG)$ is naturally isomorphic to the global section of $\cG$ for every $\cO_{\PP^2}$-module (see \cite[(3.10)]{Be17}). Hence the isomorphism $\Omega^1(1)^\vee \cong \cT(-1)$ and \cite[(6.1.3)]{Be17} imply
\[
\HH^0(\PP^2,(\cF/\cU)\!\otimes_{\PP^2}\!\Omega^1(1)) \cong \HH^0(\PP^2,(\Omega^1(1)^{\vee})^{\vee}\!\otimes_{\PP^2}\!(\cF/\cU)) \cong\Hom_{\PP^2}(\cT(-1),\cF/\cU).
\]
Application of the left exact functor $\Hom_{\PP^2}(-,\cF/\cU)$ to the Euler sequence $(0) \lra \cO_{\PP^2}(-1) \lra \cO_{\PP^2}^3 \lra \cT(-1) \lra (0)$ yields $\dim_\KK \Hom_{\PP^2}(\cT(-1),\cF/\cU)\!\leq\!
\dim_\KK\Hom_{\PP^2}(\cO_{\PP^2}^3,\cF/\cU) \stackrel{(\star)}{=}0$. Hence we have $\HH^0(\cO_{\PP^2},(\cF/\cU)\!\otimes_{\PP^2}\!\Omega^1(1))\!=\!(0)$. The long exact sequence in cohomology yields the exactness of
\[ \HH^0(\PP^2,(\cF/\cU)\!\otimes_{\PP^2}\!\Omega^{1}(1)) \lra \HH^1(\PP^{2},\cU\!\otimes_{\PP^2}\!\Omega^{1}(1)) \stackrel{\tau_1}{\lra} \HH^{1}(\PP^{2},\cF\!\otimes_{\PP^2}\! \Omega^{1}(1)).\]
Consequently, $\tau_1$ is injective. \hfill $\diamond$

\medskip
(b) {\it $\tau_2$ is injective}.

\smallskip
\noindent
Upon tensoring the sequence $(0) \lra \cU \lra \cF \lra \cF/\cU \lra (0)$ with the vector bundle $\Omega^2(2)$ the long exact sequence in cohomology yields the exactness of
 \[ \HH^{0}(\PP^2,(\cF/\cU)\!\otimes_{\PP^2}\!\Omega^{2}(2)) \lra \HH^1(\PP^2,\cU\!\otimes_{\PP^2}\!\Omega^2(2)) \stackrel{\tau_2}{\lra} \HH^{1}(\PP^2,\cF\!\otimes_{\PP^2}\!\Omega^{2}(2)).\]
By \cite[(6.10)]{Be17}, we have $\Omega^2(2) \cong \cO_{\PP^2}(-1)$. The choice of $\cU$ implies
\[ \HH^0(\PP^2,(\cF/\cU)\!\otimes_{\PP^2}\!\Omega^2(2)) \cong \Hom_{\PP^2}(\cO_{\PP^2}(1),\cF/\cU) \stackrel{(\star)}{=}(0),\]
so that $\tau_2$ is injective.  \hfill $\diamond$

\medskip

(c) {\it $\alpha_{\cU,-1}$ is injective}. 

\smallskip
\noindent
Corollary \ref{BSS4} and the naturality of the Beilinson spectral sequence yield the following commutative diagram
\[ (\dagger) \ \ \ \ \ \ 
\xymatrix{
 & \HH^1(\PP^2,\cU\!\otimes_{\PP^2}\!\Omega^2(2))\!\otimes_\KK\!\cO_{\PP^2}(-2) \ar[d]^{\tau_2 \otimes \id_{\cO_{\PP^2}(-2)}} \ar^{\alpha_{\cU,-1}}[r] & \HH^1(\PP^{2},\cU\!\otimes_{\PP^2}\!
 \Omega^1(1))\!\otimes_\KK\! \cO_{\PP^2}(-1)  \ar[d]^{\tau_1 \otimes \id_{\cO_{\PP^2}(-1)}} \\
(0) \ar[r] &\HH^1(\PP^2,\cF\!\otimes_{\PP^2}\! \Omega^2(2))\!\otimes_\KK\!\cO_{\PP^2}(-2) \ar^{\alpha_{\cF,-1}}[r] & \HH^{1}(\PP^2,\cF\!\otimes_{\PP^2}\! \Omega^1(1))\!\otimes_\KK\!\cO_{\PP^2}(-1),}
\]
whose lower row is exact. In view of (b), the morphism $\tau_2\otimes \id_{\cO_{\PP^2}(-2)}$ is a monomorphism, so that $\ker\alpha_{\cU,-1}\!=\!(0)$.  \hfill $\diamond$

\medskip 
(d) {\it There is an exact sequence 
\[ (0) \lra \coker \alpha_{\cU,-1} \lra \cU \lra \ker \delta_{\cU,-1} \lra (0)\] 
of torsion-free coherent sheaves.} 

\smallskip
\noindent
Owing to Corollary \ref{BSS2}, we have 
\[ \dim_\KK \HH^0(\PP^2,\cU\!\otimes_{\PP^2}\!\Omega^m(m)) \leq \dim_\KK \HH^0(\PP^2,\cF\!\otimes_{\PP^2}\!\Omega^m(m)) = 0\]
for $m \in \{0,1,2\}$ as well as $\HH^1(\PP^2,\cF)\!=\!(0)$. Since $(0)\!\stackrel{(\star)}{=}\!\Hom_{\PP^2}(\cO_{\PP^2},\cF/\cU)\!=\!\HH^0(\PP^2,\cF/\cU)$, we conclude with the exactness of 
\[ \HH^0(\PP^2,\cF/\cU) \lra  \HH^1(\PP^2,\cU) \lra \HH^1(\PP^2,\cF)\]
that $\HH^1(\PP^2,\cU)\!=\!(0)$. Consequently, the first page of the Beilinson spectral sequence $E_\cU$ reads
\begin{center}
\begin{tabular}{ |c|c|c| } 
 \hline
  $\HH^2(\PP^2,\cU\otimes_{\PP^2}\! \Omega^2(2))\!\otimes_\KK\!\cO_{\PP^2}(-2)$ & $\HH^2(\PP^2,\cU\!\otimes_{\PP^2}\!\Omega^1(1))\!\otimes_\KK \cO_{\PP^2}(-1)$ & $\HH^2(\PP^2,\cU)\! \otimes_\KK\! \cO_{\PP^2}$ \\
  \hline
 $\HH^1(\PP^2,\cU\!\otimes_{\PP^2}\!\Omega^2(2))\!\otimes_\KK\! \cO_{\PP^2}(-2)$ &  $\HH^1(\PP^2,\cU\!\otimes_{\PP^2}\!\Omega^1(1))\!\otimes_\KK\! \cO_{\PP^2}(-1)$  & $(0)$ \\ 
 \hline
 $(0)$ & $(0)$ & $(0)$ \\
 \hline
\end{tabular}
\end{center}
with all other terms being zero. In view of (c) the second page thus is of the following form:
\begin{center}
\begin{tabular}{ |c|c|c| } 
\hline
$\ker \delta_{\cU,-1}$ & $\ast$ & $\ast$ \\
\hline
$(0)$ & $\coker \alpha_{\cU,-1}$ & $(0)$ \\
\hline
$(0)$ & $(0)$ & $(0)$ \\
 \hline
\end{tabular}
\end{center}
Theorem \ref{BSS3} in conjunction with \cite[(XV.5.11)]{CE}\footnote{One sets $r\!=\!2\!=\!q'$, $q\!=\!1$ and $n\!=\!0$.} now provides an exact sequence
\[ (0) \lra \coker \alpha_{\cU,-1} \lra \cU \lra \ker \delta_{\cU,-1} \lra (0)\] 
of torsion-free coherent sheaves. \hfill $\diamond$

\medskip
(e) {\it  If $\mu(\cU)\!\ge\!(>)\mu(\cF)$, then $\mu(\coker \alpha_{\cU,-1})\!\ge\!(>)\mu(\cF)$.}

\smallskip
\noindent
Suppose that $\ker \delta_{\cU,-1}\!\neq\!(0)$. Since $\cO_{\PP^2}(-2)$ is semistable, we conclude with $\ker \delta_{\cU,-1}  \subseteq  \HH^2(\PP^2,\cU\!\otimes_{\PP^2}\!\Omega^2(2))\!\otimes_\KK\! \cO_{\PP^2}(-2)$, \cite[(1.2.4)]{OSS} and $\frac{\dim_{\KK} M_1}{\Delta_M} \geq 0$ that 
\[\mu(\ker \delta_{\cU,-1}) \leq \mu(\cO_{\PP^2}(-2)) = - 2 < -1\!+\!\frac{\dim_\KK M_1}{\Delta_M} =  \mu(\cF) \le \mu(\cU).\]
In particular, $\ker \delta_{\cU,-1} \not \cong \cU$ and therefore (d) forces $\coker \alpha_{\cU,-1}\!\neq\!(0)$. 
Now (d) in conjunction with Lemma \ref{Stab1} implies $\mu(\coker \alpha_{\cU,-1})\!>\!\mu(\cU)\!\ge\!\mu(\cF)$.

Alternatively, $\ker \delta_{\cU,-1}\!=\!(0)$, so that $\coker \alpha_{\cU,-1} \cong \cU$ and $\mu(\coker \alpha_{\cU,-1})\!=\!\mu(\cU)\!\ge\!(>)\mu(\cF)$. \hfill $\diamond$

\medskip
(f) {\it There exists a subrepresentation $(0) \subsetneq N \subsetneq M$ of $M$ such that $\mu(N)\!=\!\mu(\coker \alpha_{\cU,-1}(1))$}.

\smallskip
\noindent
Setting $V_i\!:=\!\HH^1(\PP^2,\cU\!\otimes_{\PP^2}\!\Omega^i(i))$ and $W_i\!:=\!\HH^1(\PP^2,\cF\!\otimes_{\PP^2}\!\Omega^i(i))$ for $i \in \{1,2\}$, application of $-\!\otimes_{\PP^2}\cO_{\PP^2}(1)$ to the 
diagram ($\dagger$) in conjunction with (c) and Corollary \ref{BSS4} yields
\[ 
\xymatrix{
 (0) \ar[r] & \widetilde{V_1} \ar[d]^{\tilde{\tau}_2} \ar^{\alpha_{\cU,-1}(1)}[r] & \widetilde{V_2}  \ar[d]^{\tilde{\tau}_1} \ar[r] & \coker\alpha_{\cU,-1}(1) \ar[d] \ar[r] & (0)\\ 
(0) \ar[r] &\widetilde{W_1} \ar^{\alpha_{\cF,-1}(1)}[r] & \widetilde{W_2} \ar[r] & \TilTheta(M) \ar[r] & (0),
}
\]
where we have used the notation introduced in Section \ref{S:Fun}. Let $x \in \PP^2$. By (b) and diagram \ref{diagram1} (see Section \ref{S:Fun}), the map $\tilde{\tau}_2(x)$ is injective and since $\TilTheta(M)$ is a vector bundle, \cite[(6.2.1(b))]{Be17} shows that this also holds for $\alpha_{\cF,-1}(1)(x)$. Passing to fibers in the diagram above, we see that $\alpha_{\cU,-1}(1)(x)$ is injective, and another application of  \cite[(6.2.1(b))]{Be17} 
ensures that $\coker\alpha_{\cU,-1}(1)$ is a vector bundle.

The proof of Theorem \ref{Fun3}(1) provides $\KK$-linear maps
\[ V(\gamma_i) : V_1 \lra V_2 \ \ \text{and} \ \ W(\gamma_i) : W_1 \lra W_2 \ \ \ \ \ \ \  (i \in \{1,\ldots, r\})\]
such that
\[ \TilTheta_V = \alpha_{\cU,-1}(1) \ \ \text{and} \ \ \TilTheta_W = \alpha_{\cF,-1}(1),\]
where $X\!:=\!(X_1,X_2,(X(\gamma_i))_{1\le i\le r})$ for $X \in \{V,W\}$. It follows from Theorem \ref{Fun3}(1) that $V$ and $W$ are $\EKP(K_3)$-representations such that $\TilTheta(W) \cong \TilTheta(M)$ and $\TilTheta(V) 
\cong \coker \alpha_{\cU,-1}(1)$. Theorem \ref{Fun3}(2) implies $M\cong W$ and the diagram above yields
\[ \tilde{\tau}_1\circ \TilTheta_V = \TilTheta_W\circ \tilde{\tau}_2,\]
so that evaluation at $v\otimes \frac{1}{\gamma_j^\ast} \in \widetilde{V_1}(D(\gamma_j^\ast))$ gives\footnote{Observe that the identification $\cU_{(r,1)}\cong \cO_{\PP^{r-1}}(-1)$ yields $\TilTheta_X(U)(x\otimes f)\!=\!
\sum_{i=1}^r X(\gamma_i)(x)\otimes \gamma_i^\ast f$ for $x \in X$ and $f \in \cO_{\PP^{r-1}}(-1)(U)$.}
\[ \sum_{i=1}^r (\tau_1\circ V(\gamma_i))(v)\otimes \frac{\gamma_i^\ast}{\gamma_j^\ast} = \sum_{i=1}^r (W(\gamma_i)\circ \tau_2)(v)\otimes \frac{\gamma_i^\ast}{\gamma_j^\ast}.\]
Evaluation at $[\gamma_j] \in \PP(A_r)\!=\!\PP^{r-1}$ shows that $(\tau_2,\tau_1): V \lra W$ is a morphism, which, by virtue of (a) and (b), is injective. Thus, $V$ is isomorphic to a non-zero subrepresentation $N$ of $M$ such that
$\mu(N)\!=\!\mu(\coker \alpha_{\cU,-1}(1))$. \hfill $\diamond$

\medskip
(1) Suppose that $\TilTheta(M)$ is not semistable. By Lemma \ref{Stab-KP1}, the vector bundle $\cF:=\!\widetilde{\Theta}(M)(-1)$ is not semistable. For a representation $(0)\!\neq\!N \in \EKP(K_3)$ with $N_1\!=\!(0)$, the 
bundle $\TilTheta(N)(-1) \cong N_2\!\otimes_{\KK}\!\cO_{\PP^2}(-1)$ is semistable (see \cite[(1.2.4)]{OSS}). In view of $\Delta_M(2)\!\ge\!0$, we thus have $0\!<\!2\dim_\KK M_1\!\leq\!\dim_\KK M_2$ as well as $\mu(\cF)\!=\!\mu(\TilTheta(M))\!-\!1\!
=\!\frac{\dim_\KK M_1}{\Delta_M}\!-\!1\!\le\!0$. Lemma \ref{HN1} provides $\cU \subsetneq \cF$ non-zero such that $\cF/\cU$ is semistable with 
\begin{enumerate}
\item[(i)] $\mu(\cU)\!>\!\mu(\cF)$, and
\item[(ii)] $\Hom_{\PP^2}(\cO_{\PP^2}(\ell),\cF/\cU)\!=\!(0)$ for all $\ell\!\geq\!0$.
\end{enumerate}
In particular,  $(\star)$ is satisfied. Thanks to (e) and (f) there exists a subrepresentation $(0) \subsetneq N \subsetneq M$ such that
\[ \mu(N) = \mu(\coker \alpha_{\cU,-1}(1)) > \mu(\TilTheta(M)) = \mu(M),\]
which contradicts $M$ being semistable. 

(2) As before, we set $\cF\!:=\!\widetilde{\Theta}(M)(-1)$. Since the assumptions of (1) are satisfied, $\cF$ is semistable. Suppose that $\cF$ is not stable. Proceeding as in (1), we first observe that $M$ is indecomposable (even a brick), so that we may assume that $M_1\!\neq\!(0)$. We apply \cite[Chap.2,(1.2.2)]{OSS}\footnote{This result holds for algebraically closed fields of arbitrary characteristic.} in conjunction with Lemma \ref{Stab1} to obtain a short 
exact sequence
\[ (0) \lra \cU \lra \cF \lra \cF/\cU \lra (0)\]
such that $0\! <\! \rk(\cU), \rk(\cF/\cU)\!<\!\rk(\cF)$ with $\cF/\cU$ being torsion-free and $\mu(\cU)\!=\!\mu(\cF/\cU)\!=\!\mu(\cF)$. We choose $\cU$ such that $\cU$ is of maximal rank subject to these properties. 

Let $\cQ$ be a torsion-free quotient of $\cF/\cU$ with $0\!<\!\rk(\cQ)\!<\!\rk(\cF/\cU)$. There results a short exact sequence $(0) \lra \cV \lra \cF \lra \cQ \lra (0)$ with $0\!<\!\rk(\cV),\rk(\cQ)\!<\!\rk(\cF)$, and, observing 
$\rk(\cQ)\!<\!\rk(\cF/\cU)$, we conclude that $\rk(\cV)\!>\!\rk(\cU)$. Since $\cF$ is semistable, the choice of $\cU$ in conjunction with Lemma \ref{Stab1} implies $\mu(\cV)\!<\!\mu(\cF)$ and Lemma \ref{Stab1} yields $\mu(\cQ)\!>\!\mu(\cF)\!=\!\mu(\cF/\cU)$. Consequently, \cite[Chap.2,(1.2.2)]{OSS} shows that $\cF/\cU$ is stable. 

Assume that $f \colon \cO_{\PP^2}(\ell) \lra \cF/\cU$ is a non-zero morphism. Since $\cO_{\PP^2}(\ell)$ is stable and $\im f \subseteq \cF/\cU$ is a non-zero torsion-free sheaf, we conclude
\[ \ell = \mu(\cO_{\PP^2}(\ell)) \leq \mu(\im f) \leq \mu(\cF/\cU) = \mu(\cF) = \frac{\dim_\KK M_1}{\Delta_M}\!-\!1 < 0. \]
We have proven:
\begin{enumerate}
\item[(i)] $\mu(\cU)\!=\!\mu(\cF)\!=\!\mu(\cF/\cU)$, and
\item[(ii)] $\Hom_{\PP^2}(\cO_{\PP^2}(\ell),\cF/\cU)\!=\!(0)$ for all $\ell\!\geq\!0$.
\end{enumerate}
Thanks to (e) and (f) the module $M$ is not stable, a contradiction. It follows that $\TilTheta(M)\!=\!\cF(1)$ is stable. \end{proof}

\bigskip
\noindent 
We summarize our findings in the following

\bigskip

\begin{Thm}\label{Stab-KP3} Let $M \in \EKP(K_3)\!\smallsetminus\!\{0\}$.
\begin{enumerate}
\item If $\Delta_M(2)\!\ge\!0$, then $\TilTheta(M)$ is semistable if and only if $M$ is semistable.
\item If $\Delta_M(2)\!>\!0$, then $\TilTheta(M)$ is stable if and only if $M$ is stable.
\end{enumerate} \end{Thm}

\begin{proof} This is just a combination of Proposition \ref{StabE1} and Proposition \ref{Stab-KP2}. \end{proof}

\bigskip

\begin{Examples} (a) Exceptional Steiner bundles. According to \cite[(4.1)]{GR87}, exceptional vector bundles on $\PP^2(\CC)$ are stable. Thanks to \cite[(1.1)]{Bra08}, this also holds for exceptional Steiner bundles on $\PP^{r-1}(\CC)$. Returning to our context, we let $\cF$ be an exceptional Steiner bundle on $\PP^2$. According to Proposition \ref{ExSt1}, there is a preprojective indecomposable representation $M \in \EKP(K_3)$ such that 
$\cF\cong \TilTheta(M)$. A consecutive application of Proposition \ref{CatRep1} and Theorem \ref{MinType2} ensures that $\Delta_M(2)\!>\!0$. Now Theorem \ref{Stab-KP3} in conjunction with the remark 
following Corollary \ref{Dist2} shows that $\cF$ is stable.

(b) "Regular" indecomposable Steiner bundles. Let $(V_1,V_2)$ be a pair of $\KK$-vector spaces such that $\Delta_{(V_1,V_2)}(2)\!>\!0$ and $q_3(\udim(V_1,V_2))\!<\!0$. The stable representations of $\rep(K_3;V_1,V_2)\cap\EKP(K_3)$ may 
be interpreted as the stable points of the variety $\PP(\rep(K_3;V_1,V_2))$ relative to the canonical action of 
$\SL(V_2)\!\times\!\SL(V_1)$. This can be used to show that the subset $\cS(V_1,V_2)$ of stable representations is non-empty open (and hence dense) whenever $\Delta_{(V_1,V_2)}\!\ge\!2$; see also Proposition \ref{Hyper2} below for the case, where $\KK\!=\!\CC$.\end{Examples}  

\bigskip

\subsection{Applications: Filtrations for Steiner bundles}\label{S:Appl}
Suppose that $M \in \EKP(K_3)$ is semistable. If $P_0(r)$ is a direct summand of $M$, then a two-fold application of the See-Saw Lemma implies $M_1\!=\!(0)$, whence $M\cong nP_0(r)$.  Alternatively $\psi_M$ is surjective, whence $\dim_\KK M_2\!\le\! 3\dim_\KK M_1$. Hence the conditions of Theorem \ref{Stab-KP3} may seem to limit its utility. The results of this section illustrate that a combination of Theorem \ref{Stab-KP3} with techniques involving AR-theory provides many examples of stable Steiner bundles on $\PP^2$. 

\bigskip

\begin{Thm} \label{Appl1} Let $E \in \EKP(K_3)$ be an elementary representation. Then $\TilTheta(E)$ is a stable. \end{Thm}

\begin{proof} Thanks to Proposition \ref{Dist3}, the representation $E$ is stable and $E \in \EKP(K_3)\!\smallsetminus\!\{0\}$ implies $\dim_\KK E_1\!<\!\dim_\KK E_2$. In view of Theorem \ref{Stab-KP3}, it therefore 
suffices to consider the case, where $\dim_\KK E_1\!< \!\dim_\KK E_2\!\leq\!2\dim_\KK E_1$. Thanks to \cite[Thm.]{Ri16}, $\udim E$ belongs to one of the $\sigma_3$-orbits generated by the dimension vectors 
$(1,1),(2,2)$. Since $\Delta_{\sigma_{K_3}(M)}\!=\!\Delta_M(2)\!<\!\Delta_M$ for every regular indecomposable $M \in \rep(K_3)$, there is a unique dimension vector $(x,y) \in \langle\sigma_3\rangle.\udim E$ such that
$\Delta_{(x,y)}\!=\!y\!-\!x \in \NN$ and $\Delta_{(x,y)}(2)\!=\!y\!-\!2x\!\le\!0$. Hence the only dimension vectors in these orbits that satisfy $x\!<\!y\!\leq\!2x$ are $(1,2)$ and $(2,4)$. By virtue of Theorem \ref{MinType2}, 
representations with dimension vector $(1,2)$ do not belong to $\EKP(K_3)$, whence $\udim E\!=\!(2,4)$. As elementary representations are bricks, $\TilTheta(E)$ is a simple vector bundle of rank $2$. We conclude with 
\cite[Chap.2.(1.2.10)]{OSS}, which holds in arbitrary characteristic (cf.\ also \cite[Prop.1]{Sc61a}), that $\TilTheta(E)$ is stable. \end{proof}

\bigskip
\noindent
Let $\cC \subseteq \Gamma(K_r)$ be a regular Auslander-Reiten component. We say that $\cC$ is \textit{(semi)stable}, provided $\cC\cap\EKP(K_r)$ contains a (semi)stable vertex\footnote{Recall from \ref{Dist2}(3) that semistable vertices are quasi-simple.}. The following result 
shows that for such components the Harder-Narasimhan filtrations of most of the bundles $\TilTheta(M)$ correspond to the quasi-composition series of $M$. 

\bigskip

\begin{Prop} \label{Appl2} 
Let $\cC \subseteq \Gamma(K_3)$ be a semistable AR-component. If $M \in (\tau^{-1}_{K_3}(M_{\cC,1})\!\rightarrow)$ has quasi-length $\ell$, then the Harder-Narasimhan filtration of $\TilTheta(M)$ is given by  
\[  (0) \subsetneq \TilTheta(M_{[1]}) \subsetneq \TilTheta(M_{[2]})\subsetneq \cdots \subsetneq \TilTheta(M_{[\ell\!-\!1]}) \subsetneq \TilTheta(M_{[\ell]}) = \TilTheta(M) \]
with filtration factors $\TilTheta(M_{[j]})/\TilTheta(M_{[j\!-\!1]}) \cong \TilTheta(\tau^{-j+1}_{K_3}(M_{[1]}]))$ for all $j \in \{2,\ldots,\ell\}$. \end{Prop}

\begin{proof} We consider the quasi-composition series
\[ (0) \subsetneq M_{[1]} \subsetneq \cdots \subsetneq M_{[\ell\!-\!1]} \subsetneq M\]
of $M$. For $1\!< \!j\!\leq\!\ell$ general AR-theory provides an exact sequence
\[ (0) \lra M_{[j-1]} \lra M_{[j]} \lra \tau^{-j+1}_{K_3}(M_{[1]}) \lra (0)\]
in $\EKP(K_3)$. According to Corollary \ref{Fun4}, application of $\TilTheta$ yields an exact sequence
\[ (0) \lra \TilTheta(M_{[j-1]}) \lra \TilTheta(M_{[j]}) \lra \TilTheta(\tau^{-j+1}_{K_3}(M_{[1]})) \lra (0)\]
so that $\TilTheta(M_{[j]})/\TilTheta(M_{[j-1]}) \cong \TilTheta(\tau^{-j+1}_{K_3}(M_{[1]}))$. Our assumption on
$M$ implies that the quasi socle $M_{[1]}$ belongs to the cone $(\tau^{-1}_{K_3}(M_{\cC,1})\!\rightarrow)$. It 
follows from Theorem \ref{CatRep3}(3) that $\tau^{-j+1}_{K_3}(M_{[1]}) \in \rep(K_3,2)$, whence $\Delta_{\tau^{-
j+1}_{K_3}(M{[1]})}(2)\!\ge\!0$. 

As $\cC$ is semistable, Corollary \ref{Dist2} ensures that the same is true for $\tau^{-n}_{K_3}(M_{[1]})$ for all $n \in \NN_0$. We apply Theorem \ref{Stab-KP3} and conclude that 
$\TilTheta(M_{[j]})/\TilTheta(M_{[j-1]})$ is semistable. Proposition \ref{Dist1}(3) yields
\[\mu(\TilTheta(\tau^{-j+1}_{K_3}(M_{[1]}))) = \mu(\tau^{-j+1}_{K_3}(M_{[1]})) > \mu(\tau^{-j}_{K_3}(M_{[1]})) = \mu(\TilTheta(\tau^{-j}_{K_3}(M_{[1]})))\]
for all $1\!\le\!j\!\leq\!\ell$. The statement now follows from the uniqueness of the Harder-Narasimhan filtration. \end{proof}

\bigskip

\begin{Prop} \label{Appl3} Let  $\cC \subseteq \Gamma(K_3)$ be a regular component. If $M \in \cC\cap\repp(K_3,d)$ has minimal type, then $\TilTheta(N)$ is stable for every quasi-simple $N \in (M_{\cC,d} \!\rightarrow)$. 
\end{Prop}

\begin{proof} Corollary \ref{MinType3} and Proposition \ref{Strat4} show that $M\!=\!M_{\cC,d}$ is a brick. If $d\!=\!2$, then Lemma \ref{StabE2} and Theorem \ref{Stab-KP3} imply that $\TilTheta(M)$ is stable. Alternatively, 
$\TilTheta(M) \in \Vect(\PP^2)$ is simple and of rank $2$, so that \cite[Chap.2,(1.2.10)]{OSS} implies the stability of $\TilTheta(M)$. As $N \in \{M_{\cC,1}\}\cup (M_{\cC,2}\!\rightarrow)$, it follows from Corollary
\ref{Dist2} that $\TilTheta(N)$ 
is stable. \end{proof}

\bigskip

\begin{Thm} \label{Appl4} Let $\cC \subseteq \Gamma(K_3)$ be a regular AR-component, $\cF$ be an indecomposable Steiner bundle.
\begin{enumerate}
\item If $\cC$ contains an elementary representation and $\cF \in \TilTheta(M_{\cC,1}\!\rightarrow)$, then $\cF$ possesses a Harder-Narasimhan filtration by Steiner bundles, whose filtration factors are stable.
\item If $\cF$ has a resolution 
\[ (0) \to \cO_{\PP^2}(-1)^a \to \cO^b_{\PP^2} \to \cF \to (0)\]
such that $(a,b)$ is an imaginary root and $q_3(a,b)\!+\!\Delta_{(a,b)}(2)\! \geq\! 1$, then $\cF$ possesses a filtration 
\[ (0) = \cF_0 \subseteq \cF_1 \subseteq \cdots \subseteq \cF_n = \cF\]
by Steiner bundles such that each quotient $\cF_i/\cF_{i-1}$ is a stable Steiner bundle. 
\item Suppose $X \in \cC$ is quasi-simple such that $ q_3(\udim X)\!+\!\Delta_X(2)\!\ge\!1$. If $\cF \in \TilTheta(X\!\rightarrow)$, then $\cF$ possesses a 
filtration 
\[ (0) = \cF_0 \subseteq \cF_1 \subseteq \cdots \subseteq \cF_n = \cF\]
by Steiner bundles such that each quotient $\cF_i/\cF_{i-1}$ is a stable Steiner bundle. 
\end{enumerate} \end{Thm}

\begin{proof} (1) Let $E \in \cC$ be elementary. By definition, $E$ is quasi-simple, so that, for every $n \in \NN_0$, the module $\tau_{K_3}^{-n}(M_{\cC,1})$ is elementary and contained $\EKP(K_3)$. Given $\cF\!:=\!\TilTheta(M) \in \TilTheta(M_{\cC,1}\!\rightarrow)$, we consider the quasi-composition series  
\[ (0) = M_{[0]}\subsetneq M_{[1]} \subsetneq \cdots \subsetneq M_{[\ql(M)]} = M,\] 
whose filtration factors belong to $\{\tau_{K_3}^{-\ell}(M_{\cC,1}) \ ; \ \ell \in \NN_0\}$. Application of $\TilTheta$ provides a filtration
\[ (0) = \TilTheta(M_{[0]})\subsetneq \TilTheta(M_{[1]}) \subsetneq \cdots \subsetneq \TilTheta(M_{[\ql(M)]}) = \cF,\] 
with filtration factors $\TilTheta(M_{[j]})/\TilTheta(M_{[j-1]}) \cong \TilTheta(\tau^{-j+1}_{K_3}(M_{[1]}))$. Thanks to Theorem \ref{Appl1}, the factors are stable and with strictly decreasing slopes (cf.\ Proposition \ref{Dist1}), so 
that we have found the Harder-Narasimhan filtration of $\cF$.

(2) By Theorem \ref{Fun3} we can find an indecomposable representation $M \in \EKP(K_3)$ such that $\TilTheta(M) \cong \cF$. We have $b\!-\! a\! = \!\rk( \cF)\! =\! \rk(\TilTheta(M))\!=\! \Delta_M$ and $a\!=\!c_1(\cF)\!=\! 
c_1(\TilTheta(M))\! =\!\dim_\KK M_1$. Hence 
\[  q_3(\udim M)\!+\!\Delta_M(2) = q_3(a,b)\!+\!\Delta_{(a,b)}(2) \geq 1.\]
We conclude with Lemma \ref{Filtrat2}  that there is a filtration 
\[ (0) = M_0 \subsetneq M_1 \subsetneq \cdots \subsetneq M_n = M\]
such that each quotient $M_i/M_{i-1}$ is an elementary representation in $\EKP(K_3)$. Exactness of $\TilTheta$ on $\EKP(K_3)$ yields a filtration 
\[ (0) = \TilTheta(M_0) \subsetneq \TilTheta(M_1) \subsetneq \cdots \subsetneq \TilTheta(M_n) = \TilTheta(M) \cong \cF\]
of Steiner bundles. By Theorem \ref{Appl1}, each quotient $\TilTheta(M_i)/\TilTheta(M_{i-1}) \cong \TilTheta(M_i/M_{i-1})$ is a stable Steiner bundle. 

(3) This follows as in (2), using Corollary \ref{Filtrat3}. \end{proof}

\bigskip

\subsection{Applications: Restrictions to hyperplanes}\label{S:Hyper}
We assume throughout that $\KK\!= \!\CC$. Let $\cF$ be a stable (respectively, semistable) vector bundle of rank $n$ on $\PP^{r-1}$ with $r\!\geq\!4$. In \cite[Problem 5]{Ha79} it was asked whether for almost all linear 
hyperplanes $H \subseteq \PP^{r-1}$ the restriction of $\cF$ to $H$ is stable (respectively, semistable).

The question was answered affirmatively by Barth \cite{Ba77} for $n\!=\!2$ and for $n\!<\! r\!-\!1$ by Maruyama \cite{Ma80} (in both cases with one exception). In the context of Steiner bundles, these rank conditions 
imply that only (direct sums of) the structure sheaf of $\PP^{r-1}$ can occur. 

Consider the tangent bundle $\cT(-1)$ of rank $r\!-\!1$. Then $\cT(-1) \cong \TilTheta(P_1(r))$ is stable. On the other hand, $P_1(r\!-\!1) \subseteq \alpha^\ast(P_1(r))$ for all $\alpha \in \Inj_\KK(A_{r-1},A_r)$, while $
\mu(\alpha^\ast(P_1(r)))\!=\!\frac{1}{r-1}\! <\! \frac{1}{r-2} \! = \! \mu(P_1(r\!-\!1))$. Thanks to Lemma \ref{HomM1} and Proposition \ref{StabE1}, $\cT(-1)|_H$ is not semistable for any hyperplane 
$H \subseteq \PP^{r-1}$. This example, which does not depend on the ground field, is well-known (cf.\ \cite{Fle84}). Building on recent work of Coskun, Huizenga and Smith \cite{CHS22}, we will show below that generic Steiner bundles on $\PP^{r-1}$ behave similarly.

\bigskip

\begin{Prop} \label{Hyper2} Let $r\!\geq\!3$, $(V_1,V_2)$ be a pair of vector spaces such that 
\begin{enumerate}
\item[(a)] $\Delta_{(V_1,V_2)}\!\geq\!r\!-\!1$, and 
\item[(b)] $q_r(\udim (V_1,V_2))\!<\!1$. \end{enumerate}
There exists a non-empty open subset $O \subseteq \rep(K_r;V_1,V_2) \cap \EKP(K_r)$ such that for every $M \in O$ the Steiner bundle $\TilTheta(M)$ is stable. \end{Prop}

\begin{proof} Following \cite{CHS22}, we define $\psi_{r-1} := \frac{r-2+\sqrt{(r-1)^2+2(r-1)-3}}{2}$, so that
\[\psi_{r-1}\!+\!1 = \frac{r+\sqrt{(r\!-\!1)^2\!+\!2(r\!-\!1)\!-\!3}}{2} = \frac{r\!+\!\sqrt{r^2\!-\!4}}{2} = L_r.\]
As $V_1\!\ne\!(0)$, our current assumptions imply
\[\frac{r\!-\!1}{\dim_\KK V_1} \leq \frac{\Delta_{(V_1,V_2)}}{\dim_\KK V_1} \ \ \ \text{and} \ \ \ \frac{\Delta_{(V_1,V_2)}}{\dim_\KK V_1} = \frac{\dim_\KK V_2}{\dim_\KK V_1}\!-\!1 < L_r\!-\!1 = \psi_{r-1},\] 
respectively. Thanks to \cite[(5.1)]{CHS22}, there is a dense open subset 
\[ O \subseteq \Hom_{\PP^{r-1}}(V_1\! \otimes_\KK\! \cO_{\PP^{r-1}}(-1),V_2\!\otimes_\KK\! \cO_{\PP^{r-1}}) \cong \Hom_\KK(V_1,V_2)^r = \rep(K_r;V_1,V_2)\] 
such that $O \subseteq \EKP(K_r)$, and for every $M \in O$ the Steiner bundle $\TilTheta(M)$ is stable. \end{proof}

\bigskip

\begin{Remark} The statement is also true if $q_r(\udim(V_1,V_2))\!=\!1$ and $\dim_\KK V_1\!\leq\!\dim_\KK V_2$. In this case there exists a non-empty open subset $O \subseteq \rep(K_r;V_1,V_2)$ such that every 
representation in $O$ is isomorphic to some $P_i(r)$. By Proposition \ref{ExSt1} and \cite[(8.1)]{Bra08}, the bundles $\TilTheta(P_i(r))$ are stable.\footnote{In \cite{Bra08} Gieseker stability was considered. It can be 
shown, however, that the relevant result of loc.\ cit.\ also holds in our context.} \end{Remark}

\bigskip

\begin{Cor}\label{Hyper3} Assume that  $q_r(\udim(V_1,V_2))\!<\!1$ and $\Delta_{(V_1,V_2)}(r\!-\!1)\!\geq\!r\!-\!1$. A general Steiner bundle $\cF$ of rank $\Delta_{(V_1,V_2)}$ and with first Chern class $\dim_\KK 
V_1$ is stable and $\cF|_H$ is not semistable for each hyperplane $H \subseteq \PP^{r-1}$. \end{Cor}

\begin{proof} Let $O \subseteq \rep(K_r;V_1,V_2)$ be the non-empty open set defined in Proposition \ref{Hyper2}. By Proposition \ref{CatRep4} the set $O\cap\repp(K_r,r\!-\!1)$ is non-empty and open. Let $M \in O \cap  
\repp(K_r,r\!-\!1)$, then $\cF\!:=\!\TilTheta(M)$ is stable. Since $M \in \repp(K_r,r\!-\!1)$, we have $\alpha^\ast(M) \cong \Delta_M(r\!-\!1)P_0(r\!-\!1)\!\oplus\!(\dim_\KK M_1)P_1(r\!-\!1)$ for all $\alpha \in 
\Inj_\KK(A_{r-1},A_r)$. Thus, $\mu(\alpha^\ast(M))\!<\!\mu(P_1(r\!-\!1))$, and the above arguments show that for every hyperplane $H \subseteq \PP^{r-1}$, the bundle $\cF|_H$  is not semistable. \end{proof}

\bigskip
\noindent
Corollary \ref{Hyper3} provides negative answers to \cite[Problem 5]{Ha79} for a fixed rank and Chern class on an open subset of $ \rep(K_r;V_1,V_2)$. 

Our next result shows that we can find an infinite family of Steiner bundles of fixed rank and Chern class that are counter examples and even belong to different $\GL(A_r)$-orbits. Moreover, these bundles are 
isomorphic when we restrict them to any hyperplane.

\bigskip

\begin{Cor}\label{Hyper4} Let $(V_1,V_2)$ be a pair of vector spaces such that  
\begin{enumerate}
\item[(a)] $q_r(\udim(V_1,V_1))\!<\!-r^2$, and 
\item[(b)] $\Delta_{(V_1,V_2)}(r\!-\!1)\!\geq\!r\!-\!1$. \end{enumerate} 
There exists an infinite family $(\cF_i)_{i \in I}$ of Steiner bundles such that for all $i,j \in I$ the following statements hold:
\begin{enumerate}
\item $\cF_i$ is stable.
\item $c(\cF_i)\!=\!\dim_\KK V_1$, $\rk(\cF_i)\!=\!\Delta_{(V_1,V_2)}$.
\item $\cF_i|_H \cong \cF_j|_H$ for each hyperplane $H \subseteq \PP^{r-1}$.
\item $\cF_i|_H$ is not semistable for each hyperplane $H \subseteq \PP^{r-1}$.
\item All $\cF_i$ lie in different orbits under the natural action of $\GL(A_r)$ on $\Coh(\PP^{r-1})$. \end{enumerate} \end{Cor}

\begin{proof} Let $O \subseteq \rep(K_r;V_1,V_2)$ be the non-empty open set defined in Proposition \ref{Hyper2}. The algebraic group $G\!:= \GL(A_r)\!\times\!\GL(V_2)\!\times\!\GL(V_1)$ acts canonically on
$\rep(K_r;V_1,V_2) \cong \Hom_\KK(A_r\!\otimes_\KK\!V_1,V_2)$. We consider the $G$-stable and open set $U\!:=\!\bigcup_{g \in G} g.O \subseteq \rep(K_r;V_1,V_2)$ and set $\cU\!:=\!U \cap \rep(K_r;V_1,V_2) 
\cap \repp(K_r,r\!-\!1)$. By Proposition \ref{CatRep4} the $G$-stable set $\cU$ is non-empty and open. 

Let $M \in \cU$. We find $(a,h_2,h_1) \in G$ and $S \in O \cap \rep(K_r;V_1,V_2) \cap \repp(K_r,r\!-\!1)$ such that $M\!=\!(a,h_2,h_1).S$. By definition, the representations $M$ and $(a,\id_{V_2},\id_{V_1}).S \!=\!S^{(a)}$ are isomorphic. We obtain 
\[ \cF := \TilTheta(M) \cong \TilTheta(S^{(a)}) \cong a^\ast.\TilTheta(S),\]
see Section \ref{S:HomM}. Since $S \in O$, we know that $\TilTheta(S)$ is stable. Therefore, the same is true for $a^\ast.\TilTheta(S) \cong \cF$. This shows that every bundle arising in this way has property (1). Owing to Theorem \ref{Fun3} and 
Corollary \ref{Dc2}, each bundle of this form satisfies (2). The arguments in the proof of Corollary \ref{Hyper3} now show that (3) and (4) hold. Since $q_r(\udim(V_1,V_2))\!<\!-r^2$ we can apply \cite[(4.2)]{Bi22} and 
conclude that $\cU/G$ is not finite. By construction, elements of $\rep(K_r;V_1,V_2)$ belonging to different $G$-orbits correspond to Steiner bundles, whose isomorphisms classes belong to different orbits under the 
action of $\GL(A_r)$. This shows (5). \end{proof}

\bigskip

\begin{Examples} We fix $r\!=\!4$.

(1) $q_r(5,18)\!=\!-11$ and $\Delta_{(5,18)}(3)\!=\!3$. We apply Corollary \ref{Hyper3} and find a stable Steiner bundle of rank $13$ on $\PP^3$ such that $\cF|_H$ is not semistable for each hyperplane 
$H \subseteq \PP^3$.

(2) $q_r(6,21)\!=\!-27\!<\!-4^2$ and $\Delta_{(6,21)}(3)\!=\!3$. We apply Corollary \ref{Hyper4} and  find an infinite family of stable Steiner bundles $(\cF_i)_{i \in I}$ of rank $15$ on $\PP^3$ such that $\cF_i|_H$ is not 
semistable for all $i \in I$ and all hyperplanes $H \subseteq \PP^3$. Moreover, we have $\cF_i|_H \cong \cF_j|_H$ for each hyperplane and all $i,j \in I$, while $\cF_i$ are $\cF_j$ are not in the same $\GL(A_r)$-orbit. 
\end{Examples}


\bigskip

\bigskip

\begin{thebibliography}{00}
\bibitem{AM15} E. Arrondo and S. Marchesi, \textit{Jumping pairs of Steiner bundles.} Forum Math. \textbf{27} (2015), 3233--3267. 
\bibitem{ASS06} I. Assem, D. Simson and A. Skowro\'nski, \textit{Elements of the Representation Theory of Associative Algebras, I}. London Mathematical Society Student Texts \textbf{65}. 
Cambridge University Press, 2006.
\bibitem{At56} M. Atiyah, \textit{On the Krull-Schmidt Theorem with applications to sheaves}. Bull. Soc. Math. France \textbf{84} (1956), 307--317.
\bibitem{ARS95} M. Auslander, I. Reiten and S. Smal\o, \textit{Representation Theory of Artin Algebras.}  Cambridge Studies in Advanced Mathematics \textbf{36}. Cambridge University Press, 1995.
\bibitem{BaEl} E. Ballico and Ph. Ellia, \textit{Fibr\'es homog\`enes sur $\PP^n$.} C.R. Acad. Sci. Paris \textbf{294} (1982), 403--406.
\bibitem{Ba77} W. Barth, \textit{Some properties of stable rank-2 vector bundles on $\PP_n$}. Math. Ann. \textbf{226} (1977), 125--150.
\bibitem{Be17} D. Benson, \textit{Representations of Elementary Abelian $p$-Groups and Vector Bundles}. Cambridge Tracts in Mathematics \textbf{208}. Cambridge University Press, 2017.
\bibitem{Bi20} D. Bissinger, \textit{Representations of constant socle rank for the Kronecker algebra}. Forum Math. \textbf{32} (2020), 23--43.
\bibitem{Bi20b} \bysame, \textit{Dimension vectors with the equal kernels property.} J. Pure Appl. Algebra \textbf{227} (2023), 107424.
\bibitem{Bi22} \bysame, \textit{Shift orbits for elementary modules of Kronecker quivers.} arXiv: 2403.01824 [math.RT]
\bibitem{Bis04} I. Biswas,  \textit{On the stability of homogeneous vector bundles}. J. Math. Sci. Univ. Tokyo \textbf{11} (2004), 133--140.
\bibitem{BoSp92} G. Bohnhorst and H. Spindler, \textit{The stability of certain vector bundles on $\PP^n$}. Lecture Notes in Math. \textbf{1507} (1992), 39--50.
\bibitem{Bra05} M. Brambilla, \textit{Simplicity of generic Steiner bundles}. Boll. Unione Mat. Ital. \textbf{(8) 8-B} (2005), 723--735.
\bibitem{Bra08} \bysame, \textit{Cokernel bundles and Fibonacci bundles}. Math. Nachr. \textbf{281} (2008), 499--516.
\bibitem{BLV15} R.-O. Buchweitz, G. Leuschke, and M. Van den Bergh. \textit{On the derived category of Grassmannians in arbitrary characteristic.} Compos. Math. \textbf{151} (2015), 1242--1264.  
\bibitem{CFP12} J. Carlson, E. Friedlander, and J. Pevtsova, \textit{Representations of elementary abelian p-groups and bundles of Grassmannians}.  Adv. Math. \textbf{229} (2012), 2985--3051. 
\bibitem{CFP15} \bysame, \bysame, \bysame, \textit{Elementary subalgebras of Lie algebras}. J. Algebra \textbf{442} (2015), 155--189. 
\bibitem{CFS11} J. Carlson, E. Friedlander and A. Suslin, \textit{Modules for $\ZZ/p\!\times\!\ZZ/p$.} Commentarii Math. Helvetici \textbf{86} (2011), 609--657.
\bibitem{CE} H. Cartan and S. Eilenberg, \textit{Homological Algebra}. Princeton Mathematical Series \textbf{19}. Princeton University Press, 1973.
\bibitem{Ch13} B. Chen, \textit{Dimension vectors in regular components over wild Kronecker quivers.} Bull. Sci. Math. \textbf{137} (2013), 730--745.
\bibitem{CHS22} I. Coskun, J. Huizenga and G. Smith, \textit{Stability and cohomology of kernel bundles on projective space.} arXiv:2204.10247 [math.AG] 
\bibitem{Da} E. Dade, \textit{Endo-permutation modules over $p$-groups II}. Ann. of Math. \textbf{108} (1978), 317--346.
\bibitem{DK93} I. Dolgachev and M. Kapranov, \textit{Arrangements of hyperplanes and vector bundles on $\PP_n$}. Duke Math. J. \textbf{71} (1993), 633--664.
\bibitem{Dr87}  J.-M. Drezet, \textit{Fibr\'es exceptionels et vari\'et\'es de modules de faisceaux semi-stables sur $\PP_2(\CC)$}. J. reine angew. Math. \textbf{49} (1987), 14--58.
\bibitem{Dr99} \bysame, \textit{Vari\'et\'es de modules alternatives}. Ann. Inst. Fourier \textbf{49} (1999), 57--139.
\bibitem{EH} D. Eisenbud and J. Harris, \textit{3264 and All That}. Cambridge University Press, 2016. 
\bibitem{El79} G. Elencwajg, \textit{Des fibr\'{e}s uniformes non homog\`{e}nes}. Math. Ann. \textbf{239} (1979), 185--192.
\bibitem{Fa04} R. Farnsteiner, \textit{Varieties of tori and Cartan subalgebras of restricted Lie algebras}. Trans. Amer. Math. Soc. \textbf{ 356} (2004), 4181--4236.
\bibitem{Fa11} \bysame, \textit{Categories of modules given by varieties of $p$-nilpotent operators}. arXiv: 1110.2706 [math.RT]
\bibitem{Fle84} H. Flenner, \textit{Restrictions of semistable bundles on projective varieties.} Commentarii Math. Helvetici \textbf{59} (1984), 635--650.
\bibitem{FPe05} E. Friedlander and J. Pevtsova, \textit{Representation-theoretic support spaces for finite group schemes.} Amer. J. Math. \textbf{127} (2005), 379--420. Erratum: 
Amer. J. Math. \textbf{128} (2006), 1067--1068. 
\bibitem{FPe11} \bysame, \bysame, \textit{Constructions for infinitesimal group schemes.} Trans. Amer. Math. Soc. \textbf{363} (2011), 6007--6061.
\bibitem{GL87} W. Geigle and H. Lenzing, \textit{A class of weighted projective curves arising in representation theory of finite dimensional algebras}. Lect. Notes Math. \textbf{1273} (1987), 265--297. 
\bibitem{GW} U. G\"ortz and T. Wedhorn, \textit{Algebraic Geometry I}. Advanced Lectures in Mathematics. Vieweg+Teubner, 2010.
\bibitem{GR87} A. Gorodentsev and A. Rudakov, \textit{Exceptional vector bundles on projective spaces}. Duke Math. J. \textbf{54} (1987), 115--130.
\bibitem{Ha77} R. Hartshorne, \textit{Algebraic Geometry}. Graduate Texts in Mathematics \textbf{52}. Springer-Verlag, 1977.
\bibitem{Ha79} \bysame, \textit{Algebraic vector bundles on projective spaces: a problem list.} Topology \textbf{18} (1979), 117--128.
\bibitem{Ho56} G. Hochschild, \textit{Relative homological algebra}. Trans. Amer. Math. Soc. \textbf{82} (1956), 246--269.
\bibitem{Hu81} K. Hulek, \textit{On the classification of stable rank-r vector bundles over the projective plane}. Progress in Math. \textbf{7} (1981), 113-144. 
\bibitem{JP15} M. Jardim and D. Prata, \textit{Vector bundles on projective varieties and representations of quivers.} Algebra Discr. Math. \textbf{20} (2015), 217--249.
\bibitem{Huy06} D. Huybrechts, \textit{Fourier-Mukai Transforms in Algebraic Geometry}. Clarendon Press, 2006.
\bibitem{Ka80} V. Kac, \textit{Infinite root systems, representations of graphs and invariant theory}. Invent. math. \textbf{56} (1980), 57--92. 
\bibitem{Ka82} \bysame, \textit{Infinite root systems, representations of graphs and invariant theory, II}. J. Algebra \textbf{78} (1982), 141--162. 
\bibitem{Kar93} B. Karpov, \textit{Semistable sheaves on a two-dimensional quadric, and Kronecker modules}. Russian Acad. Izv. Math. \textbf{40} (1993), 33--66.
\bibitem{Ke96} O. Kerner, \textit{Representations of wild quivers}. CMS Conf. Proc. \textbf{19} (1996), 65--107.
\bibitem{KL96} O. Kerner and F. Lukas, \textit{Elementary modules.} Math. Z. \textbf{223} (1996), 421--434. 
\bibitem{Ku} E. Kunz, \textit{Einf\"uhrung in die kommutative Algebra und algebraische Geometrie}. Vieweg Studium \textbf{46}. Vieweg, 1979.
\bibitem{La07} A. Langer, \textit{Lectures on torsion-free sheaves and their moduli}. In: Algebraic Cycles, Sheaves, Shtukas, and Moduli, 69--103. Springer-Verlag, 2007.
\bibitem{Ma80} M. Maruyama, \textit{Boundedness of semistable sheaves of small ranks.} Nagoya Math. J. \textbf{78} (1980), 65--94. 
\bibitem{Ma12} S. Marchesi, \textit{Jumping spaces in Steiner bundles}. Ph.D.-Thesis, Universita degli Studi di Milano, Universidad Complutense de Madrid, 2012.
\bibitem{MMR21} S. Marchesi and R. Mir\'o-Roig, \textit{Uniform Steiner bundles}. Ann. Inst. Fourier \textbf{71} (2021), 447--472.
\bibitem{Mu} D. Mumford, \textit{The Red Book of Varieties and Schemes}. Lecture Notes in Mathematics \textbf{1358}. Springer-Verlag, 1999.
\bibitem{OSS} C. Okonek, M. Schneider, and H. Spindler, \textit{Vector Bundles on Complex Projective Spaces}. Birk\"auser-Verlag, 1980 (3rd Ed. 2011).
\bibitem{Pe} D. Perrin. \textit{Algebraic Geometry. An Introduction.}  Springer University Text. Springer-Verlag, 2008.
\bibitem{Ri76} C. Ringel, \textit{Representations of K-species and bimodules.} J. Algebra \textbf{41} (1976), 269--302.
\bibitem{Ri78} \bysame, \textit{Finite-dimensional hereditary algebras of wild representation type}. Math. Z. \textbf{161} (1978), 235--255.
\bibitem{Ri16} \bysame, \textit{The elementary $3$-Kronecker modules}. arxiv:1612.09141 [math.RT]
\bibitem{Ru90} A. Rudakov, \textit{Exceptional collections, mutations, and helices}. In: Helices and Vector Bundles. London Math. Soc. Lecture Notes \textbf{148}. Cambridge University Press, 1990.
\bibitem{Sc61a} R. Schwarzenberger, \textit{Vector bundles on algebraic surfaces}. Proc. London Math. Soc. \textbf{11} (1961), 601--622.
\bibitem{Sc61b} \bysame, \textit{Vector bundles on the projective plane}. Proc. London Math. Soc. \textbf{11} (1961), 623--640.
\bibitem{Sh} I. Shafarevich, \textit{Basic Algebraic Geometry 1}. Springer-Verlag, 1994.
\bibitem{SS07} D. Simson and A. Skowro\'nski, \textit{Elements of the Representation Theory of Associative Algebras, III}. London Mathematical Society Student Texts \textbf{72}. Cambridge University Press, 2007.
\bibitem{We87} R. Westwick, \textit{Spaces of matrices of fixed rank}. Linear \& Multilinear Algebra \textbf{20} (1987), 171--174. 
\bibitem{Wo13} J. Worch, \textit{Categories of modules for elementary abelian p-groups and generalized Beilinson algebras}. J. London Math. Soc. \textbf{88} (2013), 649--668. 
\end{thebibliography}
\end{document}